\documentclass[a4paper,11pt]{amsart}

\usepackage{amsfonts}
\usepackage{amssymb}
\usepackage[utf8]{inputenc}
\usepackage{amsmath}
\usepackage{pdflscape}
\usepackage{graphicx}
\usepackage{comment}
\usepackage{float}
\usepackage{subcaption}
\usepackage{xcolor}

\usepackage[colorlinks=true, linkcolor=blue, citecolor=red]{hyperref}
\usepackage[]{epsfig}
\usepackage[]{pstricks}
\usepackage{tikz}

\numberwithin{equation}{section}

\newtheorem{theorem}{Theorem}[section]
\newtheorem{proposition}[theorem]{Proposition}
\newtheorem{lemma}[theorem]{Lemma}
\newtheorem{corollary}[theorem]{Corollary}

\newtheorem{remark}{Remark}
\theoremstyle{remark}

\newtheorem{definition}{Definition}[section]

\newcommand{	\R}{\mathbb{R}}

\newcommand{\C}{\mathbb{C}}
\newcommand{\N}{\mathbb{N}}

\newcommand{\supp}{\textrm{supp}}
\newcommand{\im}{\operatorname{Im}}
\newcommand{\re}{\operatorname{Re}}

\newcommand{\bphi}{{\boldsymbol{\varphi}}}

\def\bra#1{\langle#1\rangle}

\numberwithin{equation}{section}
\makeatletter
\@namedef{subjclassname@2010}{\textup{2020} Mathematics Subject Classification}
\makeatother

\begin{document}
\title[Higher Order KdV]{Higher order KdV-type equations with Robin boundary conditions on $\R^+$}
\author[Gallego]{Fernando. A. Gallego}
\address{Departamento de Matem\'atica y Estadística, Universidad Nacional de Colombia (UNAL), Cra 27 No. 64-60, 170003, Manizales, Colombia}
\email{fagallegor@unal.edu.co}
\author[Kwak]{Chulkwang. Kwak}
\address{Department of Mathematics, Ewha Womans University, Seoul 03760, Republic of Korea}
\address{Korea Institute for Advanced Study, Seoul 02455, Republic of Korea}
\email{ckkwak@ewha.ac.kr}

\begin{abstract}
We study the initial boundary value problem (IBVP) for the higher order Korteweg--de Vries type equation
\[\partial_tu+(-1)^{j+1}\partial_x^{2j+1}u+\frac12\partial_x(u^2)=0,\qquad j\in\N,\]
on the right half line, subject to the Robin boundary conditions
\[(\partial_x+\gamma)\partial_x^{\ell-1}u(t,0)=\varphi_\ell(t),\qquad 1\le\ell\le j,\]
where $\gamma\in\R$ is common to all boundary conditions. We prove local well posedness for
\[u_0\in H^s(\R^+),\qquad \varphi_\ell\in H^{\frac{s+j-\ell}{2j+1}}(0,T),\qquad -j+\frac14<s<\frac32.\]
In particular, the regularity range extends the low regularity theory for the higher order Dirichlet IBVP to Robin boundary conditions.

The main ingredient is an explicit unified transform representation for the higher order IBVP with Robin boundary conditions. The Robin hierarchy couples adjacent boundary traces and leads, in each spectral sector, to a nontrivial system for the unknown boundary transforms. Rather than computing the inverse of this system, we resolve precisely the linear combination required by the unified transform representation through Lagrange interpolation at the rotated spectral points. This yields an exact factorization in which the dependence on the Robin parameter is concentrated in the single factor $(k-i\gamma)^{-1}$. Consequently, the only possible Robin pole is $k=i\gamma$, which contributes the residue mode
\[e^{-\gamma x+\gamma^{2j+1}t}\]
exactly when $j$ is odd and $\gamma>0$. Combining this representation with linear estimates in modified Fourier restriction spaces and higher order KdV bilinear estimates yields the low regularity well posedness result. The representation also recovers the classical KdV formulas with Robin and Neumann boundary data when $j=1$ and is formally consistent with the higher order Dirichlet representation under the reciprocal Robin limit.
\end{abstract}

\maketitle

\tableofcontents

\section{Introduction}\label{sec:introduction}

\subsection{KdV and higher dispersion models}

The Korteweg--de Vries equation is one of the fundamental models in the theory of nonlinear dispersive waves. Its origin goes back to Russell's observation of a solitary water wave propagating along a canal without an immediate change of form \cite{Russell}. Boussinesq subsequently developed a mathematical theory for long waves in shallow water and identified the balance between weak nonlinearity and weak dispersion that produces solitary waves \cite{Boussinesq}. Korteweg and de Vries later derived the equation now bearing their names as a model for the unidirectional propagation of long, weakly nonlinear surface waves in a shallow channel \cite{KdV}. After a suitable normalization, the equation takes the form
\[\partial_tu+\partial_x^3u+u\partial_xu=0.\]
The nonlinear transport term tends to steepen the wave profile, while the third order dispersive term spreads different spatial frequencies at different speeds. The balance between these two effects is responsible for the existence of solitary waves.

The KdV equation also occupies a central position in the theory of completely integrable systems. Gardner, Greene, Kruskal, and Miura introduced the inverse scattering method for the KdV initial value problem \cite{GGKM}, while Lax formulated its integrability through an isospectral deformation of a Schr\"odinger operator \cite{Lax}. The equation admits a Hamiltonian formulation and an infinite hierarchy of conservation laws \cite{ZF}.

A natural higher dispersion extension of the KdV equation is
\begin{equation}\label{eq:higher-order-kdv}
\partial_tu+(-1)^{j+1}\partial_x^{2j+1}u+\frac12\partial_x(u^2)=0,\qquad j\in\N.
\end{equation}
The case $j=1$ is the classical KdV equation, while $j=2$ corresponds to the Kawahara equation.\footnote{In the literature, the term Kawahara equation is used for closely related fifth order KdV type models, sometimes including an additional third order dispersive term.} The fifth order dispersive model was introduced in connection with regimes in which the leading third order dispersion becomes weak or degenerate \cite{Kawahara1972}. More generally, \eqref{eq:higher-order-kdv} describes the interaction between the quadratic KdV nonlinearity and odd order dispersion. The Cauchy problem for higher order KdV and Kawahara type equations has been studied extensively on both the real line and the periodic domain. On the real line, see for instance \cite{FigueiraHimonasYan2020,KatoKawahara2011,KatoKawahara2013}, while periodic results include \cite{GorskyHimonas2009,Hirayama2012,HongKwak2016, Kato2012}. Related control and stabilization problems for higher order KdV type equations have also been investigated in \cite{CapistranoGallegoKomornik2025, CapistranoKwakVielma2022} and the references therein.

For $j\ge2$, equation \eqref{eq:higher-order-kdv} is not, in general, a completely integrable flow. In particular, it should be distinguished from the higher flows of the integrable KdV hierarchy, which contain additional nonlinear terms. Nevertheless, \eqref{eq:higher-order-kdv} retains the Hamiltonian structure of the classical equation. Indeed, on the real line, define
\[\mathcal E_j(u):=\frac12\int_{\R}|\partial_x^ju(x)|^2\,dx-\frac16\int_{\R}u(x)^3\,dx.\]
A direct variational calculation gives
\[\frac{\delta\mathcal E_j}{\delta u}=(-1)^j\partial_x^{2j}u-\frac12u^2,\]
and hence \eqref{eq:higher-order-kdv} can formally be written as
\[\partial_tu=\partial_x\frac{\delta\mathcal E_j}{\delta u}.\]
Thus $\mathcal E_j$ is formally conserved for sufficiently smooth decaying solutions on the real line, and the analogous statement holds on the periodic domain. On a domain with boundary, however, integration by parts produces boundary fluxes, and the evolution cannot be specified without an appropriate collection of boundary conditions.

\subsection{The IBVP on the half line with Robin boundary conditions and the UTM formulation}

We study \eqref{eq:higher-order-kdv} on the right half line
\[\R^+:=(0,\infty).\]
We consider the following IBVP with Robin boundary conditions
\begin{equation}\label{mainequation}
\left\{
\begin{aligned}
\partial_tu+(-1)^{j+1}\partial_x^{2j+1}u+\frac12\partial_x(u^2)=&~{}0,
&&t>0,\ x>0,\\
u(0,x)=&~{}u_0(x),
&&x>0,\\
(\partial_x+\gamma_\ell)\partial_x^{\ell-1}u(t,0)=&~{}\varphi_\ell(t),
&&t>0,\ 1\le\ell\le j.
\end{aligned}
\right.
\end{equation}
Here
\[\boldsymbol{\gamma}:=(\gamma_1,\ldots,\gamma_j)\in\R^j,\qquad \bphi:=(\varphi_1,\ldots,\varphi_j).\]
The sign of the dispersive term in \eqref{mainequation} is chosen so that the corresponding IBVP on the right half line requires $j$ boundary conditions. The $\ell$th condition couples the adjacent traces $\partial_x^{\ell-1}u(t,0)$ and $\partial_x^\ell u(t,0)$, while the parameters $\gamma_\ell$ allow the relation between these two traces to vary from one boundary condition to another. More generally, Robin conditions are mixed boundary conditions coupling a function and its normal derivative and are commonly used to describe an intermediate boundary response between purely Dirichlet and purely Neumann conditions. In the present higher order setting, this coupling is imposed successively on adjacent boundary traces.

Although \eqref{mainequation} is naturally formulated for distinct parameters $\gamma_1,\ldots,\gamma_j$, the explicit construction in this paper concerns the diagonal Robin hierarchy
\begin{equation}\label{eq:diagonal-robin}
\gamma_1=\cdots=\gamma_j=\gamma,\qquad \gamma\in\R.
\end{equation}
In this case, the boundary conditions become
\[(\partial_x+\gamma)\partial_x^{\ell-1}u(t,0)=\varphi_\ell(t),\qquad 1\le\ell\le j.\]

IBVPs for KdV type equations on the half line have been studied by several different approaches. For the classical KdV equation, the Dirichlet IBVP was treated by Fokas, Himonas, and Mantzavinos \cite{fokas2016} using the unified transform method. Himonas, Madrid, and Yan \cite{Himonas2021} subsequently studied the Robin and Neumann IBVPs in Sobolev spaces using the same approach. The corresponding low regularity theory in modified Bourgain spaces was established by Himonas and Yan \cite{Himonas2022-1}, who proved local well posedness for $s>-\frac34$ with boundary data in $H^{s/3}(0,T)$.

For higher order dispersion, Yan \cite{yan2020} studied IBVPs with Dirichlet type boundary conditions in Sobolev spaces using the unified transform method. Himonas and Yan \cite{Himonas2022} later developed the corresponding low regularity theory for the boundary hierarchy
\[\partial_x^{\ell-1}u(t,0)=h_\ell(t),\qquad 1\le\ell\le j,\]
and obtained local well posedness for
\[s>-j+\frac14.\]
Their approach combines the unified transform representation with modified Fourier restriction spaces and bilinear estimates adapted to the higher order dispersive equation.

A different approach to KdV type IBVPs on the half line is based on boundary forcing operators. This method originates in the work of Colliander and Kenig \cite{CollianderKenig2002} and was further developed for KdV by Holmer \cite{holmer2006}. For fifth order KdV type equations, Cavalcante and Kwak \cite{CavalcanteKwak2020NoDEA,CavalcanteKwak2019CPAA} established low regularity results within this framework. Earlier IBVP results for Kawahara type equations were also obtained by Kuvshinov and Faminskii \cite{KuvshinovFaminskii2009} on a half strip and by Doronin and Larkin \cite{DoroninLarkin2008} on a bounded interval. Another approach, due to Bona, Sun, and Zhang \cite{BSZ2002,BSZ2006}, treats the nonhomogeneous KdV problem in the quarter plane through the Laplace transform representation of the boundary integral operator combined with Bourgain space estimates. The present problem lies at the intersection of two of these directions: higher order dispersion and Robin boundary coupling at low regularity. To the best of our knowledge, the higher order Robin hierarchy considered here has not previously been treated in this regime.

The temporal regularity of the boundary traces is determined by the linear dispersive relation. For $0\le \ell \le j$, set
\begin{equation}\label{eq:intro-r-ell}
r_\ell=r_\ell(s):=\frac{s+j-\ell}{2j+1}.
\end{equation}
We define the local boundary data space by
\[\mathcal H_{\mathrm{loc}}^s:=\prod_{\ell=1}^jH_{\mathrm{loc}}^{r_\ell(s)}(\R).\]
For $T>0$, we also write
\[\mathcal H_T^s:=\prod_{\ell=1}^jH^{r_\ell(s)}(0,T),\qquad \|\bphi\|_{\mathcal H_T^s}^2:=\sum_{\ell=1}^j\|\varphi_\ell\|_{H^{r_\ell(s)}(0,T)}^2.\]
Thus boundary data are given on the full time line, while only their restriction to $(0,T)$ enters the local IBVP. The exponent $r_\ell$ corresponds to the temporal regularity of the trace $\partial_x^\ell u(t,0)$, which is the less regular of the two adjacent traces appearing in the $\ell$th Robin condition.

Our approach is based on the Fokas unified transform method \cite{Fokas2008}. For the forced linear IBVP associated with \eqref{mainequation} and \eqref{eq:diagonal-robin}, we denote the UTM solution operator by
\[S_j[u_0,\bphi,f].\]
Its explicit contour representation is derived in Section \ref{sec:UTM}. For sufficiently smooth and decaying data, we derive the formula from the differential problem and then verify that the function defined by this formula satisfies the differential equation, the initial condition, and all $j$ Robin boundary conditions. Thus the differential and UTM formulations are equivalent in the smooth setting. The resulting operator is subsequently extended to the low regularity spaces used in the nonlinear problem.

\begin{definition}[UTM solution and local well posedness]\label{def:utm-well-posedness}
Let $T>0$ and set
\[\Omega_T:=(0,T)\times\R^+.\]
Assume \eqref{eq:diagonal-robin}. Let $X_T$ be a solution space continuously embedded into $C([0,T];H^s(\R^+))$. A function
\[u\in X_T\]
is called a UTM solution of \eqref{mainequation} on $\Omega_T$ if
\[u=S_j\left[u_0,\bphi,-\frac12\partial_x(u^2)\right]\]
in $X_T$, where the boundary data are restricted to $(0,T)$ and $S_j$ denotes the continuous extension of the smooth UTM solution operator.

We say that \eqref{mainequation} is locally well posed in the UTM sense in
\[H^s(\R^+)\times\mathcal H_{\mathrm{loc}}^s\]
if the following properties hold.
\begin{enumerate}
\item \textit{Existence.} For every
\[u_0\in H^s(\R^+),\qquad \bphi\in\mathcal H_{\mathrm{loc}}^s,\]
there exist $T>0$ and a solution space $X_T\subset C([0,T];H^s(\R^+))$ such that \eqref{mainequation} admits a UTM solution in $X_T$ on $\Omega_T$.
\item \textit{Uniqueness.} The UTM solution is unique in $X_T$.
\item \textit{Continuous dependence.} The lifespan $T$ can be chosen uniformly in a neighborhood of each initial and boundary data set, and the corresponding data to solution map is continuous into $X_T$.
\end{enumerate}
\end{definition}

Our main result is the following.

\begin{theorem}\label{mainresult}
Let $j\in\N$ and $\gamma\in\R$. Assume that
\[-j+\frac14<s<\frac32\]
and that \eqref{eq:diagonal-robin} holds. Then there exist $b$ and $\alpha$ satisfying
\[0<b<\frac12,\qquad \frac12<\alpha<1,\qquad \alpha\le1+\frac{s}{2j+1},\]
such that \eqref{mainequation} is locally well posed in the UTM sense in
\[H^s(\R^+)\times\mathcal H_{\mathrm{loc}}^s.\]
The solution space may be taken to be
\[X_{\Omega_T}^{s,b,\alpha}\cap C([0,T];H^s(\R^+)),\]
where $X_{\Omega_T}^{s,b,\alpha}$ is defined in Section \ref{sec:prelim}.
\end{theorem}

\begin{remark}\label{rem:intro-regularity-range}
For smooth solutions, the initial and boundary data satisfy the usual corner compatibility conditions. In particular,
\[\varphi_\ell(0)=\partial_x^\ell u_0(0)+\gamma\partial_x^{\ell-1}u_0(0),\qquad 1\le\ell\le j.\]
In the range of Theorem \ref{mainresult},
\[-j+\frac14<s<\frac32,\]
we have
\[-\frac12<r_\ell(s)<\frac12,\qquad 1\le\ell\le j.\]
Therefore, the corner traces are not defined in the data topology, and no compatibility condition between the initial and boundary data is imposed. The upper bound $s<\frac32$ is only used to remain below the boundary trace threshold. Higher regularity can be treated by imposing the corresponding compatibility conditions.
\end{remark}

\begin{remark}\label{rem:intro-related-results}
Theorem \ref{mainresult} is consistent with two previously studied cases.
\begin{enumerate}
\item When $j=1$, the problem becomes
\[\partial_tu+\partial_x^3u+\frac12\partial_x(u^2)=0,\qquad \partial_xu(t,0)+\gamma u(t,0)=\varphi_1(t),\]
and
\[r_1(s)=\frac{s}{3}.\]
The regularity range in Theorem \ref{mainresult} becomes
\[-\frac34<s<\frac32,\]
with $\varphi_1\in H^{s/3}(0,T)$. Thus Theorem \ref{mainresult} recovers the low regularity Robin and Neumann theory for the classical KdV equation established by Himonas and Yan \cite{Himonas2022-1}.
\item Assume that $\gamma\neq0$ and set
\[h_\ell:=\gamma^{-1}\varphi_\ell,\qquad 1\le\ell\le j.\]
Then the Robin condition can be written as
\[\partial_x^{\ell-1}u(t,0)+\gamma^{-1}\partial_x^\ell u(t,0)=h_\ell(t).\]
Hence the formal limit $|\gamma|\to\infty$ gives
\[\partial_x^{\ell-1}u(t,0)=h_\ell(t),\qquad 1\le\ell\le j,\]
which is the higher order Dirichlet problem studied by Himonas and Yan \cite{Himonas2022}.
\end{enumerate}
Thus, at the level of the boundary conditions, the present Robin hierarchy contains the derivative hierarchy when $\gamma=0$ and formally approaches the Dirichlet hierarchy as $|\gamma|\to\infty$. The latter statement concerns only the boundary operator and does not assert convergence of the corresponding solutions. Some corresponding reductions of the UTM representation are recorded in Appendix \ref{app:special-cases}.
\end{remark}

\begin{remark}\label{rem:intro-general-robin}
The restriction \eqref{eq:diagonal-robin} is a structural assumption in the present construction. When the Robin parameters are equal, the boundary recursion involves powers of a single parameter $\gamma$, which leads to the polynomial structure used in the Lagrange interpolation argument and, ultimately, to the exact factorization \eqref{eq:intro-boundary-factorization}. For distinct parameters $\gamma_1,\ldots,\gamma_j$, the corresponding coefficients involve products of different Robin parameters, and this one parameter factorization is no longer available in the same form. Extending the explicit UTM construction and the low regularity theory to the general Robin hierarchy remains an open problem.
\end{remark}

\subsection{Construction of the UTM representation for Robin boundary conditions}

The main novelty of this paper is the explicit resolution of the boundary system arising in the UTM formulation of the higher order IBVP with Robin boundary conditions. In contrast with the Dirichlet case, the Robin conditions couple adjacent boundary traces and produce a nontrivial algebraic structure in the global relation. We briefly describe the mechanism of this construction. The spectral sectors, rotated roots, boundary transforms, and polynomial coefficients appearing below are defined precisely in Section \ref{sec:UTM}.

The global relation contains the $2j+1$ boundary transforms
\[\widetilde g_\ell(t,\lambda),\qquad 0\le \ell\le2j.\]
Under the diagonal Robin condition, the lower order transforms satisfy
\[\widetilde g_\ell(t,\lambda)=(-\gamma)^\ell\widetilde g_0(t,\lambda)+\sum_{q=1}^\ell(-\gamma)^{\ell-q}\widetilde\varphi_q(t,\lambda),\qquad 1\le \ell \le j.\]
Thus the unknown boundary transforms reduce to
\[\widetilde g_{j+1},\ldots,\widetilde g_{2j},\widetilde g_0,\]
giving exactly $j+1$ unknowns.

For each spectral sector $D_{2p}^+$, the global relation is evaluated at $j+1$ rotated spectral points
\[\alpha_{p,n}k,\qquad 1\le n\le j+1,\]
which leads to a sectorwise linear system
\begin{equation}\label{eq:intro-sector-system}
A_p(k)X(t,k)=\mathcal I_p(t,k).
\end{equation}
The matrix $A_p(k)$ has a Vandermonde structure in its first $j$ columns, while its last column is generated by the Robin recursion through the polynomial
\[\mathfrak a_0(z,k):=\sum_{q=0}^ji^{j-q}k^{j-q}(-\gamma)^qz^{2j-q}.\]

A direct computation of the full inverse $A_p(k)^{-1}$ would be lengthy and unnecessary. The UTM representation requires only the linear combination
\begin{equation}\label{eq:intro-required-combination}
\sum_{m=0}^{j-1}X_{m+1}(t,k)+\mathfrak a_0(1,k)X_{j+1}(t,k).
\end{equation}
The key observation is that this quantity can be determined directly from \eqref{eq:intro-sector-system} by Lagrange interpolation.

Let $Q_p$ be the polynomial whose roots are $\alpha_{p,n}$, and let $L_{p,n}$ denote the corresponding Lagrange basis polynomials. The interpolation argument yields
\[\mathcal P_{p,n}(k)=L_{p,n}(1)-\frac{Q_p(1)}{Q_p'(\alpha_{p,n})}\frac{k}{k-i\gamma},\qquad 1\le n\le j+1,\]
which gives precisely the coefficients needed to recover \eqref{eq:intro-required-combination}. The common Robin parameter is crucial in this step. The corresponding telescoping identities concentrate the dependence on $\gamma$ into the single factor
\[\frac{k}{k-i\gamma}.\]

A second interpolation argument gives the exact boundary factorization
\begin{equation}\label{eq:intro-boundary-factorization}
(-1)^{j+1}\left(\mathfrak b_\ell(k)-\sum_{n=1}^{j+1}\mathcal P_{p,n}(k)\mathfrak b_\ell(\alpha_{p,n}k)\right)=c_{p,\ell}\frac{k^{2j+1-\ell}}{k-i\gamma},\qquad 1\le p,\ell\le j.
\end{equation}
The coefficients $c_{p,\ell}$ depend only on $j$, $p$, and $\ell$, and are independent of both $k$ and $\gamma$. This identity identifies simultaneously the high frequency order of the boundary multiplier and the only possible Robin pole $k=i\gamma$.

The factor $k-i\gamma$ also leads to a residue contribution when the pole $k=i\gamma$ lies in the relevant UTM contour. The sector geometry, the interpolation arguments, the residue contribution, and the verification of the initial and Robin boundary conditions are carried out in Section \ref{sec:UTM}.

\medskip

Once the UTM representation has been constructed, the proof of Theorem \ref{mainresult} is completed by combining the linear estimates for $S_j$ with the bilinear estimates recalled in Section \ref{sec:prelim}. In particular, Proposition \ref{prop:linear-estimates} controls the forced UTM solution operator in the modified restriction spaces, while Lemma \ref{lem:restriction-nonlinear} provides the time localized nonlinear estimate with a positive power of the lifespan. These estimates yield a contraction for
\[u=S_j\left[u_0,\bphi,-\frac12\partial_x(u^2)\right]\]
on a sufficiently short time interval. Thus the main additional difficulty lies in the construction and analysis of the Robin UTM formula described above.
\subsection{Organization of the paper}\label{sec:intro-organization}

The rest of the paper is organized as follows. In Section \ref{sec:prelim}, we introduce the function spaces and basic estimates used throughout the paper and formulate the linear IBVP with Robin boundary conditions. In Section \ref{sec:UTM}, we derive the UTM representation, resolve the sectorwise boundary systems, and verify the equivalence between the differential and UTM formulations. In Section \ref{sec:reduced-ibvp}, we establish the linear estimates for the reduced IBVP with Robin boundary conditions. In Section \ref{sec:forced-linear}, we obtain the estimates for the forced linear IBVP. In Section \ref{sec:well-posedness}, we prove Theorem \ref{mainresult} by a contraction argument. Finally, Appendix \ref{app:special-cases} records the reductions to the classical KdV Robin and Neumann cases and the formal relation with the higher order Dirichlet problem.

\subsection*{Acknowledgments}
This work was carried out during some visits of the authors to the Universidad Nacional de Colombia-Sede Manizales and Ewha Womans University. The authors would like to thank the Universities for its hospitality. C. K. was partially supported by Young Research Program of the National Research Foundation of Korea(NRF) grant funded by the Korea government(MSIT) (No. RS-2023-00210210) and Global - Learning \& Academic research institution for Master’s·PhD students, and Postdocs(G-LAMP) Program of the National Research Foundation of Korea(NRF) grant funded by the Ministry of Education(No. RS-2025-25442252).

\section{Preliminaries}\label{sec:prelim}
We write 
\[\R^{+}:=(0,\infty),\qquad \Omega_T:=(0,T) \times \R^{+}.\]
The notation $A\lesssim B$ means that $A\le CB$ with a universal constant $C>0$. When the implicit constant depends on parameters, we indicate this dependence by a subscript. We write $A\sim B$ if both $A\lesssim B$ and $B\lesssim A$ hold. 

For a function $u=u(t,x)$ on $\R^{2}$, we use the space-time Fourier transform
\[\mathcal{F}(u)(\tau,\xi)=\int_{\R^{2}}e^{-i(x\xi+t\tau)}u(t,x)\,dx\,dt,\]
with the inversion formula
\[u(t,x)=\frac{1}{(2\pi)^2}\int_{\R^{2}}e^{i(x\xi+t\tau)}\mathcal{F}(u)(\tau,\xi)\,d\xi\,d\tau.\]
For functions of one variable (either temporal or spatial), we use the analogous one-dimensional Fourier transform and inversion conventions as follows:
\[\widehat{f}(\zeta)=\int_{\R}e^{-iy\zeta}f(y)\,dy,\qquad f(y)=\frac{1}{2\pi}\int_{\R}e^{iy\zeta}\widehat{f}(\zeta)\,d\zeta.\]
We use the same notation $\widehat{\cdot}$ for both temporal and spatial one-dimensional Fourier transforms; the underlying variable will be clear from the context.

\subsection{Function spaces}

For an interval $I\subset\R$, possibly unbounded, the Sobolev space $H^r(I)$ is understood as the restriction space
\[H^r(I):=\{g:\ g=G|_I\ \text{for some }G\in H^r(\R)\},\]
with norm
\[\|g\|_{H^r(I)}:=\inf\{\|G\|_{H^r(\R)}:\ G|_I=g\}.\]
We shall use the standard embedding
\[H^{\sigma_1}(I)\hookrightarrow H^{\sigma_0}(I),\qquad \sigma_1\ge\sigma_0.\]

For real numbers $s$ and $b$, the $X^{s,b}(\R^{2})$ spaces corresponding to the linear part of \eqref{mainequation} are defined by \[\|u\|_{X^{s,b}(\R^{2})}^{2}:=\int_{\R}\int_{\R}\bra{\xi}^{2s}\bra{\tau-\xi^{2j+1}}^{2b}|\mathcal{F}(u)(\tau,\xi)|^{2}\,d\xi\,d\tau,\]
where we use the Japanese bracket $\bra{\cdot} = (1+ |\cdot|^2)^{\frac12}$. The $X^{s,b}$ spaces were first introduced in their present form by Bourgain \cite{Bourgain1993-1, Bourgain1993-2} in his works on the periodic NLS and generalized KdV equations, although related spaces had already appeared in Beals' work on one-dimensional wave equations \cite{Beals}. Since then, these spaces have played a central role in the analysis of dispersive equations and have been further developed by many authors, in particular, Kenig, Ponce, and Vega \cite{KPV1996} and Tao \cite{Tao2001}.

Although the $X^{s,b}$ framework with $b>\frac12$ is well adapted to the initial value problem for dispersive equations, the half-line problems require a different choice of the $X^{s,b}$ exponent. Indeed, the estimates for the UTM solution of the reduced Robin IBVP yield the required $X^{s,b}$ controls only for $0\le b<\frac12$ (see Proposition \ref{prop:reduced-ibvp} below). Hence the solution space for the IBVP must be built with an $X^{s,b}$ exponent below $\frac12$.

This creates an additional difficulty in the controls of the inhomogeneous and nonlinear terms. The temporal trace estimates required to construct the boundary correction are not controlled by the $X^{s,-b}$ norm alone over the regularity range considered here, especially when $s<0$. To recover the required temporal control, we supplement the $X^{s,-b}$ norm with the $Y^{s,-b}$ norm, see Proposition \ref{prop:inhomogeneous-whole-line} below or Theorem 3.3 in \cite{Himonas2022}. The nonlinear estimates are therefore formulated to provide the corresponding $Y^{s,-b}$ control.

We therefore introduce the temporal space $Y^{s,b}(\R^2)$, which is used to control the temporal traces of the whole line Duhamel term. For $s,b\in\R$, we define
\[\|u\|_{Y^{s,b}(\R^2)}^2:=\int_{\R}\int_{\R}\langle\tau\rangle^{\frac{2s}{2j+1}}\langle\tau-\xi^{2j+1}\rangle^{2b}|\mathcal F(u)(\tau,\xi)|^2\,d\xi\,d\tau.\]

The choice $b<\frac12$, however, creates a loss of temporal integrability in the bilinear estimates when one or both input frequencies are small. To compensate for this loss, we introduce the space $\mathcal D^\alpha(\R^2)$, defined by
\[\|u\|_{\mathcal D^\alpha(\R^2)}^2:=\int_{\R}\int_{-1}^1\langle\tau\rangle^{2\alpha}|\mathcal F(u)(\tau,\xi)|^2\,d\xi\,d\tau.\]
The restriction of the $\xi$ integration to $[-1,1]$ localizes this additional control to low spatial frequencies, while the condition $\alpha>\frac12$ provides the temporal integrability required in the bilinear estimates. We then define
\[X^{s,b,\alpha}(\R^2):=X^{s,b}(\R^2)\cap\mathcal D^\alpha(\R^2),\]
equipped with the norm
\[\|u\|_{X^{s,b,\alpha}(\R^2)}^2:=\|u\|_{X^{s,b}(\R^2)}^2+\|u\|_{\mathcal D^\alpha(\R^2)}^2.\]
Equivalently,
\[\|u\|_{X^{s,b,\alpha}(\R^2)}^2\sim\int_{\R}\int_{\R}\left(\langle\xi\rangle^s\langle\tau-\xi^{2j+1}\rangle^b+\mathbf{1}_{|\xi|\le1}(\xi)\langle\tau\rangle^\alpha\right)^2|\mathcal F(u)(\tau,\xi)|^2\,d\xi\,d\tau,\]
where $\mathbf{1}_A$ is the characteristic function of a set $A$.

The following estimates are the forms of the higher order KdV bilinear estimates from \cite[Theorems 1.3 and 1.4]{Himonas2022} that will be used in the nonlinear argument.

\begin{proposition}[Bilinear estimates]\label{prop:bilinear}
Assume that
\[-j+\frac14<s<\frac32.\]
Then there exist $b_0$, $b$, $\widetilde b$, $b_1$, $\alpha$, and $\widetilde\alpha$ satisfying
\[0<b_0<b<\widetilde b<b_1<\frac12,\qquad \frac12<\alpha<\widetilde\alpha<1,\]
and
\[\widetilde\alpha\le1-\widetilde b,\qquad \alpha\le1+\frac{s}{2j+1},\]
such that, for all $F,G\in X^{s,b,\alpha}(\R^2)$,
\begin{equation}\label{eq:bilinear-X-used}
\|\partial_x(FG)\|_{X^{s,-\widetilde b,\widetilde\alpha-1}(\R^2)}\lesssim\|F\|_{X^{s,b,\alpha}(\R^2)}\|G\|_{X^{s,b,\alpha}(\R^2)},
\end{equation}
and, for all $F,G\in X^{s,b_0}(\R^2)$,
\begin{equation}\label{eq:bilinear-Y-used}
\|\partial_x(FG)\|_{Y^{s,-b_1}(\R^2)}\lesssim\|\partial_x(FG)\|_{X^{s,-b_1}(\R^2)}+\|F\|_{X^{s,b_0}(\R^2)}\|G\|_{X^{s,b_0}(\R^2)}.
\end{equation}
\end{proposition}

For the forcing term, we use the intersection space
\[\mathcal Z^{s,b,\alpha}(\R^2):=X^{s,b,\alpha}(\R^2)\cap Y^{s,b}(\R^2),\]
equipped with the norm
\[\|f\|_{\mathcal Z^{s,b,\alpha}(\R^2)}^2:=\|f\|_{X^{s,b,\alpha}(\R^2)}^2+\|f\|_{Y^{s,b}(\R^2)}^2.\]

For $T>0$, we define the corresponding restriction spaces on $\Omega_T$. First,
\[X_{\Omega_T}^{s,b,\alpha}:=\{u:\ u=v|_{\Omega_T}\ \text{for some }v\in X^{s,b,\alpha}(\R^2)\},\]
with norm
\[\|u\|_{X_{\Omega_T}^{s,b,\alpha}}:=\inf\{\|v\|_{X^{s,b,\alpha}(\R^2)}:\ v|_{\Omega_T}=u\}.\]
We also write
\[X_{(0,T)\times\R^+}^{s,b,\alpha}:=X_{\Omega_T}^{s,b,\alpha}.\]

Similarly,
\[Y_{\Omega_T}^{s,b}:=\{u:\ u=v|_{\Omega_T}\ \text{for some }v\in Y^{s,b}(\R^2)\},\]
with norm
\[\|u\|_{Y_{\Omega_T}^{s,b}}:=\inf\{\|v\|_{Y^{s,b}(\R^2)}:\ v|_{\Omega_T}=u\},\]
and we write
\[Y_{(0,T)\times\R^+}^{s,b}:=Y_{\Omega_T}^{s,b}.\]

Finally,
\[\mathcal Z_{\Omega_T}^{s,b,\alpha}:=\{f:\ f=F|_{\Omega_T}\ \text{for some }F\in\mathcal Z^{s,b,\alpha}(\R^2)\},\]
with norm
\[\|f\|_{\mathcal Z_{\Omega_T}^{s,b,\alpha}}:=\inf\{\|F\|_{\mathcal Z^{s,b,\alpha}(\R^2)}:\ F|_{\Omega_T}=f\}.\]
We also write
\[\mathcal Z_{(0,T)\times\R^+}^{s,b,\alpha}:=\mathcal Z_{\Omega_T}^{s,b,\alpha}.\]
We use throughout the notation $r_\ell=r_\ell(s)$ and the boundary data spaces $\mathcal H_{\mathrm{loc}}^s$ and $\mathcal H_T^s$ introduced in Section \ref{sec:introduction}.

\subsection{Basic lemmas}
We fix a cutoff function 
\[\psi\in C_0^\infty(-1,1),\qquad 0\le\psi\le1,\qquad \psi(t)=1\quad \mbox{for} \quad |t|\le\frac12.\]
For $T>0$, we set
\begin{equation}\label{eq:psi_T}
\psi_T(t):=\psi\left(\frac{t}{2T}\right).
\end{equation}
Then 
\[\psi_T(t)=1,\qquad |t|\le T.\]

\begin{lemma}\label{lem:bourgain-embedding}
Let $s\in\R$. If $b_1\ge b_0$ and $\alpha_1\ge\alpha_0$, then
\[X^{s,b_1,\alpha_1}(\R^2)\hookrightarrow X^{s,b_0,\alpha_0}(\R^2).\]
Consequently,
\[X_{\Omega_T}^{s,b_1,\alpha_1}\hookrightarrow X_{\Omega_T}^{s,b_0,\alpha_0}.\]
\end{lemma}

\begin{proof}
The whole line embedding follows directly from
\[\bra{\tau-\xi^{2j+1}}^{b_0}\le\bra{\tau-\xi^{2j+1}}^{b_1}\]
and
\[\bra{\tau}^{\alpha_0}\le\bra{\tau}^{\alpha_1}.\]
The restriction space embedding follows by taking the infimum over all admissible extensions.
\end{proof}

\begin{lemma}\label{lem:time-cutoff}
Let $s\in\R$. Suppose that
\[-\frac12<b_0<b_1<\frac12,\qquad \frac12<\alpha_0<\alpha_1<1.\]
Then, for every $0<T\le1$,
\begin{equation}\label{eq:time-cutoff-X}
\|\psi_Tu\|_{X^{s,b_0,\alpha_0-1}(\R^2)} \lesssim_{\psi,b_0,b_1,\alpha_0,\alpha_1} T^{\min\{b_1-b_0,\alpha_1-\alpha_0\}} \|u\|_{X^{s,b_1,\alpha_1-1}(\R^2)}.
\end{equation}
In particular,
\begin{equation}\label{eq:time-cutoff-pure-X}
\|\psi_Tu\|_{X^{s,b_0}(\R^2)}\lesssim_{\psi,b_0,b_1} T^{b_1-b_0}\|u\|_{X^{s,b_1}(\R^2)}.
\end{equation}
Here, $\psi_T$ is introduced as in \eqref{eq:psi_T}.
\end{lemma}

\begin{proof}
The estimate \eqref{eq:time-cutoff-pure-X} is the standard time cutoff estimate in $X^{s,b}$. Since
\[-\frac12<\alpha_0-1<\alpha_1-1<0,\]
we use the same estimate on the low spatial frequency region. For $|\xi|\le1$,
\[\bra{\tau}\sim\bra{\tau-\xi^{2j+1}},\]
and hence, for every $\beta\in\R$,
\[\|u\|_{\mathcal D^\beta(\R^2)}\sim\|P_{\le1}u\|_{X^{0,\beta}(\R^2)},\]
where $P_{\le1}$ denotes the spatial Fourier projection onto $|\xi|\le1$. Since $P_{\le1}$ commutes with multiplication by $\psi_T$, \eqref{eq:time-cutoff-pure-X} with $s=0$, $b_0=\alpha_0-1$, and $b_1=\alpha_1-1$ gives
\[\begin{aligned}
\|\psi_Tu\|_{\mathcal D^{\alpha_0-1}(\R^2)} &\lesssim \|\psi_TP_{\le1}u\|_{X^{0,\alpha_0-1}(\R^2)}\\
&\lesssim_{\psi,\alpha_0,\alpha_1} T^{\alpha_1-\alpha_0}\|P_{\le1}u\|_{X^{0,\alpha_1-1}(\R^2)}\\
&\lesssim T^{\alpha_1-\alpha_0}\|u\|_{\mathcal D^{\alpha_1-1}(\R^2)}.
\end{aligned}\]
Combining this estimate with \eqref{eq:time-cutoff-pure-X} proves \eqref{eq:time-cutoff-X}.
\end{proof}

\begin{lemma}\label{lem:boundary-extension}
Let $-\frac12<r<\frac12$ and $T>0$. For every $g\in H^r(0,T)$, let $\mathcal E_Tg$ denote the zero extension of $g$ to $\R$. Then
\begin{equation}\label{eq:boundary-extension}
\|\mathcal E_Tg\|_{H^r(\R)}\lesssim \|g\|_{H^r(0,T)}.
\end{equation}
The implicit constant is independent of $T>0$.
\end{lemma}

\begin{proof}
Let $G\in H^r(\R)$ be an extension of $g$. For $-\frac12<r<\frac12$, multiplication by the characteristic function of a half line is bounded on $H^r(\R)$, see for instance \cite{Strichartz1967}. By translation invariance, the multiplier norm is independent of the endpoint. Since
\[\mathbf{1}_{(0,T)}=\mathbf{1}_{(0,\infty)}-\mathbf{1}_{(T,\infty)},\]
we obtain
\[\|\mathcal E_Tg\|_{H^r(\R)}=\|\mathbf{1}_{(0,T)}G\|_{H^r(\R)}\lesssim_r\|G\|_{H^r(\R)},\]
with a constant independent of $T$. Taking the infimum over all extensions $G$ proves \eqref{eq:boundary-extension}.
\end{proof}

\begin{lemma}\label{lem:ray-integral}
Let $a\in\C$ with $\im a>0$. If $\sigma\ge0$ and $\bra{k}^\sigma q\in L^2(\R^+)$, then
\[\left\|\int_0^\infty e^{iakx}q(k)\,dk\right\|_{H^\sigma(\R^+)}\lesssim_{\sigma,a}\|\bra{k}^\sigma q\|_{L^2(\R^+)}.\]
The same estimate holds when the $k$ integration is restricted to any measurable subset of $\R^+$.
\end{lemma}

\begin{proof}
Write
\[a=\re a+i\im a,\qquad \beta:=\im a>0.\]
We first prove the $L^2$ estimate. For $q\in C_0^\infty(\R^+)$, set
\[Lq(x):=\int_0^\infty e^{-kx}q(k)\,dk.\]
A direct computation gives
\[\begin{aligned}\|Lq\|_{L^2(\R^+)}^2=&~{}\int_0^\infty\int_0^\infty\frac{q(k)\overline{q(\ell)}}{k+\ell}\,dk\,d\ell\\
=&~{}\int_0^\infty\left(\int_0^\infty\frac{q(k)}{k+\ell}\,dk\right)\overline{q(\ell)}\,d\ell\\
\le&~{}\left(\int_0^\infty\left|\int_0^\infty\frac{q(k)}{k+\ell}\,dk\right|^2d\ell\right)^{\frac12}\|q\|_{L^2(\R^+)}.
\end{aligned}\]
Note that
\[\int_0^\infty\frac{k^{-\frac12}}{k+\ell}\,dk=\ell^{-\frac12}\int_0^\infty\frac{y^{-\frac12}}{1+y}\,dy=\pi\ell^{-\frac12}.\]
Since
\[\begin{aligned}\left|\int_0^\infty\frac{q(k)}{k+\ell}\,dk\right|^2\le&~{}\left(\int_0^\infty\frac{k^{-\frac12}}{(k+\ell)\ell^{-\frac12}}\,dk\right)\left(\int_0^\infty\frac{|q(k)|^2\ell^{-\frac12}}{(k+\ell)k^{-\frac12}}\,dk\right)\\
=&~{}\pi\int_0^\infty\frac{|q(k)|^2\ell^{-\frac12}}{(k+\ell)k^{-\frac12}}\,dk,
\end{aligned}\]
we obtain
\[\begin{aligned}\int_0^\infty\left|\int_0^\infty\frac{q(k)}{k+\ell}\,dk\right|^2d\ell\le&~{}\pi\int_0^\infty\int_0^\infty\frac{|q(k)|^2\ell^{-\frac12}}{(k+\ell)k^{-\frac12}}\,dk\,d\ell\\
=&~{}\pi\int_0^\infty\frac{|q(k)|^2}{k^{-\frac12}}\left(\int_0^\infty\frac{\ell^{-\frac12}}{k+\ell}\,d\ell\right)dk\\
=&~{}\pi^2\int_0^\infty|q(k)|^2\,dk.
\end{aligned}\]
It follows that
\[\|Lq\|_{L^2(\R^+)}\lesssim\|q\|_{L^2(\R^+)}.\]
By density, the same estimate holds for every $q\in L^2(\R^+)$.

Since
\[\left|\int_0^\infty e^{iakx}q(k)\,dk\right|\le\int_0^\infty e^{-\beta kx}|q(k)|\,dk,\]
the change of variables $y=\beta x$ and the estimate above yield
\[\left\|\int_0^\infty e^{iakx}q(k)\,dk\right\|_{L^2(\R^+)}\lesssim_a\|q\|_{L^2(\R^+)}.\]

Let $N\in\N$. For every $0\le\nu\le N$,
\[\partial_x^\nu\int_0^\infty e^{iakx}q(k)\,dk=(ia)^\nu\int_0^\infty e^{iakx}k^\nu q(k)\,dk.\]
Applying the $L^2$ estimate to each derivative, we obtain
\begin{equation}\label{eq:integer-derivative}
\left\|\int_0^\infty e^{iakx}q(k)\,dk\right\|_{H^N(\R^+)}
\lesssim_{N,a}\sum_{\nu=0}^N\|k^\nu q(k)\|_{L^2(\R^+)}
\lesssim_{N,a}\|\bra{k}^Nq(k)\|_{L^2(\R^+)}.
\end{equation}
For general $\sigma\ge0$, choose $N\in\N_0$ and $0\le\theta<1$ such that
\[\sigma=N+\theta.\]
If $\theta=0$, the desired estimate is exactly \eqref{eq:integer-derivative}. If $0<\theta<1$, applying the complex interpolation theorem to \eqref{eq:integer-derivative} for the consecutive integers $N$ and $N+1$, we conclude
\[\left\|\int_0^\infty e^{iakx}q(k)\,dk\right\|_{H^\sigma(\R^+)}
\lesssim_{\sigma,a}\|\bra{k}^\sigma q(k)\|_{L^2(\R^+)}.\]
\end{proof}

\subsection{The linear IBVP with Robin boundary conditions}\label{sec:linear-robin}
For the linear estimates and the fixed point argument, we consider
\begin{equation}\label{eq:linear-robin}
\left\{
\begin{aligned}
\partial_tu+(-1)^{j+1}\partial_x^{2j+1}u=&~{}f(t,x),&&0<t<T,\ x>0,\\
u(0,x)=&~{}u_0(x),&&x>0,\\
(\partial_x+\gamma)\partial_x^{\ell-1}u(t,0)=&~{}\varphi_\ell(t),&&0<t<T,\ 1\le\ell\le j.
\end{aligned}
\right.
\end{equation}
The corresponding UTM solution operator, which will be derived in Section \ref{sec:UTM}, is denoted by
\[S_j[u_0,\bphi,f],\qquad \bphi=(\varphi_1,\ldots,\varphi_j).\]
Thus a solution of the nonlinear problem \eqref{mainequation} is constructed as a fixed point of
\[u=S_j\left[u_0,\bphi,-\frac12\partial_x(u^2)\right].\]

For the Fourier transforms on the half line, we write
\[\widehat{u_0}(k):=\int_0^\infty e^{-ikx}u_0(x)\,dx,\qquad \widehat f(t,k):=\int_0^\infty e^{-ikx}f(t,x)\,dx,\qquad \im k\le0.\]
We also define
\begin{equation}\label{eq:F}
F(t,k):=\int_0^t e^{-ik^{2j+1}\tau}\widehat f(\tau,k)\,d\tau
\end{equation}
and
\[\widetilde\varphi_\ell(t,\zeta):=\int_0^t e^{-i\zeta\tau}\varphi_\ell(\tau)\,d\tau,\qquad 1\le\ell\le j.\]
Let
\[\omega:=e^{i2\pi/(2j+1)}.\]

For $1\le\kappa\le2j+1$, define
\[D_\kappa^+:=\left\{re^{i\theta}:r>0,\ \frac{(\kappa-1)\pi}{2j+1}<\theta<\frac{\kappa\pi}{2j+1}\right\}.\]
We then set
\[D^+:=\bigcup_{p=1}^jD_{2p}^+.\]
The boundary $\partial D^+$ is the union of the positively oriented boundaries $\partial D_{2p}^+$.

The sectors $D_\kappa^+$ and the domain $D^+$ are illustrated in Figure \ref{fig:domains}. Since
\[\frac{j\pi}{2j+1}<\frac{\pi}{2}<\frac{(j+1)\pi}{2j+1},\]
the positive imaginary axis lies inside the sector $D_{j+1}^+$, whose index $j+1$ is even exactly when $j$ is odd. Hence the possible Robin pole $k=i\gamma$ belongs to $D^+$ precisely when $j$ is odd and $\gamma>0$. If $j$ is even or $\gamma<0$, the pole lies outside $D^+$, while for $\gamma=0$ the apparent singularity at the origin is removable. This distinction will be relevant in the derivation of the UTM representation in Section \ref{sec:UTM}.

\begin{figure}[H]
    \centering
    \begin{subfigure}[t]{0.43\textwidth}
        \centering
        \resizebox{\textwidth}{!}{
        \begin{tikzpicture}[x=0.75pt,y=0.75pt,yscale=-1,xscale=1]
        % uncomment if require: \path (0,258);

        % Shape: Polygon
        \draw [color={rgb,255:red,255; green,255; blue,255},draw opacity=1]
        [fill={rgb,255:red,232; green,249; blue,212},fill opacity=1]
        (5,44.6) -- (9.86,26.13) -- (18.36,7.67) -- (75.42,-4.5) -- (203.52,165.87) -- (5,65.58) -- cycle;

        \draw [color={rgb,255:red,255; green,255; blue,255},draw opacity=1]
        [fill={rgb,255:red,232; green,249; blue,212},fill opacity=1]
        (166.49,-0.72) -- (193.2,-0.72) -- (216.87,0.96) -- (239.34,-2.4) -- (203.52,165.87) -- cycle;

        \draw [color={rgb,255:red,255; green,255; blue,255},draw opacity=1]
        [fill={rgb,255:red,232; green,249; blue,212},fill opacity=1]
        (331.01,-3.24) -- (383.82,13.54) -- (395.36,64.74) -- (203.52,165.87) -- cycle;

        \draw [line width=3,line join=round,line cap=round]
        (203.52,82.78) .. controls (203.52,82.78) and (203.52,82.78) .. (203.52,82.78);

        \draw (237.51,149.5) .. controls (242.98,154.12) and (242.06,163.3) .. (238.73,165.45);

        \draw [line width=0.75,line join=round,line cap=round]
        (252.08,33.69) .. controls (252.08,33.69) and (252.08,33.69) .. (252.08,33.69);
        \draw [line width=0.75,line join=round,line cap=round]
        (255.12,37.88) .. controls (255.12,37.88) and (255.12,37.88) .. (255.12,37.88);
        \draw [line width=0.75,line join=round,line cap=round]
        (258.16,42.08) .. controls (258.16,42.08) and (258.16,42.08) .. (258.16,42.08);
        \draw [line width=0.75,line join=round,line cap=round]
        (261.19,46.27) .. controls (261.19,46.27) and (261.19,46.27) .. (261.19,46.27);

        \draw [line width=0.75,line join=round,line cap=round]
        (139.17,45.44) .. controls (139.17,45.44) and (139.17,45.44) .. (139.17,45.44);
        \draw [line width=0.75,line join=round,line cap=round]
        (141.59,41.24) .. controls (141.59,41.24) and (141.59,41.24) .. (141.59,41.24);
        \draw [line width=0.75,line join=round,line cap=round]
        (144.63,37.04) .. controls (144.63,37.04) and (144.63,37.04) .. (144.63,37.04);
        \draw [line width=0.75,line join=round,line cap=round]
        (147.06,33.69) .. controls (147.06,33.69) and (147.06,33.69) .. (147.06,33.69);

        % Axes
        \draw [line width=0.75] (203.53,1.02) -- (204.12,240.5);
        \draw [shift={(204.12,243.5)},rotate=269.86]
        [fill={rgb,255:red,0; green,0; blue,0},line width=0.08,draw opacity=0]
        (10.72,-5.15) -- (0,0) -- (10.72,5.15) -- (7.12,0) -- cycle;
        \draw [shift={(203.52,-1.98)},rotate=89.86]
        [fill={rgb,255:red,0; green,0; blue,0},line width=0.08,draw opacity=0]
        (10.72,-5.15) -- (0,0) -- (10.72,5.15) -- (7.12,0) -- cycle;

        \draw [line width=0.75] (18.32,165.45) -- (394.79,165.45);
        \draw [shift={(397.79,165.45)},rotate=180]
        [fill={rgb,255:red,0; green,0; blue,0},line width=0.08,draw opacity=0]
        (10.72,-5.15) -- (0,0) -- (10.72,5.15) -- (7.12,0) -- cycle;
        \draw [shift={(15.32,165.45)},rotate=0]
        [fill={rgb,255:red,0; green,0; blue,0},line width=0.08,draw opacity=0]
        (10.72,-5.15) -- (0,0) -- (10.72,5.15) -- (7.12,0) -- cycle;

        % Dashed rays
        \draw [line width=0.75,dash pattern={on 4.5pt off 4.5pt}] (399,65.16) -- (203.52,165.87);
        \draw [shift={(307.48,112.31)},rotate=152.74]
        [color={rgb,255:red,0; green,0; blue,0},line width=0.75]
        (10.93,-3.29) .. controls (6.95,-1.4) and (3.31,-0.3) .. (0,0)
        .. controls (3.31,0.3) and (6.95,1.4) .. (10.93,3.29);

        \draw [line width=0.75,dash pattern={on 4.5pt off 4.5pt}] (331.01,-3.24) -- (203.52,165.87);
        \draw [shift={(263.65,86.1)},rotate=307.01]
        [color={rgb,255:red,0; green,0; blue,0},line width=0.75]
        (10.93,-3.29) .. controls (6.95,-1.4) and (3.31,-0.3) .. (0,0)
        .. controls (3.31,0.3) and (6.95,1.4) .. (10.93,3.29);

        \draw [line width=0.75,dash pattern={on 4.5pt off 4.5pt}] (239.34,-2.4) -- (203.52,165.87);
        \draw [shift={(222.88,74.89)},rotate=102.02]
        [color={rgb,255:red,0; green,0; blue,0},line width=0.75]
        (10.93,-3.29) .. controls (6.95,-1.4) and (3.31,-0.3) .. (0,0)
        .. controls (3.31,0.3) and (6.95,1.4) .. (10.93,3.29);

        \draw [line width=0.75,dash pattern={on 4.5pt off 4.5pt}] (78.46,-1.56) -- (203.52,165.87);
        \draw [shift={(136.8,76.54)},rotate=53.24]
        [color={rgb,255:red,0; green,0; blue,0},line width=0.75]
        (10.93,-3.29) .. controls (6.95,-1.4) and (3.31,-0.3) .. (0,0)
        .. controls (3.31,0.3) and (6.95,1.4) .. (10.93,3.29);

        \draw [line width=0.75,dash pattern={on 4.5pt off 4.5pt}] (166.49,-0.72) -- (203.52,165.87);
        \draw [shift={(186.3,88.43)},rotate=257.47]
        [color={rgb,255:red,0; green,0; blue,0},line width=0.75]
        (10.93,-3.29) .. controls (6.95,-1.4) and (3.31,-0.3) .. (0,0)
        .. controls (3.31,0.3) and (6.95,1.4) .. (10.93,3.29);

        \draw [line width=0.75,dash pattern={on 4.5pt off 4.5pt}] (12.29,68.52) -- (203.52,165.87);
        \draw [shift={(113.25,119.91)},rotate=206.98]
        [color={rgb,255:red,0; green,0; blue,0},line width=0.75]
        (10.93,-3.29) .. controls (6.95,-1.4) and (3.31,-0.3) .. (0,0)
        .. controls (3.31,0.3) and (6.95,1.4) .. (10.93,3.29);

        % Labels
        \draw (74.03,64.54) node [anchor=north west,inner sep=0.75pt,font=\footnotesize,xscale=0.95,yscale=0.95] {$D_{2j}^{+}$};
        \draw (54.21,124.3) node [anchor=north west,inner sep=0.75pt,font=\footnotesize,xscale=0.95,yscale=0.95] {$D_{2j+1}^{+}$};
        \draw (206.61,66.76) node [anchor=north west,inner sep=0.75pt,font=\footnotesize,xscale=0.95,yscale=0.95] {$i\gamma$};
        \draw (247.69,146.78) node [anchor=north west,inner sep=0.75pt,font=\tiny,xscale=0.95,yscale=0.95] {$\frac{\pi}{2j+1}$};
        \draw (342.11,169.96) node [anchor=north west,inner sep=0.75pt,font=\large,xscale=0.95,yscale=0.95] {$D^{+}$};
        \draw (319.42,125.48) node [anchor=north west,inner sep=0.75pt,font=\footnotesize,xscale=0.95,yscale=0.95] {$D_{1}^{+}$};
        \draw (298.78,64.21) node [anchor=north west,inner sep=0.75pt,font=\footnotesize,xscale=0.95,yscale=0.95] {$D_{2}^{+}$};
        \draw (173.2,7.83) node [anchor=north west,inner sep=0.75pt,font=\footnotesize,xscale=0.95,yscale=0.95] {$D_{j+1}^{+}$};
        \end{tikzpicture}
        }
        \caption{Domains for $j$ odd, $\gamma>0$.}
        \label{j-odd}
    \end{subfigure}
    \qquad
    \begin{subfigure}[t]{0.43\textwidth}
        \centering
        \resizebox{\textwidth}{!}{
        \begin{tikzpicture}[x=0.75pt,y=0.75pt,yscale=-1,xscale=1]
        % uncomment if require: \path (0,270);

        % Shape: Polygon
        \draw [color={rgb,255:red,255; green,255; blue,255},draw opacity=1]
        [fill={rgb,255:red,212; green,237; blue,249},fill opacity=1]
        (-15.71,102.4) -- (-7.29,37.66) -- (-18,-5.5) -- (195.55,180.18) -- (194.51,180.18) -- cycle;

        \draw [color={rgb,255:red,255; green,255; blue,255},draw opacity=1]
        [fill={rgb,255:red,212; green,237; blue,249},fill opacity=1]
        (94.66,2.14) -- (142.59,-1.9) -- (163.24,-5.5) -- (196.6,180.18) -- (195.55,180.18) -- cycle;

        \draw [color={rgb,255:red,255; green,255; blue,255},draw opacity=1]
        [fill={rgb,255:red,212; green,237; blue,249},fill opacity=1]
        (225.94,-2.8) -- (265.71,-1.9) -- (293.89,3.94) -- (197.65,180.18) -- (196.6,180.18) -- cycle;

        \draw [color={rgb,255:red,255; green,255; blue,255},draw opacity=1]
        [fill={rgb,255:red,212; green,237; blue,249},fill opacity=1]
        (384.88,15.63) -- (398,59.24) -- (394.18,105.55) -- (196.6,180.18) -- (195.55,180.18) -- cycle;

        % Axes
        \draw [line width=0.75] (194.95,3.34) -- (195.6,259);
        \draw [shift={(195.6,262)},rotate=269.85]
        [fill={rgb,255:red,0; green,0; blue,0},line width=0.08,draw opacity=0]
        (10.72,-5.15) -- (0,0) -- (10.72,5.15) -- (7.12,0) -- cycle;
        \draw [shift={(194.94,0.34)},rotate=89.85]
        [fill={rgb,255:red,0; green,0; blue,0},line width=0.08,draw opacity=0]
        (10.72,-5.15) -- (0,0) -- (10.72,5.15) -- (7.12,0) -- cycle;

        \draw [line width=0.75] (17.96,180.18) -- (375.24,180.18);
        \draw [shift={(378.24,180.18)},rotate=180]
        [fill={rgb,255:red,0; green,0; blue,0},line width=0.08,draw opacity=0]
        (10.72,-5.15) -- (0,0) -- (10.72,5.15) -- (7.12,0) -- cycle;
        \draw [shift={(14.96,180.18)},rotate=0]
        [fill={rgb,255:red,0; green,0; blue,0},line width=0.08,draw opacity=0]
        (10.72,-5.15) -- (0,0) -- (10.72,5.15) -- (7.12,0) -- cycle;

        % Dashed rays
        \draw [line width=0.75,dash pattern={on 4.5pt off 4.5pt}] (394.18,105.55) -- (196.6,180.18);
        \draw [shift={(301.94,140.39)},rotate=159.31]
        [color={rgb,255:red,0; green,0; blue,0},line width=0.75]
        (10.93,-3.29) .. controls (6.95,-1.4) and (3.31,-0.3) .. (0,0)
        .. controls (3.31,0.3) and (6.95,1.4) .. (10.93,3.29);

        \draw [line width=0.75,dash pattern={on 4.5pt off 4.5pt}] (293.89,3.94) -- (196.6,180.18);
        \draw [shift={(248.63,85.93)},rotate=118.9]
        [color={rgb,255:red,0; green,0; blue,0},line width=0.75]
        (10.93,-3.29) .. controls (6.95,-1.4) and (3.31,-0.3) .. (0,0)
        .. controls (3.31,0.3) and (6.95,1.4) .. (10.93,3.29);

        \draw [line width=0.75,dash pattern={on 4.5pt off 4.5pt}] (164.39,3.04) -- (196.6,180.18);
        \draw [shift={(179.24,84.72)},rotate=79.69]
        [color={rgb,255:red,0; green,0; blue,0},line width=0.75]
        (10.93,-3.29) .. controls (6.95,-1.4) and (3.31,-0.3) .. (0,0)
        .. controls (3.31,0.3) and (6.95,1.4) .. (10.93,3.29);

        \draw [line width=0.75,dash pattern={on 4.5pt off 4.5pt}] (1.02,12.03) -- (196.6,180.18);
        \draw [shift={(93.5,91.54)},rotate=40.69]
        [color={rgb,255:red,0; green,0; blue,0},line width=0.75]
        (10.93,-3.29) .. controls (6.95,-1.4) and (3.31,-0.3) .. (0,0)
        .. controls (3.31,0.3) and (6.95,1.4) .. (10.93,3.29);

        \draw [line width=0.75,dash pattern={on 4.5pt off 4.5pt}] (384.88,15.63) -- (196.6,180.18);
        \draw [shift={(286.22,101.85)},rotate=318.85]
        [color={rgb,255:red,0; green,0; blue,0},line width=0.75]
        (10.93,-3.29) .. controls (6.95,-1.4) and (3.31,-0.3) .. (0,0)
        .. controls (3.31,0.3) and (6.95,1.4) .. (10.93,3.29);

        \draw [line width=0.75,dash pattern={on 4.5pt off 4.5pt}] (225.49,3.04) -- (196.6,180.18);
        \draw [shift={(210.08,97.53)},rotate=279.26]
        [color={rgb,255:red,0; green,0; blue,0},line width=0.75]
        (10.93,-3.29) .. controls (6.95,-1.4) and (3.31,-0.3) .. (0,0)
        .. controls (3.31,0.3) and (6.95,1.4) .. (10.93,3.29);

        \draw [line width=0.75,dash pattern={on 4.5pt off 4.5pt}] (94.66,2.14) -- (196.6,180.18);
        \draw [shift={(148.61,96.37)},rotate=240.2]
        [color={rgb,255:red,0; green,0; blue,0},line width=0.75]
        (10.93,-3.29) .. controls (6.95,-1.4) and (3.31,-0.3) .. (0,0)
        .. controls (3.31,0.3) and (6.95,1.4) .. (10.93,3.29);

        \draw [line width=0.75,dash pattern={on 4.5pt off 4.5pt}] (0.35,107.34) -- (196.6,180.18);
        \draw [shift={(104.1,145.85)},rotate=200.36]
        [color={rgb,255:red,0; green,0; blue,0},line width=0.75]
        (10.93,-3.29) .. controls (6.95,-1.4) and (3.31,-0.3) .. (0,0)
        .. controls (3.31,0.3) and (6.95,1.4) .. (10.93,3.29);

        \draw (248.88,162.19) .. controls (255.76,167.14) and (254.6,176.98) .. (250.41,179.28);

        \draw [line width=0.75,line join=round,line cap=round]
        (287.12,52.5) .. controls (287.12,52.5) and (287.12,52.5) .. (287.12,52.5);
        \draw [line width=0.75,line join=round,line cap=round]
        (290.94,56.99) .. controls (290.94,56.99) and (290.94,56.99) .. (290.94,56.99);
        \draw [line width=0.75,line join=round,line cap=round]
        (294.76,61.49) .. controls (294.76,61.49) and (294.76,61.49) .. (294.76,61.49);
        \draw [line width=0.75,line join=round,line cap=round]
        (298.59,65.98) .. controls (298.59,65.98) and (298.59,65.98) .. (298.59,65.98);

        \draw [line width=0.75,line join=round,line cap=round]
        (91.35,61.49) .. controls (91.35,61.49) and (91.35,61.49) .. (91.35,61.49);
        \draw [line width=0.75,line join=round,line cap=round]
        (94.41,56.99) .. controls (94.41,56.99) and (94.41,56.99) .. (94.41,56.99);
        \draw [line width=0.75,line join=round,line cap=round]
        (98.24,52.5) .. controls (98.24,52.5) and (98.24,52.5) .. (98.24,52.5);
        \draw [line width=0.75,line join=round,line cap=round]
        (101.29,48.9) .. controls (101.29,48.9) and (101.29,48.9) .. (101.29,48.9);

        \draw [line width=3,line join=round,line cap=round]
        (195.35,61.04) .. controls (195.35,61.04) and (195.35,61.04) .. (195.35,61.04);

        % Labels
        \draw (259.76,160.79) node [anchor=north west,inner sep=0.75pt,font=\tiny,xscale=0.95,yscale=0.95] {$\frac{\pi}{2j+1}$};
        \draw (331,193.85) node [anchor=north west,inner sep=0.75pt,font=\large,xscale=0.95,yscale=0.95] {$D^{+}$};
        \draw (297.39,146.26) node [anchor=north west,inner sep=0.75pt,font=\footnotesize,xscale=0.95,yscale=0.95] {$D_{1}^{+}$};
        \draw (331.82,84.31) node [anchor=north west,inner sep=0.75pt,font=\footnotesize,xscale=0.95,yscale=0.95] {$D_{2}^{+}$};
        \draw (168.76,14.17) node [anchor=north west,inner sep=0.75pt,font=\footnotesize,xscale=0.95,yscale=0.95] {$D_{j+1}^{+}$};
        \draw (229.35,18.67) node [anchor=north west,inner sep=0.75pt,font=\footnotesize,xscale=0.95,yscale=0.95] {$D_{j}^{+}$};
        \draw (124.41,15.07) node [anchor=north west,inner sep=0.75pt,font=\footnotesize,xscale=0.95,yscale=0.95] {$D_{j+2}^{+}$};
        \draw (32.71,84.31) node [anchor=north west,inner sep=0.75pt,font=\footnotesize,xscale=0.95,yscale=0.95] {$D_{2j}^{+}$};
        \draw (19.71,148.15) node [anchor=north west,inner sep=0.75pt,font=\footnotesize,xscale=0.95,yscale=0.95] {$D_{2j+1}^{+}$};
        \draw (198.18,46.99) node [anchor=north west,inner sep=0.75pt,font=\footnotesize,xscale=0.95,yscale=0.95] {$i\gamma$};
        \end{tikzpicture}
        }
        \caption{Domains for $j$ even, $\gamma>0$.}
        \label{j-even}
    \end{subfigure}
    \caption{Sector decomposition of the upper half plane for odd and even values of $j$. The shaded sectors constitute $D^+$, and the marked point is the possible Robin pole $k=i\gamma$ for $\gamma>0$.}
    \label{fig:domains}
\end{figure}
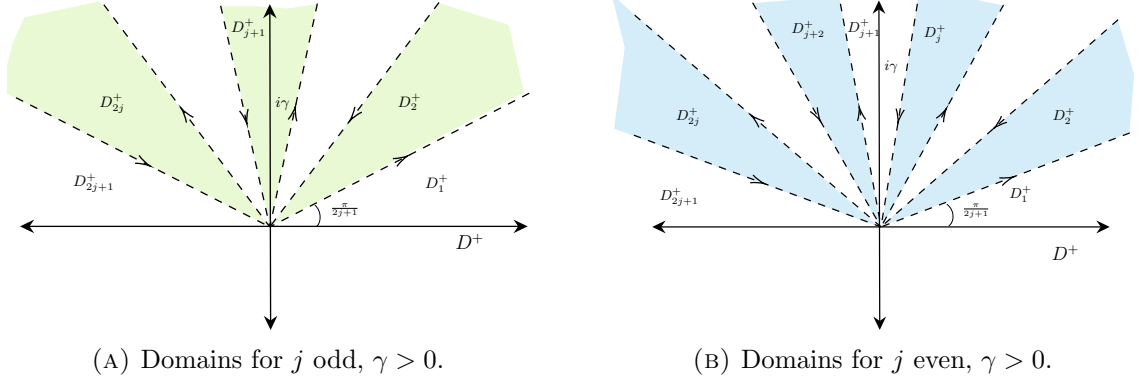

\section{Derivation of the UTM Representation}\label{sec:UTM}
We first derive the UTM representation for sufficiently smooth and decaying data and solutions, so that all boundary traces and differentiations below are justified. We then verify that the resulting integral formula satisfies the linear IBVP with Robin boundary conditions and thereby establish the equivalence of the differential and UTM formulations for smooth data. The resulting solution operator is extended to the low regularity setting by the estimates proved in the subsequent sections. For general background on the unified transform method, we refer to \cite{Fokas1997,DeconinckTrogdonVasan,Fokas2008}.

\subsection{Global relation and contour representation}

Consider
\begin{equation}\label{eq:UTM-1}
\partial_tu+(-1)^{j+1}\partial_x^{2j+1}u=f
\end{equation}
on $(0,T)\times\R^+$, together with
\begin{equation}\label{eq:Robin-hierarchy}
(\partial_x+\gamma)\partial_x^{\ell-1}u(t,0)=\varphi_\ell(t),\qquad 1\le\ell\le j.
\end{equation}
For $\ell\ge0$ and $\lambda\in\C$, define
\begin{equation}\label{eq:g_ell}
\widetilde g_\ell(t,\lambda):=\int_0^t e^{-i\lambda\tau}\partial_x^\ell u(\tau,0)\,d\tau.
\end{equation}
For each fixed $t\in[0,T]$, the function $\widetilde g_\ell(t,\cdot)$ is entire. Moreover, \eqref{eq:Robin-hierarchy} gives
\begin{equation}\label{eq:boundary-recursion}
\widetilde g_\ell(t,\lambda)+\gamma\widetilde g_{\ell-1}(t,\lambda)=\widetilde\varphi_\ell(t,\lambda),\qquad 1\le\ell\le j.
\end{equation}
Iterating \eqref{eq:boundary-recursion}, we obtain
\begin{equation}\label{eq:g-recursion}
\widetilde g_q(t,\lambda)=(-\gamma)^q\widetilde g_0(t,\lambda)+\sum_{\ell=1}^q(-\gamma)^{q-\ell}\widetilde\varphi_\ell(t,\lambda),\qquad 1\le q\le j.
\end{equation}

Taking the Fourier transform on the half line of \eqref{eq:UTM-1}, we obtain
\[\partial_t\widehat u(t,k)+(-1)^{j+1}\int_0^\infty e^{-ikx}\partial_x^{2j+1}u(t,x)\,dx=\widehat f(t,k).\]
Using the decay of $u$ at infinity, integration by parts gives
\[\int_0^\infty e^{-ikx}\partial_x^\ell u(t,x)\,dx=-\partial_x^{\ell-1}u(t,0)+ik\int_0^\infty e^{-ikx}\partial_x^{\ell-1}u(t,x)\,dx,\qquad 1\le\ell\le2j+1.\]
Applying this identity first with $\ell=2j+1$ and then with $\ell=2j$, we obtain
\[\begin{aligned}
\int_0^\infty e^{-ikx}\partial_x^{2j+1}u(t,x)\,dx=&~{}-\partial_x^{2j}u(t,0)+ik\int_0^\infty e^{-ikx}\partial_x^{2j}u(t,x)\,dx\\
=&~{}-\partial_x^{2j}u(t,0)-ik\partial_x^{2j-1}u(t,0)+(ik)^2\int_0^\infty e^{-ikx}\partial_x^{2j-1}u(t,x)\,dx.
\end{aligned}\]
More generally, for $1\le r\le2j+1$, after $r$ iterations,
\[\int_0^\infty e^{-ikx}\partial_x^{2j+1}u(t,x)\,dx=-\sum_{m=0}^{r-1}(ik)^m\partial_x^{2j-m}u(t,0)+(ik)^r\int_0^\infty e^{-ikx}\partial_x^{2j+1-r}u(t,x)\,dx.\]
Taking $r=2j+1$, we obtain
\[\int_0^\infty e^{-ikx}\partial_x^{2j+1}u(t,x)\,dx=-\sum_{m=0}^{2j}(ik)^m\partial_x^{2j-m}u(t,0)+(ik)^{2j+1}\widehat u(t,k).\]
Therefore,
\begin{equation}\label{e1}
\partial_t\widehat u(t,k)=(-1)^{j+1}\sum_{m=0}^{2j}(ik)^m\partial_x^{2j-m}u(t,0)+ik^{2j+1}\widehat u(t,k)+\widehat f(t,k).
\end{equation}
Since
\[\partial_t\left(e^{-ik^{2j+1}t}\widehat u(t,k)\right)=e^{-ik^{2j+1}t}\left(\partial_t\widehat u(t,k)-ik^{2j+1}\widehat u(t,k)\right),\]
multiplying \eqref{e1} by $e^{-ik^{2j+1}t}$ and integrating in time from $0$ to $t$, we obtain
\begin{equation}\label{eq:raw-global-relation}
e^{-ik^{2j+1}t}\widehat u(t,k)-\widehat u_0(k)=(-1)^{j+1}\sum_{m=0}^{2j}(ik)^m\widetilde g_{2j-m}(t,k^{2j+1})+F(t,k),
\end{equation}
where $F$ is given by \eqref{eq:F}. Set
\begin{equation}\label{eq:g-boundary}
\widetilde g(t,k):=(-1)^{j+1}\sum_{m=0}^{2j}(ik)^m\widetilde g_{2j-m}(t,k^{2j+1}).
\end{equation}
Then \eqref{eq:raw-global-relation} becomes
\begin{equation}\label{eq:global-relation-preliminary}
e^{-ik^{2j+1}t}\widehat u(t,k)-\widehat u_0(k)=\widetilde g(t,k)+F(t,k).
\end{equation}
Applying the inverse Fourier transform gives
\begin{equation}\label{formula1}
u(t,x)=\frac{1}{2\pi}\int_{\R}e^{ikx+ik^{2j+1}t}\left(\widehat u_0(k)+\widetilde g(t,k)+F(t,k)\right)\,dk.
\end{equation}

We now turn to the boundary integral in \eqref{formula1}. Its contour can be deformed according to the sign of $\im (k^{2j+1})$ in the sectors $D_\kappa^+$. More precisely, for $0\le p\le j$,
\[\im (k^{2j+1})>0,\qquad k\in D_{2p+1}^+,\]
whereas, for $1\le p\le j$,
\begin{equation}\label{eq:negative imaginary k^2j+1}
\im (k^{2j+1})<0,\qquad k\in D_{2p}^+.
\end{equation}
The sector decomposition used here is illustrated in Figure \ref{fig:domains}.

By \eqref{eq:g_ell} and \eqref{eq:g-boundary},
\[\widetilde g(t,k)=\int_0^t e^{-ik^{2j+1}\tau}\mathcal B(\tau,k)\,d\tau,\]
where
\[\mathcal B(\tau,k):=(-1)^{j+1}\sum_{m=0}^{2j}(ik)^m\partial_x^{2j-m}u(\tau,0).\]
Consequently,
\begin{equation}\label{eq:g-integral-representation}
e^{ikx+ik^{2j+1}t}\widetilde g(t,k)=e^{ikx}\int_0^t e^{ik^{2j+1}(t-\tau)}\mathcal B(\tau,k)\,d\tau.
\end{equation}
For $k\in D_{2p+1}^+$, we have
\[\left|e^{ik^{2j+1}(t-\tau)}\right|\le1,\qquad 0\le\tau\le t.\]
Moreover, integration by parts in $\tau$ gives
\begin{equation}\label{eq:arc-decay}
\begin{aligned}
\int_0^t e^{ik^{2j+1}(t-\tau)}\mathcal B(\tau,k)\,d\tau=&~{}\frac{e^{ik^{2j+1}t}\mathcal B(0,k)-\mathcal B(t,k)}{ik^{2j+1}}\\
&~{}+\frac{1}{ik^{2j+1}}\int_0^t e^{ik^{2j+1}(t-\tau)}\partial_\tau\mathcal B(\tau,k)\,d\tau.
\end{aligned}
\end{equation}
By the smoothness assumption on $u$, the coefficients of $\mathcal B(\tau,k)$ and $\partial_\tau\mathcal B(\tau,k)$ are uniformly bounded for $\tau\in[0,T]$. Since both are polynomials in $k$ of degree at most $2j$, we have
\[|\mathcal B(\tau,k)|+|\partial_\tau\mathcal B(\tau,k)|\lesssim_T|k|^{2j}\]
for $|k|\ge1$ and $\tau\in[0,T]$. Hence \eqref{eq:arc-decay} gives
\begin{equation}\label{eq:boundary decay}
\left|\int_0^t e^{ik^{2j+1}(t-\tau)}\mathcal B(\tau,k)\,d\tau\right|\lesssim_T\frac{1}{|k|}
\end{equation}
for large $|k|$ in each closed odd sector.

We use the standard sectorial form of Jordan's lemma, see for instance \cite{DeconinckTrogdonVasan}. For $0\le p\le j$ and $R>0$, let
\[C_{p,R}:=\{k\in\overline{D_{2p+1}^+}:|k|=R\}.\]
By \eqref{eq:g-integral-representation} and \eqref{eq:boundary decay},
\[\lim_{R\to\infty}\int_{C_{p,R}}e^{ikx+ik^{2j+1}t}\widetilde g(t,k)\,dk=0,\qquad x>0.\]
Since $\widetilde g(t,\cdot)$ is entire, Cauchy's theorem in each truncated odd sector and passage to the limit $R\to\infty$ give
\begin{equation}\label{fokas-contour-deformation}
\int_{\R}e^{ikx+ik^{2j+1}t}\widetilde g(t,k)\,dk
=\int_{\partial D^+}e^{ikx+ik^{2j+1}t}\widetilde g(t,k)\,dk,\qquad x>0.
\end{equation}

We next replace $\widetilde g(t,k)$ by $\widetilde g(T,k)$ in the boundary integral. By \eqref{eq:g_ell} and \eqref{eq:g-boundary},
\begin{equation}\label{eq:g-t-T}
e^{ik^{2j+1}t}\left(\widetilde g(t,k)-\widetilde g(T,k)\right)=-\int_t^T e^{ik^{2j+1}(t-\tau)}\mathcal B(\tau,k)\,d\tau.
\end{equation}
For $k\in D_{2p}^+$, we have $\im (k^{2j+1})<0$ and $t-\tau\le0$, and hence
\[\left|e^{ik^{2j+1}(t-\tau)}\right|\le1,\qquad t\le\tau\le T.\]
Integration by parts gives
\[\begin{aligned}
-\int_t^T e^{ik^{2j+1}(t-\tau)}\mathcal B(\tau,k)\,d\tau=&~{}\frac{e^{ik^{2j+1}(t-T)}\mathcal B(T,k)-\mathcal B(t,k)}{ik^{2j+1}}\\
&~{}-\frac{1}{ik^{2j+1}}\int_t^T e^{ik^{2j+1}(t-\tau)}\partial_\tau\mathcal B(\tau,k)\,d\tau.
\end{aligned}\]
Thus the right-hand side of \eqref{eq:g-t-T} is $O(|k|^{-1})$ for large $|k|$ in each closed even sector. Applying the same large arc argument in each $D_{2p}^+$, $1\le p\le j$, gives
\begin{equation}\label{eq:g-t-T-vanishing}
\int_{\partial D^+}e^{ikx+ik^{2j+1}t}\left(\widetilde g(t,k)-\widetilde g(T,k)\right)\,dk=0.
\end{equation}
Combining \eqref{formula1}, \eqref{fokas-contour-deformation}, and \eqref{eq:g-t-T-vanishing}, we obtain
\begin{equation}\label{formula2}
u(t,x)=\frac{1}{2\pi}\int_{\R}e^{ikx+ik^{2j+1}t}\left(\widehat u_0(k)+F(t,k)\right)\,dk+\frac{1}{2\pi}\int_{\partial D^+}e^{ikx+ik^{2j+1}t}\widetilde g(T,k)\,dk.
\end{equation}

To determine $\widetilde g(T,k)$ in \eqref{formula2}, we rewrite $\widetilde g(t,k)$ for arbitrary $0\le t\le T$. For $0\le\ell\le j$, set
\begin{equation}\label{eq:b-ell}
\mathfrak b_\ell(k):=\sum_{q=\ell}^j(ik)^{2j-q}(-\gamma)^{q-\ell}.
\end{equation}
Using \eqref{eq:g-recursion} in \eqref{eq:g-boundary}, we obtain
\begin{equation}\label{e4}
\begin{aligned}
\widetilde g(t,k)=&~{}(-1)^{j+1}\sum_{m=0}^{j-1}(ik)^m\widetilde g_{2j-m}(t,k^{2j+1})\\
&~{}+(-1)^{j+1}\mathfrak b_0(k)\widetilde g_0(t,k^{2j+1})\\
&~{}+(-1)^{j+1}\sum_{\ell=1}^j\mathfrak b_\ell(k)\widetilde\varphi_\ell(t,k^{2j+1}).
\end{aligned}
\end{equation}
Consequently, the global relation \eqref{eq:global-relation-preliminary} can be written as
\begin{equation}\label{e3}
\begin{aligned}
e^{-ik^{2j+1}t}\widehat u(t,k)-\widehat u_0(k)=&~{}(-1)^{j+1}\sum_{m=0}^{j-1}(ik)^m\widetilde g_{2j-m}(t,k^{2j+1})\\
&~{}+(-1)^{j+1}\mathfrak b_0(k)\widetilde g_0(t,k^{2j+1})\\
&~{}+(-1)^{j+1}\sum_{\ell=1}^j\mathfrak b_\ell(k)\widetilde\varphi_\ell(t,k^{2j+1})+F(t,k).
\end{aligned}
\end{equation}

\subsection{Explicit resolution of the sectorwise Robin boundary systems}\label{sec:resolution}
We now eliminate the remaining unknown boundary terms in \eqref{e4}. Let
\[\omega:=e^{i2\pi/(2j+1)}.\]
For $1\le p\le j$ and $1\le n\le j+1$, define
\begin{equation}\label{eq:alpha-pn}
\alpha_{p,n}:=\omega^{j-p+n}.
\end{equation}

The principal claim in this subsection is as follows.

\begin{proposition}\label{prop:boundary-elimination}
Let $1\le p\le j$, $k\in D_{2p}^+$, and $k\neq i\gamma$. Then, for $0\le t\le T$,
\begin{equation}\label{formula3}
\begin{aligned}\widetilde g(t,k)=&~{}\sum_{n=1}^{j+1}e^{-ik^{2j+1}t}\widehat u(t,\alpha_{p,n}k)\mathcal P_{p,n}(k)\\
&~{}-\sum_{n=1}^{j+1}\left(\widehat u_0(\alpha_{p,n}k)+F(t,\alpha_{p,n}k)\right)\mathcal P_{p,n}(k)\\
&~{}+\sum_{\ell=1}^jc_{p,\ell}\frac{k^{2j+1-\ell}}{k-i\gamma}\widetilde\varphi_\ell(t,k^{2j+1}),
\end{aligned}
\end{equation}
where the coefficients $\mathcal P_{p,n}$ and $c_{p,\ell}$ will be defined below in \eqref{eq:P-pn} and \eqref{eq:c-p-ell}, respectively.
\end{proposition}

We first record several identities that will be used in the proof. Since $\alpha_{p,n}^{2j+1}=1$, we have
\[(\alpha_{p,n}k)^{2j+1}=k^{2j+1}.\]
Moreover, if $k\in D_{2p}^+$, then
\[\frac{(2p-1)\pi}{2j+1}<\arg k<\frac{2p\pi}{2j+1}.\]
Since $\alpha_{p,n}=e^{i2\pi(j-p+n)/(2j+1)}$, we have
\[\pi+\frac{2(n-1)\pi}{2j+1}<\arg(\alpha_{p,n}k)<\pi+\frac{(2n-1)\pi}{2j+1},\qquad 1\le n\le j+1.\]
Hence
\begin{equation}\label{eq:imaginary alpha k}
\im (\alpha_{p,n}k)<0,\qquad 1\le n\le j+1.
\end{equation}

Set
\begin{equation}\label{eq:unknown-vector}
X(t,k):=
\begin{bmatrix}
\widetilde g_{2j}(t,k^{2j+1})\\
ik\widetilde g_{2j-1}(t,k^{2j+1})\\
\vdots\\
(ik)^{j-1}\widetilde g_{j+1}(t,k^{2j+1})\\
(ik)^j\widetilde g_0(t,k^{2j+1})
\end{bmatrix}
=:
\begin{bmatrix}
X_1(t,k)\\
X_2(t,k)\\
\vdots\\
X_j(t,k)\\
X_{j+1}(t,k)
\end{bmatrix}.
\end{equation}
When \eqref{e3} is evaluated at $\alpha_{p,n}k$, the unknown boundary terms satisfy
\[(i\alpha_{p,n}k)^m\widetilde g_{2j-m}(t,k^{2j+1})=\alpha_{p,n}^mX_{m+1}(t,k),\qquad 0\le m\le j-1,\]
while
\[\mathfrak b_0(\alpha_{p,n}k)\widetilde g_0(t,k^{2j+1})=\frac{\mathfrak b_0(\alpha_{p,n}k)}{(ik)^j}X_{j+1}(t,k).\]
For $z\in\C$ and $k\neq0$, define
\begin{equation}\label{eq:a0}
\mathfrak a_0(z,k):=\frac{\mathfrak b_0(zk)}{(ik)^j}=\sum_{q=0}^ji^{j-q}k^{j-q}(-\gamma)^qz^{2j-q}.
\end{equation}
Next, for $1\le n\le j+1$, set
\begin{equation}\label{eq:I-pn}
\begin{aligned}
\mathcal I_{p,n}(t,k):=&~{}(-1)^j\left(-e^{-ik^{2j+1}t}\widehat u(t,\alpha_{p,n}k)+\widehat u_0(\alpha_{p,n}k)+F(t,\alpha_{p,n}k)\right)\\
&~{}-\sum_{\ell=1}^j\mathfrak b_\ell(\alpha_{p,n}k)\widetilde\varphi_\ell(t,k^{2j+1}).
\end{aligned}
\end{equation}
Evaluating \eqref{e3} at $\alpha_{p,n}k$ and using $\alpha_{p,n}^{2j+1}=1$, we obtain
\begin{equation}\label{eq:sector-system}
A_p(k)X(t,k)=
\begin{bmatrix}
\mathcal I_{p,1}(t,k)\\
\vdots\\
\mathcal I_{p,j+1}(t,k)
\end{bmatrix},
\end{equation}
where
\begin{equation}\label{eq:A-p}
A_p(k):=
\begin{bmatrix}
1&\alpha_{p,1}&\cdots&\alpha_{p,1}^{j-1}&\mathfrak a_0(\alpha_{p,1},k)\\
1&\alpha_{p,2}&\cdots&\alpha_{p,2}^{j-1}&\mathfrak a_0(\alpha_{p,2},k)\\
\vdots&\vdots&&\vdots&\vdots\\
1&\alpha_{p,j+1}&\cdots&\alpha_{p,j+1}^{j-1}&\mathfrak a_0(\alpha_{p,j+1},k)
\end{bmatrix}.
\end{equation}

The system \eqref{eq:sector-system} determines the unknown boundary vector $X(t,k)$. For the representation of $\widetilde g(t,k)$, however, it suffices to determine a particular linear combination of its components. Indeed, by \eqref{e4} and the definition of $X(t,k)$,
\[\widetilde g(t,k)=(-1)^{j+1}\left(\sum_{m=0}^{j-1}X_{m+1}(t,k)+\mathfrak a_0(1,k)X_{j+1}(t,k)+\sum_{\ell=1}^j\mathfrak b_\ell(k)\widetilde\varphi_\ell(t,k^{2j+1})\right).\]
We therefore first express
\[\sum_{m=0}^{j-1}X_{m+1}(t,k)+\mathfrak a_0(1,k)X_{j+1}(t,k)\]
in terms of the quantities $\mathcal I_{p,n}$ \eqref{eq:I-pn} by Lagrange interpolation. The resulting coefficients of the given boundary data will then be computed separately.

To this end, we apply Lagrange interpolation in the complex variable $z$ at the points $\alpha_{p,1},\ldots,\alpha_{p,j+1}$. Since
\[j-p+n\in\{1,\ldots,2j\},\qquad 1\le n\le j+1,\]
these points are pairwise distinct and none of them equals $1$. For $1\le n\le j+1$, define the Lagrange basis polynomial
\begin{equation}\label{eq:L-pn}
L_{p,n}(z):=\prod_{\substack{1\le m\le j+1\\m\neq n}}\frac{z-\alpha_{p,m}}{\alpha_{p,n}-\alpha_{p,m}}.
\end{equation}
In particular,
\begin{equation}\label{eq:lagrange delta}
L_{p,n}(\alpha_{p,m})=\delta_{nm},\qquad 1\le n,m\le j+1.
\end{equation}
For convenience, set
\begin{equation}\label{eq:Q-p}
Q_p(z):=\prod_{m=1}^{j+1}(z-\alpha_{p,m}).
\end{equation}
Since $\alpha_{p,m}\neq1$ for every $1\le m\le j+1$, we have
\[Q_p(1)\neq0.\]
Then
\begin{equation}\label{eq:L-pn-Q}
L_{p,n}(z)=\frac{Q_p(z)}{(z-\alpha_{p,n})Q_p'(\alpha_{p,n})}.
\end{equation}

\begin{lemma}\label{lem:sector-interpolation}
Let $1\le p\le j$ and $k\in D_{2p}^+$ with $k\neq i\gamma$. Then the vector $X(t,k)$ defined in \eqref{eq:unknown-vector} satisfies
\begin{equation}\label{eq:sector-interpolation}
\sum_{m=0}^{j-1}X_{m+1}(t,k)+\mathfrak a_0(1,k)X_{j+1}(t,k)
=\sum_{n=1}^{j+1}\left(L_{p,n}(1)-\frac{Q_p(1)}{Q_p'(\alpha_{p,n})}\frac{k}{k-i\gamma}\right)\mathcal I_{p,n}(t,k).
\end{equation}
\end{lemma}

\begin{proof}
For $z \in \C$, let
\[R_p(z,k):=\sum_{m=1}^{j+1}\mathfrak a_0(\alpha_{p,m},k)L_{p,m}(z)\]
and
\begin{equation}\label{eq:J_p}
\mathcal J_p(z,t,k):=\sum_{m=1}^{j+1}\mathcal I_{p,m}(t,k)L_{p,m}(z).
\end{equation}
Since each $L_{p,m}$ is a polynomial of degree $j$, both $R_p(\cdot,k)$ and $\mathcal J_p(\cdot,t,k)$ are polynomials of degree at most $j$.
Note by \eqref{eq:lagrange delta} that
\begin{equation}\label{eq:at alpha_pn}
R_p(\alpha_{p,n},k)=\mathfrak a_0(\alpha_{p,n},k),\qquad \mathcal J_p(\alpha_{p,n},t,k)=\mathcal I_{p,n}(t,k),\qquad 1\le n\le j+1.
\end{equation}
By \eqref{eq:sector-system}, for $1\le n\le j+1$,
\[\mathcal I_{p,n}(t,k)=\sum_{m=0}^{j-1}\alpha_{p,n}^mX_{m+1}(t,k)+\mathfrak a_0(\alpha_{p,n},k)X_{j+1}(t,k).\]
Together with \eqref{eq:at alpha_pn}, we obtain
\[\mathcal J_p(\alpha_{p,n},t,k)=\sum_{m=0}^{j-1}X_{m+1}(t,k)\alpha_{p,n}^m+X_{j+1}(t,k)R_p(\alpha_{p,n},k),\qquad 1\le n\le j+1.\]
Therefore, the two polynomials
\[\mathcal J_p(z,t,k)\]
and
\[\sum_{m=0}^{j-1}X_{m+1}(t,k)z^m+X_{j+1}(t,k)R_p(z,k)\]
agree at the $j+1$ distinct points $\alpha_{p,1},\ldots,\alpha_{p,j+1}$. Since both polynomials have degree at most $j$, their difference also has degree at most $j$. On the other hand, the difference vanishes at the $j+1$ distinct points $\alpha_{p,1},\ldots,\alpha_{p,j+1}$. Since a nonzero polynomial of degree at most $j$ has at most $j$ distinct zeros, the difference must vanish identically. Hence
\begin{equation}\label{eq:JXR-local}
\mathcal J_p(z,t,k)=\sum_{m=0}^{j-1}X_{m+1}(t,k)z^m+X_{j+1}(t,k)R_p(z,k).
\end{equation}

We next compare the coefficients of $z^j$ in \eqref{eq:JXR-local}. Denote the coefficient of $z^j$ in $R_p(z,k)$ by $\lambda_p(k)$. Since the first sum on the right-hand side of \eqref{eq:JXR-local} has degree at most $j-1$, the coefficient of $z^j$ on the right-hand side is
\[\lambda_p(k)X_{j+1}(t,k).\]
On the other hand, since $Q_p$ is monic of degree $j+1$, \eqref{eq:L-pn-Q} shows that the coefficient of $z^j$ in $L_{p,n}(z)$ is
\[\frac{1}{Q_p'(\alpha_{p,n})}.\]
Thus, by \eqref{eq:J_p}, the coefficient of $z^j$ in $\mathcal J_p(z,t,k)$ is
\[\sum_{n=1}^{j+1}\frac{\mathcal I_{p,n}(t,k)}{Q_p'(\alpha_{p,n})}.\]
Comparing the coefficients of $z^j$ on both sides of \eqref{eq:JXR-local}, we obtain
\begin{equation}\label{eq:lambda-X}
\lambda_p(k)X_{j+1}(t,k)=\sum_{n=1}^{j+1}\frac{\mathcal I_{p,n}(t,k)}{Q_p'(\alpha_{p,n})}.
\end{equation}

Evaluating \eqref{eq:JXR-local} at $z=1$, we obtain
\[\mathcal J_p(1,t,k)=\sum_{m=0}^{j-1}X_{m+1}(t,k)+X_{j+1}(t,k)R_p(1,k).\]
Therefore,
\begin{equation}\label{eq:boundary-combination}
\sum_{m=0}^{j-1}X_{m+1}(t,k)+\mathfrak a_0(1,k)X_{j+1}(t,k)=\mathcal J_p(1,t,k)+\left(\mathfrak a_0(1,k)-R_p(1,k)\right)X_{j+1}(t,k).
\end{equation}
It remains to determine the second term on the right-hand side of \eqref{eq:boundary-combination}. Set $\zeta:=ik$. By \eqref{eq:a0},
\[\mathfrak a_0(z,k)=\sum_{q=0}^j\zeta^{j-q}(-\gamma)^qz^{2j-q}.\]
Multiplying by $\zeta z+\gamma$, we obtain
\[(\zeta z+\gamma)\mathfrak a_0(z,k)=\sum_{q=0}^j\left(\zeta^{j-q+1}(-\gamma)^qz^{2j-q+1}-\zeta^{j-q}(-\gamma)^{q+1}z^{2j-q}\right).\]
Since the sum on the right-hand side telescopes, we obtain
\begin{equation}\label{eq:a0-telescoping}
(\zeta z+\gamma)\mathfrak a_0(z,k)=\zeta^{j+1}z^{2j+1}-(-\gamma)^{j+1}z^j.
\end{equation}

By $\alpha_{p,n}^{2j+1}=1$ and \eqref{eq:at alpha_pn}, evaluating \eqref{eq:a0-telescoping} at $z=\alpha_{p,n}$ gives
\[(\zeta\alpha_{p,n}+\gamma)R_p(\alpha_{p,n},k)=\zeta^{j+1}-(-\gamma)^{j+1}\alpha_{p,n}^j,\qquad 1\le n\le j+1.\]
It follows that the polynomial
\begin{equation}\label{eq:j+1 polynomial}
(\zeta z+\gamma)R_p(z,k)-\zeta^{j+1}+(-\gamma)^{j+1}z^j
\end{equation}
vanishes at the $j+1$ distinct roots $\alpha_{p,1},\ldots,\alpha_{p,j+1}$. Note that \eqref{eq:j+1 polynomial} has degree at most $j+1$ as a polynomial in $z$. Moreover, since $\lambda_p(k)$ is the coefficient of $z^j$ in $R_p(z,k)$, the coefficient of $z^{j+1}$ in \eqref{eq:j+1 polynomial} is $\zeta\lambda_p(k)$. Thus we obtain
\begin{equation}\label{eq:R-identity}
(\zeta z+\gamma)R_p(z,k)-\zeta^{j+1}+(-\gamma)^{j+1}z^j=\zeta\lambda_p(k)\prod_{m=1}^{j+1}(z-\alpha_{p,m})=\zeta\lambda_p(k)Q_p(z).
\end{equation}
Taking the difference of \eqref{eq:a0-telescoping} and \eqref{eq:R-identity} at $z=1$, we obtain
\[(\zeta+\gamma)\left(\mathfrak a_0(1,k)-R_p(1,k)\right)=-\zeta\lambda_p(k)Q_p(1).\]
Using $\zeta=ik$ and \eqref{eq:lambda-X}, it follows that
\begin{equation}\label{eq:a0-R-X}
\left(\mathfrak a_0(1,k)-R_p(1,k)\right)X_{j+1}(t,k)=-Q_p(1)\frac{k}{k-i\gamma}\sum_{n=1}^{j+1}\frac{\mathcal I_{p,n}(t,k)}{Q_p'(\alpha_{p,n})}.
\end{equation}
By \eqref{eq:J_p},
\[\mathcal J_p(1,t,k)=\sum_{n=1}^{j+1}\mathcal I_{p,n}(t,k)L_{p,n}(1).\]
Substituting this identity and \eqref{eq:a0-R-X} into \eqref{eq:boundary-combination}, we obtain \eqref{eq:sector-interpolation}.
\end{proof}

The coefficients appearing on the right hand side of \eqref{eq:sector-interpolation} will be used repeatedly. For convenience, define
\begin{equation}\label{eq:P-pn}
\mathcal P_{p,n}(k):=L_{p,n}(1)-\frac{Q_p(1)}{Q_p'(\alpha_{p,n})}\frac{k}{k-i\gamma},\qquad 1\le n\le j+1.
\end{equation}
Then \eqref{eq:sector-interpolation} becomes
\begin{equation}\label{eq:sector-interpolation-P}
\sum_{m=0}^{j-1}X_{m+1}(t,k)+\mathfrak a_0(1,k)X_{j+1}(t,k)=\sum_{n=1}^{j+1}\mathcal P_{p,n}(k)\mathcal I_{p,n}(t,k).
\end{equation}

We next determine the coefficients of the given boundary data.

\begin{lemma}\label{lem:boundary-coefficients}
For $1\le p,\ell\le j$ and $k\neq i\gamma$,
\begin{equation}\label{eq:boundary-coefficients}
(-1)^{j+1}\left(\mathfrak b_\ell(k)-\sum_{n=1}^{j+1}\mathcal P_{p,n}(k)\mathfrak b_\ell(\alpha_{p,n}k)\right)
=c_{p,\ell}\frac{k^{2j+1-\ell}}{k-i\gamma},
\end{equation}
where $\mathfrak b_\ell$ and $\mathcal P_{p,n}$ are defined as in \eqref{eq:b-ell} and \eqref{eq:P-pn}, respectively, and $c_{p,\ell}$ depends only on $j$, $p$, and $\ell$.
\end{lemma}

\begin{proof}
Fix $1\le\ell\le j$. For $z \in \C$, set
\[B(z):=\mathfrak b_\ell(zk).\]
By the definition of $\mathfrak b_\ell$,
\[B(z)=\sum_{q=\ell}^j(ikz)^{2j-q}(-\gamma)^{q-\ell}.\]
Multiplying by $ikz+\gamma$, we obtain
\[(ikz+\gamma)B(z)=\sum_{q=\ell}^j\left((ikz)^{2j-q+1}(-\gamma)^{q-\ell}-(ikz)^{2j-q}(-\gamma)^{q-\ell+1}\right).\]
Since the sum on the right hand side telescopes, we have
\begin{equation}\label{eq:b-telescoping}
(ikz+\gamma)B(z)=(ik)^{2j+1-\ell}z^{2j+1-\ell}-(-\gamma)^{j-\ell+1}(ik)^jz^j.
\end{equation}

We divide $B(z)$ by $Q_p(z)$ and write
\begin{equation}\label{eq:b-division}
B(z)=Q_p(z)d(z)+r(z),\qquad \deg r\le j.
\end{equation}
Since $Q_p(\alpha_{p,n})=0$, \eqref{eq:b-division} gives
\[r(\alpha_{p,n})=B(\alpha_{p,n}),\qquad 1\le n\le j+1.\]
Consider the polynomial
\[\sum_{n=1}^{j+1}B(\alpha_{p,n})L_{p,n}(z).\]
By \eqref{eq:lagrange delta}, a direct computation shows that this polynomial has the same values as $r(z)$ at the $j+1$ distinct points $\alpha_{p,1},\ldots,\alpha_{p,j+1}$. Since both polynomials have degree at most $j$, they coincide. Hence
\begin{equation}\label{eq:r-interpolation}
r(z)=\sum_{n=1}^{j+1}B(\alpha_{p,n})L_{p,n}(z).
\end{equation}
Denote the coefficient of $z^j$ in $r(z)$ by $\mu$. As observed in the proof of Lemma \ref{lem:sector-interpolation}, the coefficient of $z^j$ in $L_{p,n}(z)$ is
\[\frac{1}{Q_p'(\alpha_{p,n})}.\]
Hence \eqref{eq:r-interpolation} gives
\begin{equation}\label{eq:mu}
\mu=\sum_{n=1}^{j+1}\frac{B(\alpha_{p,n})}{Q_p'(\alpha_{p,n})}.
\end{equation}
By \eqref{eq:P-pn},
\[\sum_{n=1}^{j+1}\mathcal P_{p,n}(k)B(\alpha_{p,n})=\sum_{n=1}^{j+1}B(\alpha_{p,n})L_{p,n}(1)-Q_p(1)\frac{k}{k-i\gamma}\sum_{n=1}^{j+1}\frac{B(\alpha_{p,n})}{Q_p'(\alpha_{p,n})}.\]
Using \eqref{eq:r-interpolation} and \eqref{eq:mu}, we obtain
\[\sum_{n=1}^{j+1}\mathcal P_{p,n}(k)B(\alpha_{p,n})=r(1)-Q_p(1)\frac{k}{k-i\gamma}\mu.\]
On the other hand, setting $z=1$ in \eqref{eq:b-division} gives
\[B(1)=Q_p(1)d(1)+r(1).\]
Therefore,
\begin{equation}\label{eq:boundary-difference}
B(1)-\sum_{n=1}^{j+1}\mathcal P_{p,n}(k)B(\alpha_{p,n})=Q_p(1)\left(d(1)+\frac{k}{k-i\gamma}\mu\right).
\end{equation}
We next compute
\[d(1)+\frac{k}{k-i\gamma}\mu.\] 
The term $z^{2j+1-\ell}$ appearing on the right hand side of \eqref{eq:b-telescoping} has degree at least $j+1$. We write its Euclidean division by $Q_p(z)$ as
\begin{equation}\label{eq:S-p-ell}
z^{2j+1-\ell}=Q_p(z)S_{p,\ell}(z)+E_{p,\ell}(z),\qquad \deg E_{p,\ell}\le j.
\end{equation}
Since $Q_p$ is monic of degree $j+1$, we have
\[\deg S_{p,\ell}=j-\ell.\]
Recall that $\mu$ denotes the coefficient of $z^j$ in $r(z)$. If $\deg r=j$, then the coefficient of $z^{j+1}$ in $(ikz+\gamma)r(z)$ is $ik\mu$. If $\deg r<j$, then $\mu=0$ and $(ikz+\gamma)r(z)$ has degree at most $j$. Hence, in either case, we can write
\[(ikz+\gamma)r(z)=ik\mu Q_p(z)+\widetilde r(z),\qquad \deg\widetilde r\le j.\]
Multiplying \eqref{eq:b-division} by $ikz+\gamma$, we obtain
\[(ikz+\gamma)B(z)=Q_p(z)\left((ikz+\gamma)d(z)+ik\mu\right)+\widetilde r(z).\]
On the other hand, using \eqref{eq:b-telescoping} and \eqref{eq:S-p-ell}, we have
\[(ikz+\gamma)B(z)=Q_p(z)(ik)^{2j+1-\ell}S_{p,\ell}(z)+(ik)^{2j+1-\ell}E_{p,\ell}(z)-(-\gamma)^{j-\ell+1}(ik)^jz^j.\]
Since
\[\deg\widetilde r\le j\]
and
\[\deg\left((ik)^{2j+1-\ell}E_{p,\ell}(z)-(-\gamma)^{j-\ell+1}(ik)^jz^j\right)\le j,\]
the uniqueness of the Euclidean division by $Q_p(z)$ gives
\[(ikz+\gamma)d(z)+ik\mu=(ik)^{2j+1-\ell}S_{p,\ell}(z).\]
Setting $z=1$, we get
\[(ik+\gamma)d(1)+ik\mu=(ik)^{2j+1-\ell}S_{p,\ell}(1).\]
Since $ik+\gamma=i(k-i\gamma)$, it follows that
\begin{equation}\label{eq:d-mu}
d(1)+\frac{k}{k-i\gamma}\mu=i^{2j-\ell}\frac{k^{2j+1-\ell}}{k-i\gamma}S_{p,\ell}(1).
\end{equation}
Substituting \eqref{eq:d-mu} into \eqref{eq:boundary-difference}, we obtain
\[\mathfrak b_\ell(k)-\sum_{n=1}^{j+1}\mathcal P_{p,n}(k)\mathfrak b_\ell(\alpha_{p,n}k)=i^{2j-\ell}Q_p(1)S_{p,\ell}(1)\frac{k^{2j+1-\ell}}{k-i\gamma}.\]
Therefore, setting
\begin{equation}\label{eq:c-p-ell}
c_{p,\ell}:=(-1)^{j+1}i^{2j-\ell}Q_p(1)S_{p,\ell}(1),
\end{equation}
we obtain \eqref{eq:boundary-coefficients}.
\end{proof}

We now combine \eqref{eq:sector-interpolation-P} and \eqref{eq:boundary-coefficients} to prove Proposition \ref{prop:boundary-elimination}.

\begin{proof}[Proof of Proposition \ref{prop:boundary-elimination}]
By \eqref{e4}, the definition of $X(t,k)$, and \eqref{eq:sector-interpolation-P},
\[\widetilde g(t,k)=(-1)^{j+1}\left(\sum_{n=1}^{j+1}\mathcal P_{p,n}(k)\mathcal I_{p,n}(t,k)+\sum_{\ell=1}^j\mathfrak b_\ell(k)\widetilde\varphi_\ell(t,k^{2j+1})\right).\]
Substituting \eqref{eq:I-pn}, we obtain
\[\begin{aligned}\widetilde g(t,k)=&~{}\sum_{n=1}^{j+1}e^{-ik^{2j+1}t}\widehat u(t,\alpha_{p,n}k)\mathcal P_{p,n}(k)\\
&~{}-\sum_{n=1}^{j+1}\left(\widehat u_0(\alpha_{p,n}k)+F(t,\alpha_{p,n}k)\right)\mathcal P_{p,n}(k)\\
&~{}+(-1)^{j+1}\sum_{\ell=1}^j\left(\mathfrak b_\ell(k)-\sum_{n=1}^{j+1}\mathcal P_{p,n}(k)\mathfrak b_\ell(\alpha_{p,n}k)\right)\widetilde\varphi_\ell(t,k^{2j+1}).
\end{aligned}\]
Applying \eqref{eq:boundary-coefficients} to the last line gives \eqref{formula3}.
\end{proof}

Thus the remaining unknown boundary terms have been eliminated from $\widetilde g(t,k)$. The coefficients in \eqref{formula3} are analytic away from the possible pole $k=i\gamma$.

\subsection{The UTM representation}
We now set $t=T$ in \eqref{formula3} and substitute the resulting expression for $\widetilde g(T,k)$ into \eqref{formula2}. It remains to eliminate the integrals containing the unknown terms
\[\widehat u(T,\alpha_{p,n}k),\qquad 1\le n\le j+1.\]
We first consider the only possible pole of the coefficients $\mathcal P_{p,n}(k)$. By the sector geometry described in Section \ref{sec:linear-robin}, the pole $k=i\gamma$ belongs to $D^+$ precisely when $j$ is odd and $\gamma>0$. In this case, set
\[p_*:=\frac{j+1}{2},\]
so that
\[i\gamma\in D_{2p_*}^+.\]

Let $1\le p\le j$ and suppose that $i\gamma\notin D_{2p}^+$. For $k\in D_{2p}^+$, it follows from \eqref{eq:imaginary alpha k} that
\[\im(\alpha_{p,n}k)<0,\qquad 1\le n\le j+1,\]
and hence
\[k\longmapsto\widehat u(T,\alpha_{p,n}k)\]
is analytic in $D_{2p}^+$. Moreover, $\mathcal P_{p,n}(k)$ is analytic in $D_{2p}^+$. Since \eqref{eq:negative imaginary k^2j+1} and $t-T\le0$, we have
\[\left|e^{ik^{2j+1}(t-T)}\right|\le1,\qquad k\in D_{2p}^+.\]
Together with the decay of $e^{ikx}$ for $x>0$, this gives the required decay on the large circular arcs. Therefore, Cauchy's theorem gives
\[\sum_{n=1}^{j+1}\int_{\partial D_{2p}^+}e^{ikx+ik^{2j+1}(t-T)}\widehat u(T,\alpha_{p,n}k)\mathcal P_{p,n}(k)\,dk=0.\]
Thus these integrals vanish for every $1\le p\le j$ if $\gamma<0$, or if $\gamma>0$ and $j$ is even. If $\gamma>0$ and $j$ is odd, they vanish for every $p\neq p_*$.

We next consider the exceptional case where $j$ is odd, $\gamma>0$, and $p=p_*$. By \eqref{eq:P-pn},
\[\operatorname{Res}_{k=i\gamma}\mathcal P_{p_*,n}(k)=-\frac{i\gamma Q_{p_*}(1)}{Q_{p_*}'(\alpha_{p_*,n})}.\]
Since $\widetilde g(T,k)$ is entire, the residue of the right hand side of \eqref{formula3} at $k=i\gamma$ must vanish. Therefore,
\[\begin{aligned}
\operatorname{Res}_{k=i\gamma}\left(\sum_{n=1}^{j+1}e^{-ik^{2j+1}T}\widehat u(T,\alpha_{p_*,n}k)\mathcal P_{p_*,n}(k)\right)=&~{}-i\gamma Q_{p_*}(1)\sum_{n=1}^{j+1}\frac{\widehat u_0(i\gamma\alpha_{p_*,n})+F(T,i\gamma\alpha_{p_*,n})}{Q_{p_*}'(\alpha_{p_*,n})}\\
&~{}-\sum_{\ell=1}^jc_{p_*,\ell}(i\gamma)^{2j+1-\ell}\widetilde\varphi_\ell(T,-i\gamma^{2j+1}).
\end{aligned}\]
Here we used
\[(i\gamma)^{2j+1}=-i\gamma^{2j+1},\]
which holds because $j$ is odd.

With the orientation of $\partial D_{2p_*}^+$ used in \eqref{formula2}, the residue theorem gives
\begin{equation}\label{eq:unknown-residue-contribution}
\begin{aligned}
&~{}\frac{1}{2\pi}\sum_{n=1}^{j+1}\int_{\partial D_{2p_*}^+}e^{ikx+ik^{2j+1}(t-T)} \widehat u(T,\alpha_{p_*,n}k)\mathcal P_{p_*,n}(k)\,dk \\
=&~{}\gamma Q_{p_*}(1)e^{-\gamma x+\gamma^{2j+1}t}\sum_{n=1}^{j+1}\frac{\widehat u_0(i\gamma\alpha_{p_*,n})+F(T,i\gamma\alpha_{p_*,n})}{Q_{p_*}'(\alpha_{p_*,n})}\\
&~{}-ie^{-\gamma x+\gamma^{2j+1}t}\sum_{\ell=1}^jc_{p_*,\ell}(i\gamma)^{2j+1-\ell}\widetilde\varphi_\ell(T,-i\gamma^{2j+1}).
\end{aligned}
\end{equation}
Accordingly, define
\begin{equation}\label{eq:residue-term}
\mathcal R_{j,\gamma}[u_0,\bphi,f](t,x):=
\begin{cases}
\displaystyle
\gamma Q_{p_*}(1)e^{-\gamma x+\gamma^{2j+1}t}
\sum_{n=1}^{j+1}
\frac{\widehat u_0(i\gamma\alpha_{p_*,n})+F(T,i\gamma\alpha_{p_*,n})}
{Q_{p_*}'(\alpha_{p_*,n})}
\\
\displaystyle\qquad
-ie^{-\gamma x+\gamma^{2j+1}t}
\sum_{\ell=1}^jc_{p_*,\ell}(i\gamma)^{2j+1-\ell}
\widetilde\varphi_\ell(T,-i\gamma^{2j+1}),
& j \text{ odd and } \gamma>0,\\[2mm]
0,
& \text{otherwise}.
\end{cases}
\end{equation}

\begin{proposition}[UTM representation]\label{prop:utm-formula}
Let $u_0$, $\bphi=(\varphi_1,\ldots,\varphi_j)$, and $f$ be sufficiently smooth and decaying, and let $u$ be a sufficiently smooth solution of \eqref{eq:linear-robin} whose spatial derivatives decay at infinity. Then, for $x>0$ and $0\le t\le T$,
\begin{equation}\label{eq:utm-full-formula}
\begin{aligned}
u(t,x)=&~{}\frac{1}{2\pi}\int_{\R}e^{ikx+ik^{2j+1}t}\left(\widehat u_0(k)+F(t,k)\right)\,dk\\
&~{}-\frac{1}{2\pi}\sum_{p=1}^j\sum_{n=1}^{j+1}\int_{\partial D_{2p}^+}e^{ikx+ik^{2j+1}t} \left(\widehat u_0(\alpha_{p,n}k)+F(T,\alpha_{p,n}k)\right)\mathcal P_{p,n}(k)\,dk\\
&~{}+\frac{1}{2\pi}\sum_{p=1}^j\sum_{\ell=1}^jc_{p,\ell}\int_{\partial D_{2p}^+}e^{ikx+ik^{2j+1}t}\frac{k^{2j+1-\ell}}{k-i\gamma}\widetilde\varphi_\ell(T,k^{2j+1})\,dk\\
&~{}+\mathcal R_{j,\gamma}[u_0,\bphi,f](t,x).
\end{aligned}
\end{equation}
\end{proposition}

\begin{proof}
Set $t=T$ in \eqref{formula3} and substitute the resulting expression into \eqref{formula2}. The contour integrals containing $\widehat u(T,\alpha_{p,n}k)$ vanish by the argument above, except when $j$ is odd, $\gamma>0$, and $p=p_*$. In the exceptional case, their contribution is given by \eqref{eq:unknown-residue-contribution}, namely $\mathcal R_{j,\gamma}[u_0,\bphi,f]$. The case $\gamma=0$ follows by removing the singularity at $k=0$ as in Remark \ref{rem:gamma-zero}. This yields \eqref{eq:utm-full-formula}.
\end{proof}

\begin{remark}[The case $\gamma=0$]\label{rem:gamma-zero}
When $\gamma=0$, the factor $k/(k-i\gamma)$ in \eqref{eq:P-pn} has a removable singularity at $k=0$. Thus
\[\mathcal P_{p,n}(0):=L_{p,n}(1)-\frac{Q_p(1)}{Q_p'(\alpha_{p,n})}.\]
Moreover,
\[\frac{k^{2j+1-\ell}}{k-i\gamma}=k^{2j-\ell},\qquad 1\le\ell\le j.\]
Hence the boundary integrands in \eqref{eq:utm-full-formula} are regular at the origin, and no residue contribution arises from $k=0$.
\end{remark}

\subsection{Equivalence of the differential and UTM formulations}
For the verification of the boundary conditions, we first record two algebraic identities. For $1\le p\le j$, set
\[a_p:=e^{i(2p-1)\pi/(2j+1)},\qquad a_p':=e^{i2p\pi/(2j+1)}.\]
Then
\[a_p^{2j+1}=-1,\qquad (a_p')^{2j+1}=1,\]
and, by \eqref{eq:alpha-pn},
\[\alpha_{p,n}a_p=-\omega^{n-1},\qquad \alpha_{p,n}a_p'=\omega^{j+n},\qquad 1\le n\le j+1.\]

\begin{lemma}\label{lem:utm-boundary-algebra}
For every $1\le\ell\le j$, $1\le n\le j+1$, and $r>0$,
\begin{equation}\label{eq:P-ray-identity}
\sum_{p=1}^ja_p^\ell(a_pr-i\gamma)\mathcal P_{p,n}(a_pr)=-(-\omega^{n-1})^\ell(-\omega^{n-1}r-i\gamma)
\end{equation}
and
\begin{equation}\label{eq:P-ray-identity-prime}
\sum_{p=1}^j(a_p')^\ell(a_p'r-i\gamma)\mathcal P_{p,n}(a_p'r)=-(\omega^{j+n})^\ell(\omega^{j+n}r-i\gamma).
\end{equation}
Moreover, for every $1\le\ell,q\le j$,
\begin{equation}\label{eq:c-ray-identity}
\sum_{p=1}^jc_{p,q}a_p^{\ell-q}=\sum_{p=1}^jc_{p,q}(a_p')^{\ell-q}=-(2j+1)i^{-\ell}\delta_{\ell q}.
\end{equation}
\end{lemma}

\begin{proof}
By \eqref{eq:L-pn-Q},
\[\frac{Q_p(1)}{Q_p'(\alpha_{p,n})}=(1-\alpha_{p,n})L_{p,n}(1).\]
Hence \eqref{eq:P-pn} gives
\begin{equation}\label{eq:P-pole-cancellation}
(k-i\gamma)\mathcal P_{p,n}(k)=(\alpha_{p,n}k-i\gamma)L_{p,n}(1).
\end{equation}
For fixed $1\le n\le j+1$, define
\[\Lambda_n(z):=\prod_{\substack{1\le m\le j+1\\m\neq n}}\frac{z+\omega^{m-1}}{-\omega^{n-1}+\omega^{m-1}}.\]
Since
\[\alpha_{p,m}=\frac{-\omega^{m-1}}{a_p},\]
the definition \eqref{eq:L-pn} gives
\[L_{p,n}(1)=\Lambda_n(a_p).\]
The numbers
\[a_1,\ldots,a_j,-1,-\omega,\ldots,-\omega^j\]
are precisely the $2j+1$ roots of $z^{2j+1}+1$. Since $z^\ell\Lambda_n(z)$ has degree at most $2j$ and has zero constant term,
\[\sum_{p=1}^ja_p^\ell\Lambda_n(a_p)+\sum_{m=1}^{j+1}(-\omega^{m-1})^\ell\Lambda_n(-\omega^{m-1})=0.\]
By the definition of $\Lambda_n$,
\[\Lambda_n(-\omega^{m-1})=\delta_{nm},\]
and therefore
\begin{equation}\label{eq:L-ray-identity}
\sum_{p=1}^ja_p^\ell L_{p,n}(1)=-(-\omega^{n-1})^\ell.
\end{equation}
Since
\[a_p'=e^{i\pi/(2j+1)}a_p,\qquad \omega^{j+n}=e^{i\pi/(2j+1)}(-\omega^{n-1}),\]
we also obtain
\begin{equation}\label{eq:L-ray-identity-prime}
\sum_{p=1}^j(a_p')^\ell L_{p,n}(1)=-(\omega^{j+n})^\ell.
\end{equation}
Using \eqref{eq:P-pole-cancellation}, \eqref{eq:L-ray-identity}, and \eqref{eq:L-ray-identity-prime}, we obtain \eqref{eq:P-ray-identity} and \eqref{eq:P-ray-identity-prime}.

We next prove \eqref{eq:c-ray-identity}. Since the coefficient $c_{p,q}$ in \eqref{eq:c-p-ell} is independent of $\gamma$, we may evaluate \eqref{eq:boundary-coefficients} at $\gamma=0$. In this case,
\[\mathfrak b_q(k)=(ik)^{2j-q},\]
while \eqref{eq:P-pole-cancellation} gives
\[\mathcal P_{p,n}(k)=\alpha_{p,n}L_{p,n}(1).\]
Thus \eqref{eq:boundary-coefficients} gives
\begin{equation}\label{eq:c-interpolation-identity}
c_{p,q}=(-1)^{j+1}i^{2j-q}\left(1-\sum_{n=1}^{j+1}\alpha_{p,n}^{2j+1-q}L_{p,n}(1)\right).
\end{equation}
Multiplying \eqref{eq:c-interpolation-identity} by $a_p^{\ell-q}$ and summing over $p$, we obtain
\[\sum_{p=1}^jc_{p,q}a_p^{\ell-q}=(-1)^{j+1}i^{2j-q}\left(\sum_{p=1}^ja_p^{\ell-q}-\sum_{n=1}^{j+1}(-\omega^{n-1})^{2j+1-q}\sum_{p=1}^ja_p^{\ell-2j-1}L_{p,n}(1)\right).\]
Since $a_p^{2j+1}=-1$, \eqref{eq:L-ray-identity} gives
\[\sum_{p=1}^ja_p^{\ell-2j-1}L_{p,n}(1)=(-\omega^{n-1})^\ell.\]
Hence
\[\sum_{p=1}^jc_{p,q}a_p^{\ell-q}=(-1)^{j+1}i^{2j-q}\left(\sum_{p=1}^ja_p^{\ell-q}+\sum_{n=1}^{j+1}(-\omega^{n-1})^{\ell-q}\right).\]
The two sums on the right give the sum of the $(\ell-q)$th powers of all roots of $z^{2j+1}+1$. Since the roots are nonzero,
\[\sum_{\zeta^{2j+1}=-1}\zeta^m=0\]
for every nonzero integer $m$ with $|m|<2j+1$, while the sum equals $2j+1$ for $m=0$. Since
\[1-j\le\ell-q\le j-1,\]
we obtain
\[\sum_{p=1}^ja_p^{\ell-q}+\sum_{n=1}^{j+1}(-\omega^{n-1})^{\ell-q}=(2j+1)\delta_{\ell q}.\]
Therefore,
\[\sum_{p=1}^jc_{p,q}a_p^{\ell-q}=-(2j+1)i^{-\ell}\delta_{\ell q}.\]
Finally,
\[\sum_{p=1}^jc_{p,q}(a_p')^{\ell-q}=e^{i(\ell-q)\pi/(2j+1)}\sum_{p=1}^jc_{p,q}a_p^{\ell-q}=-(2j+1)i^{-\ell}\delta_{\ell q}.\]
This proves \eqref{eq:c-ray-identity}.
\end{proof}

\begin{proposition}[Verification of the UTM representation]\label{prop:utm-verification}
Let $u_0$, $\bphi=(\varphi_1,\ldots,\varphi_j)$, and $f$ be sufficiently smooth and decaying. Define $u$ by the right hand side of \eqref{eq:utm-full-formula}. Then $u$ satisfies the linear IBVP \eqref{eq:linear-robin}.
\end{proposition}

\begin{proof}
All contour calculations below are first made on bounded parts of the contours. The smoothness and decay assumptions, together with the same estimates used in the derivation of \eqref{eq:utm-full-formula}, justify differentiation under the integrals, taking the limit $x\to0^+$, and then letting the contour radius tend to infinity. We verify the differential equation, the initial condition, and the boundary conditions separately.

First, since
\[\left(\partial_t+(-1)^{j+1}\partial_x^{2j+1}\right)e^{ikx+ik^{2j+1}t}=0,\]
the second and third terms in \eqref{eq:utm-full-formula} satisfy the homogeneous equation. By \eqref{eq:F},
\[\partial_tF(t,k)=e^{-ik^{2j+1}t}\widehat f(t,k),\]
and therefore
\[\left(\partial_t+(-1)^{j+1}\partial_x^{2j+1}\right)\frac{1}{2\pi}\int_\R e^{ikx+ik^{2j+1}t}\left(\widehat u_0(k)+F(t,k)\right)\,dk=\frac{1}{2\pi}\int_\R e^{ikx}\widehat f(t,k)\,dk=f(t,x).\]
When $\mathcal R_{j,\gamma}[u_0,\bphi,f]$ is nonzero, $j$ is odd and $\gamma>0$. In this case,
\[\left(\partial_t+(-1)^{j+1}\partial_x^{2j+1}\right)e^{-\gamma x+\gamma^{2j+1}t}=0.\]
Hence
\[\partial_tu+(-1)^{j+1}\partial_x^{2j+1}u=f.\]

We next verify the initial condition. Since $F(0,k)=0$, the first term in \eqref{eq:utm-full-formula} gives
\[\frac{1}{2\pi}\int_\R e^{ikx}\widehat u_0(k)\,dk=u_0(x),\qquad x>0.\]
At $t=0$, the remaining contour integrals may be closed in the sectors $D_{2p}^+$ because $e^{ikx}$ decays there for $x>0$. If $\gamma<0$, the pole $k=i\gamma$ lies in the lower half plane. If $\gamma>0$ and $j$ is even, the pole lies in $D_{j+1}^+$ and hence outside $D^+$. If $\gamma=0$, the singularity at $k=0$ is removable by Remark \ref{rem:gamma-zero}. Thus all remaining contour integrals vanish in these cases.

Suppose that $j$ is odd and $\gamma>0$. Then $i\gamma$ lies in $D_{2p_*}^+$, where
\[p_*=\frac{j+1}{2}.\]
By \eqref{eq:P-pn},
\[\operatorname{Res}_{k=i\gamma}\mathcal P_{p_*,n}(k)=-\frac{i\gamma Q_{p_*}(1)}{Q_{p_*}'(\alpha_{p_*,n})}.\]
Therefore, the second term in \eqref{eq:utm-full-formula} at $t=0$ contributes
\[-\gamma Q_{p_*}(1)e^{-\gamma x}\sum_{n=1}^{j+1}\frac{\widehat u_0(i\gamma\alpha_{p_*,n})+F(T,i\gamma\alpha_{p_*,n})}{Q_{p_*}'(\alpha_{p_*,n})},\]
while the third term contributes
\[ie^{-\gamma x}\sum_{\ell=1}^jc_{p_*,\ell}(i\gamma)^{2j+1-\ell}\widetilde\varphi_\ell(T,-i\gamma^{2j+1}).\]
By \eqref{eq:residue-term}, the sum of these two contributions is
\[-\mathcal R_{j,\gamma}[u_0,\bphi,f](0,x).\]
Hence all remaining terms cancel, and
\[u(0,x)=u_0(x),\qquad x>0.\]

It remains to verify the boundary conditions. Fix $1\le\ell\le j$. Since
\begin{equation}\label{eq:Robin-trace-exponential}
(\partial_x+\gamma)\partial_x^{\ell-1}e^{ikx}=i^\ell(k-i\gamma)k^{\ell-1}e^{ikx},
\end{equation}
and every term in \eqref{eq:residue-term} contains the factor $e^{-\gamma x}$, we have
\[(\partial_x+\gamma)\partial_x^{\ell-1}\mathcal R_{j,\gamma}[u_0,\bphi,f](t,0)=0.\]
Applying $(\partial_x+\gamma)\partial_x^{\ell-1}$ to \eqref{eq:utm-full-formula} and setting $x=0$, we obtain
\begin{equation}\label{eq:UTM-boundary-trace}
\begin{aligned}
(\partial_x+\gamma)\partial_x^{\ell-1}u(t,0)=&~{}\frac{i^\ell}{2\pi}\int_\R e^{ik^{2j+1}t}(k-i\gamma)k^{\ell-1}\left(\widehat u_0(k)+F(t,k)\right)\,dk\\
&~{}-\frac{i^\ell}{2\pi}\sum_{p=1}^j\sum_{n=1}^{j+1}\int_{\partial D_{2p}^+}e^{ik^{2j+1}t}(k-i\gamma)k^{\ell-1}\left(\widehat u_0(\alpha_{p,n}k)+F(T,\alpha_{p,n}k)\right)\mathcal P_{p,n}(k)\,dk\\
&~{}+\frac{i^\ell}{2\pi}\sum_{p=1}^j\sum_{q=1}^jc_{p,q}\int_{\partial D_{2p}^+}e^{ik^{2j+1}t}k^{2j+\ell-q}\widetilde\varphi_q(T,k^{2j+1})\,dk.
\end{aligned}
\end{equation}

We first replace $F(T,\alpha_{p,n}k)$ by $F(t,\alpha_{p,n}k)$ in the second term. Since $\alpha_{p,n}^{2j+1}=1$, \eqref{eq:F} gives
\[e^{ik^{2j+1}t}\left(F(T,\alpha_{p,n}k)-F(t,\alpha_{p,n}k)\right)=\int_t^T e^{ik^{2j+1}(t-\tau)}\widehat f(\tau,\alpha_{p,n}k)\,d\tau.\]
By \eqref{eq:P-pole-cancellation}, the factor $k-i\gamma$ removes the possible pole at $k=i\gamma$. By \eqref{eq:imaginary alpha k}, $\alpha_{p,n}k$ lies in the lower half plane, where $\widehat f(\tau,\cdot)$ is analytic. Hence the integrand is analytic in $D_{2p}^+$. Moreover, by \eqref{eq:negative imaginary k^2j+1} and $t-\tau\le0$, the exponential factor is bounded there. An integration by parts in $\tau$, followed by the same large arc estimate used in the proof of \eqref{eq:g-t-T-vanishing}, gives
\begin{equation}\label{eq:future-forcing-boundary-zero}
\int_{\partial D_{2p}^+}e^{ik^{2j+1}t}(k-i\gamma)k^{\ell-1}\left(F(T,\alpha_{p,n}k)-F(t,\alpha_{p,n}k)\right)\mathcal P_{p,n}(k)\,dk=0.
\end{equation}

We next compare the first two terms in \eqref{eq:UTM-boundary-trace}. Set
\[\mathcal K_\ell(t,k):=e^{ik^{2j+1}t}(k-i\gamma)k^{\ell-1}\left(\widehat u_0(k)+F(t,k)\right).\]
Parametrizing the lower and upper rays of $\partial D_{2p}^+$ by $k=a_pr$ and $k=a_p'r$, respectively, and using their orientations, \eqref{eq:P-ray-identity} and \eqref{eq:P-ray-identity-prime} give
\[\begin{aligned}
&~{}-\sum_{p=1}^j\sum_{n=1}^{j+1}\int_{\partial D_{2p}^+}e^{ik^{2j+1}t}(k-i\gamma)k^{\ell-1}\left(\widehat u_0(\alpha_{p,n}k)+F(t,\alpha_{p,n}k)\right)\mathcal P_{p,n}(k)\,dk\\
=&~{}\sum_{n=1}^{j+1}\left(\int_{-\omega^{n-1}\R^+}\mathcal K_\ell(t,k)\,dk-\int_{\omega^{j+n}\R^+}\mathcal K_\ell(t,k)\,dk\right).
\end{aligned}\]
For $1\le n\le j$, the rays
\[\omega^{j+n}\R^+,\qquad -\omega^n\R^+\]
bound a sector in the lower half plane in which
\[\im(k^{2j+1})>0.\]
The function $\mathcal K_\ell(t,k)$ is analytic there. Applying the same large arc estimate to
\[e^{ik^{2j+1}t}F(t,k)=\int_0^t e^{ik^{2j+1}(t-\tau)}\widehat f(\tau,k)\,d\tau\]
gives
\[\int_{-\omega^n\R^+}\mathcal K_\ell(t,k)\,dk=\int_{\omega^{j+n}\R^+}\mathcal K_\ell(t,k)\,dk,\qquad 1\le n\le j.\]
Since
\[-\omega^0=-1,\qquad \omega^{2j+1}=1,\]
all intermediate rays cancel and we obtain
\[-\sum_{p=1}^j\sum_{n=1}^{j+1}\int_{\partial D_{2p}^+}e^{ik^{2j+1}t}(k-i\gamma)k^{\ell-1}\left(\widehat u_0(\alpha_{p,n}k)+F(t,\alpha_{p,n}k)\right)\mathcal P_{p,n}(k)\,dk=-\int_\R\mathcal K_\ell(t,k)\,dk.\]
Together with \eqref{eq:future-forcing-boundary-zero}, this shows that the first two terms on the right hand side of \eqref{eq:UTM-boundary-trace} cancel.

It remains to evaluate the third term. Parametrizing the two rays of $\partial D_{2p}^+$ by
\[k=a_pr,\qquad k=a_p'r,\]
using
\[a_p^{2j+1}=-1,\qquad (a_p')^{2j+1}=1,\]
and then setting $\tau=r^{2j+1}$, we obtain
\[\begin{aligned}
(\partial_x+\gamma)\partial_x^{\ell-1}u(t,0)=&~{}-\frac{i^\ell}{2\pi(2j+1)}\sum_{q=1}^j\left(\sum_{p=1}^jc_{p,q}a_p^{\ell-q}\right)\int_0^\infty e^{-i\tau t}\tau^{\frac{\ell-q}{2j+1}}\widetilde\varphi_q(T,-\tau)\,d\tau\\
&~{}-\frac{i^\ell}{2\pi(2j+1)}\sum_{q=1}^j\left(\sum_{p=1}^jc_{p,q}(a_p')^{\ell-q}\right)\int_0^\infty e^{i\tau t}\tau^{\frac{\ell-q}{2j+1}}\widetilde\varphi_q(T,\tau)\,d\tau.
\end{aligned}\]
By \eqref{eq:c-ray-identity}, all terms with $q\neq\ell$ vanish. For $q=\ell$, we obtain
\[\begin{aligned}
(\partial_x+\gamma)\partial_x^{\ell-1}u(t,0)=&~{}\frac{1}{2\pi}\int_0^\infty e^{-i\tau t}\widetilde\varphi_\ell(T,-\tau)\,d\tau\\
&~{}+\frac{1}{2\pi}\int_0^\infty e^{i\tau t}\widetilde\varphi_\ell(T,\tau)\,d\tau\\
=&~{}\frac{1}{2\pi}\int_\R e^{i\tau t}\widetilde\varphi_\ell(T,\tau)\,d\tau\\
=&~{}\varphi_\ell(t).
\end{aligned}\]
The last equality follows from the definition of $\widetilde\varphi_\ell(T,\cdot)$ and Fourier inversion for the zero extension of $\varphi_\ell$ outside $(0,T)$. Hence
\[(\partial_x+\gamma)\partial_x^{\ell-1}u(t,0)=\varphi_\ell(t),\qquad 1\le\ell\le j.\]
Thus the function defined by \eqref{eq:utm-full-formula} satisfies \eqref{eq:linear-robin}.
\end{proof}

\begin{corollary}\label{cor:utm-equivalence}
For sufficiently smooth and decaying data, the differential formulation \eqref{eq:linear-robin} is equivalent to the UTM formulation \eqref{eq:utm-full-formula}.
\end{corollary}

\begin{proof}
Proposition \ref{prop:utm-formula} shows that every sufficiently smooth and decaying solution of \eqref{eq:linear-robin} satisfies \eqref{eq:utm-full-formula}. Conversely, Proposition \ref{prop:utm-verification} shows that the function defined by \eqref{eq:utm-full-formula} satisfies \eqref{eq:linear-robin}.
\end{proof}

\section{Reduced linear IBVP with Robin boundary conditions}\label{sec:reduced-ibvp}
We consider
\begin{equation}\label{eq:reduced-ibvp}
\begin{cases}
\partial_tu+(-1)^{j+1}\partial_x^{2j+1}u=0, & 0<t<2,\ x>0,\\
u(0,x)=0, & x>0,\\
(\partial_x+\gamma)\partial_x^{\ell-1}u(t,0)=\varphi_\ell(t), & 0<t<2,\ 1\le\ell\le j.
\end{cases}
\end{equation}
We first assume that each $\varphi_\ell$ belongs to $C_0^\infty((0,2))$ and identify it with its zero extension to $\R$. Then, for $\zeta\in\C$,
\[\widetilde\varphi_\ell(2,\zeta)=\widehat\varphi_\ell(\zeta).\]
For $1\le p,\ell\le j$, define
\[u_{p,\ell}(t,x):=\int_{\partial D_{2p}^+}e^{ikx+ik^{2j+1}t}\frac{k^{2j+1-\ell}}{k-i\gamma}\widehat\varphi_\ell(k^{2j+1})\,dk.\]
By Corollary \ref{cor:utm-equivalence}, for smooth boundary data the solution of \eqref{eq:reduced-ibvp} is given by the specialization of \eqref{eq:utm-full-formula} with $u_0=0$ and $f=0$. Thus, if $j$ is even or $\gamma\le0$, then
\[S_j[0,\bphi,0]=\frac{1}{2\pi}\sum_{p=1}^j\sum_{\ell=1}^jc_{p,\ell}u_{p,\ell}.\]
If $j$ is odd and $\gamma>0$, then
\[S_j[0,\bphi,0]=\frac{1}{2\pi}\sum_{p=1}^j\sum_{\ell=1}^jc_{p,\ell}u_{p,\ell}+\mathcal R_{j,\gamma}[0,\bphi,0].\]
When $\gamma=0$, the apparent singularity at $k=0$ is removable, see Remark \ref{rem:gamma-zero}.

The following proposition establishes the estimates for the reduced solution operator and its extension to the natural boundary data space $\mathcal H_2^s$.
\begin{proposition}\label{prop:reduced-ibvp}
Let
\[-j-\frac12<s<\frac32\]
and let $\bphi\in\mathcal H_2^s$. The reduced solution operator defined above for smooth boundary data extends uniquely to $\mathcal H_2^s$. If $s\ge-j$, then
\begin{equation}\label{eq:reduced-sobolev}
\sup_{t\in[0,2]}\|S_j[0,\bphi,0](t)\|_{H^s(\R^+)}
\lesssim_s\|\bphi\|_{\mathcal H_2^s}.
\end{equation}
If
\[0\le b<\frac12,\qquad \frac12<\alpha\le1+\frac{s}{2j+1},\]
then
\begin{equation}\label{eq:reduced-X}
\|S_j[0,\bphi,0]\|_{X_{\Omega_2}^{s,b,\alpha}}
\lesssim_{s,b,\alpha}\|\bphi\|_{\mathcal H_2^s}.
\end{equation}
The implicit constants may depend on the fixed parameters $j$ and $\gamma$.
\end{proposition}

\subsection{Contour decomposition}

Fix $1\le p,\ell\le j$. The boundary $\partial D_{2p}^+$ consists of the two rays with arguments
\[\frac{(2p-1)\pi}{2j+1}\qquad\text{and}\qquad \frac{2p\pi}{2j+1}.\]
Recall that
\[a_p=e^{i(2p-1)\pi/(2j+1)},\qquad a_p'=e^{i2p\pi/(2j+1)}.\]
Then $a_pk$ and $a_p'k$, $k\ge0$, parametrize the two boundary rays of $D_{2p}^+$. Since both rays lie in the open upper half plane,
\[\im a_p>0,\qquad \im a_p'>0.\]
Moreover, neither $\re a_p$ nor $\re a_p'$ vanishes, since neither of the angles
\[\frac{(2p-1)\pi}{2j+1},\qquad \frac{2p\pi}{2j+1}\]
can equal $\pi/2$.

Parametrizing the two rays of $\partial D_{2p}^+$ gives
\begin{equation}\label{eq:reduced-right-ray}
u_r^{p,\ell}(t,x):=\int_0^\infty e^{ia_pkx-ik^{2j+1}t}\frac{a_p^{2j+2-\ell}k^{2j+1-\ell}}{a_pk-i\gamma}\widehat\varphi_\ell(-k^{2j+1})\,dk
\end{equation}
and
\begin{equation}\label{eq:reduced-left-ray}
u_l^{p,\ell}(t,x):=-\int_0^\infty e^{ia_p'kx+ik^{2j+1}t}\frac{(a_p')^{2j+2-\ell}k^{2j+1-\ell}}{a_p'k-i\gamma}\widehat\varphi_\ell(k^{2j+1})\,dk.
\end{equation}
Thus
\[u_{p,\ell}=u_r^{p,\ell}+u_l^{p,\ell}.\]

For $k\ge0$,
\[|a_pk-i\gamma|^2=(\re a_p)^2k^2+((\im a_p)k-\gamma)^2\ge(\re a_p)^2k^2,\]
and the same argument applies to $a_p'$. Since $1\le p\le j$ and $\re a_p,\re a_p'\neq0$, we have
\begin{equation}\label{eq:reduced-denominator}
|a_pk-i\gamma|\gtrsim_j k,\qquad |a_p'k-i\gamma|\gtrsim_j k,\qquad k\ge0.
\end{equation}

We split
\[u_r^{p,\ell}=u_{r,0}^{p,\ell}+u_{r,1}^{p,\ell},\]
where $u_{r,0}^{p,\ell}$ and $u_{r,1}^{p,\ell}$ are the low and high frequency contributions corresponding to $0\le k\le1$ and $k>1$, respectively. The left contribution is decomposed and estimated in the same way.

\subsection{High frequency estimates}

For the high frequency contribution $u_{r,1}^{p,\ell}$, we write
\begin{equation}\label{eq:reduced-high-derivative}
u_{r,1}^{p,\ell}(t,x)=\frac{a_p^{2j+2-\ell}}{(ia_p)^j}\partial_x^j\int_1^\infty e^{-ik^{2j+1}t}e^{ia_pkx}\frac{k^{j+1-\ell}}{a_pk-i\gamma}\widehat\varphi_\ell(-k^{2j+1})\,dk.
\end{equation}
Let $\rho\in C^\infty(\R)$ satisfy
\[0\le\rho\le1,\qquad \rho(y)=1\quad\text{for }y\ge0,\qquad \rho(y)=0\quad\text{for }y\le-1.\]
We extend \eqref{eq:reduced-high-derivative} to $x\in\R$ by replacing $e^{ia_pkx}$ with
\[e^{ia_pkx}\rho((\im a_p)kx).\]
Since $\im a_p>0$, this extension agrees with the original function for $x>0$. Keeping the same notation for the extension, define
\[\eta_{a_p}(y):=e^{i(\re a_p)y/\im a_p}e^{-y}\rho(y).\]
Note that $\eta_{a_p} \in \mathcal S(\R)$. After the change of variables $\tau=-k^{2j+1}$, we obtain
\begin{equation}\label{eq:reduced-high-tau}
u_{r,1}^{p,\ell}(t,x)=\frac{a_p^{j+2-\ell}}{i^j(2j+1)}\partial_x^j\int_{-\infty}^{-1}e^{i\tau t}\eta_{a_p}\left((\im a_p)(-\tau)^{1/(2j+1)}x\right)\frac{(-\tau)^{-(j+\ell-1)/(2j+1)}}{a_p(-\tau)^{1/(2j+1)}-i\gamma}\widehat\varphi_\ell(\tau)\,d\tau.
\end{equation}
For $\tau<-1$, define
\[H_{a_p}(\tau,\xi):=\int_\R e^{-ix\xi}\eta_{a_p}\left((\im a_p)(-\tau)^{1/(2j+1)}x\right)\,dx.\]

\begin{lemma}\label{lem:high-kernel}
For every $n\ge0$,
\begin{equation}\label{eq:high-kernel-decay}
|H_{a_p}(\tau,\xi)|\lesssim_{j,n}|\tau|^{-1/(2j+1)}\left(\frac{|\tau|^{1/(2j+1)}}{|\xi|+|\tau|^{1/(2j+1)}}\right)^n
\end{equation}
for $\tau<-1$ and $\xi\in\R$. Moreover, if $s>-j-\frac12$ and $b\ge0$, then
\begin{equation}\label{eq:high-kernel-weighted}
\int_\R\bra{\xi}^{2s}\bra{\tau-\xi^{2j+1}}^{2b}|\xi^jH_{a_p}(\tau,\xi)|^2\,d\xi \lesssim_{s,b}|\tau|^{\frac{2(s+j)+2(2j+1)b-1}{2j+1}}.
\end{equation}
\end{lemma}

\begin{proof}
Set
\[R:=(-\tau)^{1/(2j+1)}.\]
Since $\tau<-1$, we have $R>1$. By the definition of $H_{a_p}$ and the change of variables $y=(\im a_p)Rx$,
\[H_{a_p}(\tau,\xi)=\frac{1}{(\im a_p)R}\widehat{\eta_{a_p}}\left(\frac{\xi}{(\im a_p)R}\right).\]
Since $\eta_{a_p}\in\mathcal S(\R)$, for every $n\ge0$,
\[|H_{a_p}(\tau,\xi)|\lesssim_{j,n}R^{-1}\left(\frac{R}{|\xi|+R}\right)^n.\]
This proves \eqref{eq:high-kernel-decay}.

To prove \eqref{eq:high-kernel-weighted}, we split the $\xi$ integral into $|\xi|\le2R$ and $|\xi|>2R$. On $|\xi|\le2R$,
\[\bra{\tau-\xi^{2j+1}}^{2b}\lesssim R^{2(2j+1)b}.\]
Since $s+j>-\frac12$,
\[\int_{|\xi|\le2R}\bra{\xi}^{2s}|\xi|^{2j}\,d\xi\lesssim_sR^{2(s+j)+1}.\]
Using \eqref{eq:high-kernel-decay} with $n=0$, we obtain
\[\int_{|\xi|\le2R}\bra{\xi}^{2s}\bra{\tau-\xi^{2j+1}}^{2b}|\xi^jH_{a_p}(\tau,\xi)|^2\,d\xi\lesssim_{s,b}R^{2(s+j)+2(2j+1)b-1}.\]

On $|\xi|>2R$, we have
\[\bra{\tau-\xi^{2j+1}}\sim\bra{\xi}^{2j+1}.\]
Choosing $n$ sufficiently large in \eqref{eq:high-kernel-decay}, we obtain
\[\begin{aligned}
\int_{|\xi|>2R}\bra{\xi}^{2s}\bra{\tau-\xi^{2j+1}}^{2b}|\xi^jH_{a_p}(\tau,\xi)|^2\,d\xi
\lesssim&~{}R^{2n-2}\int_{|\xi|>2R}|\xi|^{2(s+j)+2(2j+1)b-2n}\,d\xi\\
\lesssim_{s,b}&~{}R^{2(s+j)+2(2j+1)b-1}.
\end{aligned}\]
Combining the two regions and using $R=|\tau|^{1/(2j+1)}$ gives \eqref{eq:high-kernel-weighted}.
\end{proof}

The estimate in Lemma \ref{lem:high-kernel} first gives the $X^{s,b}$ space bound for the high frequency contribution.

\begin{lemma}\label{lem:high-X}
If $s>-j-\frac12$ and $b\ge0$, then
\begin{equation}\label{eq:high-X-general}
\|u_{r,1}^{p,\ell}\|_{X^{s,b}(\R^2)}
\lesssim_{s,b}
\|\varphi_\ell\|_{H^{\frac{s+(2j+1)b-\ell-\frac12}{2j+1}}(\R)}.
\end{equation}
In particular, if $b<\frac12$, then
\begin{equation}\label{eq:high-X-boundary}
\|u_{r,1}^{p,\ell}\|_{X^{s,b}(\R^2)}
\lesssim_{s,b}
\|\varphi_\ell\|_{H^{r_\ell}(\R)}.
\end{equation}
\end{lemma}

\begin{proof}
By \eqref{eq:reduced-high-tau},
\[|\mathcal F(u_{r,1}^{p,\ell})(\tau,\xi)|\lesssim\mathbf{1}_{(-\infty,-1)}(\tau)|\xi|^j|H_{a_p}(\tau,\xi)|\frac{|\tau|^{-\frac{j+\ell-1}{2j+1}}}{|a_p(-\tau)^{\frac{1}{2j+1}}-i\gamma|}|\widehat\varphi_\ell(\tau)|.\]
By \eqref{eq:reduced-denominator},
\[|a_p(-\tau)^{\frac{1}{2j+1}}-i\gamma|\gtrsim_j|\tau|^{\frac{1}{2j+1}},\qquad \tau<-1.\]
Using Lemma \ref{lem:high-kernel}, we obtain
\[\begin{aligned}
\|u_{r,1}^{p,\ell}\|_{X^{s,b}(\R^2)}^2\lesssim&~{}\int_{-\infty}^{-1}|\tau|^{\frac{2(s+j)+2(2j+1)b-1}{2j+1}}|\tau|^{-\frac{2(j+\ell)}{2j+1}}|\widehat\varphi_\ell(\tau)|^2\,d\tau\\
=&~{}\int_{-\infty}^{-1}|\tau|^{\frac{2s+2(2j+1)b-2\ell-1}{2j+1}}|\widehat\varphi_\ell(\tau)|^2\,d\tau\\
\lesssim&~{}\|\varphi_\ell\|_{H^{\frac{s+(2j+1)b-\ell-\frac12}{2j+1}}(\R)}^2.
\end{aligned}\]
This proves \eqref{eq:high-X-general}. If $b<\frac12$, then
\[\frac{s+(2j+1)b-\ell-\frac12}{2j+1}<\frac{s+j-\ell}{2j+1}=r_\ell,\]
and therefore \eqref{eq:high-X-boundary} follows.
\end{proof}

We next estimate the $\mathcal D^\alpha$ norm of the high frequency contribution.

\begin{lemma}\label{lem:high-D}
If
\[s>-j-\frac12,\qquad \frac12<\alpha\le1+\frac{s}{2j+1},\]
then
\begin{equation}\label{eq:high-D}
\|u_{r,1}^{p,\ell}\|_{\mathcal D^\alpha(\R^2)}\lesssim_{s,\alpha}\|\varphi_\ell\|_{H^{r_\ell}(\R)}.
\end{equation}
\end{lemma}

\begin{proof}
For $|\xi|\le1$, \eqref{eq:high-kernel-decay} with $n=0$ gives
\[|H_{a_p}(\tau,\xi)|\lesssim_j|\tau|^{-\frac{1}{2j+1}},\qquad \tau<-1.\]
Together with \eqref{eq:reduced-denominator}, this yields
\[\|u_{r,1}^{p,\ell}\|_{\mathcal D^\alpha(\R^2)}^2\lesssim\int_{-\infty}^{-1}\langle\tau\rangle^{2\alpha}|\tau|^{-\frac{2(j+\ell+1)}{2j+1}}|\widehat\varphi_\ell(\tau)|^2\,d\tau.\]
Since
\[\alpha-\frac{j+\ell+1}{2j+1}\le\frac{s+j-\ell}{2j+1}=r_\ell,\]
we obtain \eqref{eq:high-D}.
\end{proof}

Finally, we estimate the high frequency contribution in $C_tH_x^s$.

\begin{lemma}\label{lem:high-sobolev}
If $s\ge-j$, then
\begin{equation}\label{eq:high-sobolev}
\sup_{t\in[0,2]}\|u_{r,1}^{p,\ell}(t)\|_{H^s(\R^+)}\lesssim_s\|\varphi_\ell\|_{H^{r_\ell}(\R)}.
\end{equation}
\end{lemma}

\begin{proof}
To estimate $u_{r,1}^{p,\ell}$ in $H^s(\R^+)$, we extract the $j$ spatial derivatives in \eqref{eq:reduced-high-derivative} and write
\[u_{r,1}^{p,\ell}=\partial_x^jw_{r,1}^{p,\ell},\]
where
\[w_{r,1}^{p,\ell}(t,x):=\frac{a_p^{2j+2-\ell}}{(ia_p)^j}\int_1^\infty e^{-ik^{2j+1}t}e^{ia_pkx}\frac{k^{j+1-\ell}}{a_pk-i\gamma}\widehat\varphi_\ell(-k^{2j+1})\,dk.\]
By \eqref{eq:reduced-denominator},
\[\left|\frac{k^{j+1-\ell}}{a_pk-i\gamma}\right|\lesssim_jk^{j-\ell},\qquad k\ge1.\]
Since $s+j\ge0$ and $\im a_p>0$, Lemma \ref{lem:ray-integral} gives
\[\begin{aligned}
\|w_{r,1}^{p,\ell}(t)\|_{H^{s+j}(\R^+)}^2\lesssim&~{}\int_1^\infty k^{2(s+2j-\ell)}|\widehat\varphi_\ell(-k^{2j+1})|^2\,dk\\
=&~{}\frac{1}{2j+1}\int_{-\infty}^{-1}|\tau|^{\frac{2(s+j-\ell)}{2j+1}}|\widehat\varphi_\ell(\tau)|^2\,d\tau\\
\lesssim&~{}\|\varphi_\ell\|_{H^{r_\ell}(\R)}^2.
\end{aligned}\]
The boundedness of
\[\partial_x^j:H^{s+j}(\R^+)\longrightarrow H^s(\R^+)\]
proves \eqref{eq:high-sobolev}.
\end{proof}

The same arguments apply to the high frequency left contribution after replacing $a_p$ by $a_p'$ and using the change of variables $\tau=k^{2j+1}$.

\subsection{Low frequency estimates}

For the high frequency contribution, we used the decay estimates for $H_{a_p}$ established in Lemma \ref{lem:high-kernel}. On the bounded frequency interval $0\le k\le1$, we instead construct a smooth spatial extension and estimate its derivatives directly.

Let $\vartheta\in C^\infty(\R)$ be a smooth function satisfying
\begin{equation}\label{eq:vartheta}
\vartheta(x)=x\quad\text{for }x\ge0,\qquad \vartheta(x)=-x\quad\text{for }x\le-1.
\end{equation}
We extend the low frequency right contribution to $\R^2$ by
\begin{equation}\label{eq:reduced-low-extension}
u_{r,0}^{p,\ell}(t,x):=\int_0^1e^{-ik^{2j+1}t}e^{ia_pk\vartheta(x)}\frac{a_p^{2j+2-\ell}k^{2j+1-\ell}}{a_pk-i\gamma}\widehat\varphi_\ell(-k^{2j+1})\,dk.
\end{equation}
This agrees with the original contour contribution for $x>0$. By \eqref{eq:reduced-denominator},
\begin{equation}\label{eq:reduced-low-symbol}
\left|\frac{k^{2j+1-\ell}}{a_pk-i\gamma}\right|\lesssim_j k^{2j-\ell},\qquad 0<k\le1.
\end{equation}
When $\gamma=0$, the quotient is extended continuously to $k=0$ as in Remark \ref{rem:gamma-zero}.

\begin{lemma}\label{lem:low-derivative}
For every $m_1,m_2\in\N_0$,
\begin{equation}\label{eq:low-derivative}
\sup_{|t|\le4}\left\|\partial_x^{m_1}\partial_t^{m_2}[\psi_2(t)u_{r,0}^{p,\ell}(t)]\right\|_{L^2(\R)}^2\lesssim_{m_1,m_2}\int_{-1}^0|\widehat\varphi_\ell(\tau)|^2\,d\tau,
\end{equation}
where $\psi_2$ is defined in \eqref{eq:psi_T} with $T=2$.
\end{lemma}

\begin{proof}
We split the $x$ integration into $[-1,1]$, $(1,\infty)$, and $(-\infty,-1)$. 

By Leibniz' rule,
\[\partial_x^{m_1}\partial_t^{m_2}\left[\psi_2(t)e^{-ik^{2j+1}t}e^{ia_pk\vartheta(x)}\right]
=e^{-ik^{2j+1}t}\sum_{\mu=0}^{m_2}\binom{m_2}{\mu}(-ik^{2j+1})^{m_2-\mu}\partial_t^\mu\psi_2(t)\partial_x^{m_1}\left(e^{ia_pk\vartheta(x)}\right).\]
Since all derivatives of $\vartheta$ are bounded on $[-1,1]$, we have
\[\left|(-ik^{2j+1})^{m_2-\mu}\partial_x^{m_1}\left(e^{ia_pk\vartheta(x)}\right)\right|\lesssim_{\vartheta,m_1,m_2}1,\qquad |x|\le1,\ 0\le k\le1.\]
Moreover, since $\psi\in C_0^\infty(-1,1)$, we have
\[|\partial_t^\mu\psi_2(t)|\lesssim_{\psi,m_2}1,\qquad |t|\le4.\]
Thus,
\[\sup_{\substack{|x|\le1\\ |t|\le4}}\left|\partial_x^{m_1}\partial_t^{m_2}\left[\psi_2(t)e^{-ik^{2j+1}t}e^{ia_pk\vartheta(x)}\right]\right|\lesssim_{\psi,\vartheta,m_1,m_2}1.\]
Hence, by the Cauchy Schwarz inequality and \eqref{eq:reduced-low-symbol},
\[\sup_{|t|\le4}\left\|\partial_x^{m_1}\partial_t^{m_2}[\psi_2u_{r,0}^{p,\ell}](t)\right\|_{L^2(-1,1)}^2\lesssim_{\psi,\vartheta,m_1,m_2}\int_0^1k^{2(2j-\ell)}|\widehat\varphi_\ell(-k^{2j+1})|^2\,dk.\]

For $x>1$, we have $\vartheta(x)=x$ and
\[e^{ia_pkx}=e^{i(\re a_p)kx}e^{-(\im a_p)kx}.\]
Since $\im a_p>0$, Lemma \ref{lem:ray-integral} with $\sigma=0$, applied to the corresponding $k$ dependent amplitude supported in $[0,1]$, gives
\[\sup_{|t|\le4}\left\|\partial_x^{m_1}\partial_t^{m_2}[\psi_2u_{r,0}^{p,\ell}](t)\right\|_{L^2(1,\infty)}^2\lesssim_{\psi,m_1,m_2}\int_0^1k^{2(2j-\ell)}|\widehat\varphi_\ell(-k^{2j+1})|^2\,dk.\]

For $x<-1$, we have $\vartheta(x)=-x$. Setting $y=-x$, we have
\[e^{ia_pk\vartheta(x)}=e^{ia_pky},\qquad y>1,\]
and the $x$ derivatives differ from the corresponding $y$ derivatives only by a sign. Hence the same argument as for $x>1$ gives
\[\sup_{|t|\le4}\left\|\partial_x^{m_1}\partial_t^{m_2}[\psi_2u_{r,0}^{p,\ell}](t)\right\|_{L^2(-\infty,-1)}^2\lesssim_{\psi,m_1,m_2}\int_0^1k^{2(2j-\ell)}|\widehat\varphi_\ell(-k^{2j+1})|^2\,dk.\]
Finally, the change of variables $\tau=-k^{2j+1}$ gives
\[\begin{aligned}\int_0^1k^{2(2j-\ell)}|\widehat\varphi_\ell(-k^{2j+1})|^2\,dk=&~{}\frac{1}{2j+1}\int_{-1}^0|\tau|^{\frac{2(j-\ell)}{2j+1}}|\widehat\varphi_\ell(\tau)|^2\,d\tau\\
\lesssim&~{}\int_{-1}^0|\widehat\varphi_\ell(\tau)|^2\,d\tau,\end{aligned}\]
because $1\le\ell\le j$. Combining the estimates on the three spatial regions proves \eqref{eq:low-derivative}.
\end{proof}

The derivative estimate controls both the $X^{s,b,\alpha}$ norm and the spatial Sobolev norm of the low frequency contribution.

\begin{lemma}\label{lem:low-estimates}
For every $s,b,\alpha\in\R$,
\begin{equation}\label{eq:low-X}
\|u_{r,0}^{p,\ell}\|_{X_{\Omega_2}^{s,b,\alpha}}\lesssim_{s,b,\alpha}\|\varphi_\ell\|_{H^{r_\ell}(\R)}.
\end{equation}
If $s\ge-j$, then
\begin{equation}\label{eq:low-sobolev}
\sup_{t\in[0,2]}\|u_{r,0}^{p,\ell}(t)\|_{H^s(\R^+)}\lesssim_s\|\varphi_\ell\|_{H^{r_\ell}(\R)}.
\end{equation}
\end{lemma}

\begin{proof}
Choose integers $M_1,M_2\ge0$ sufficiently large that
\[\bra{\xi}^{2s}\bra{\tau-\xi^{2j+1}}^{2b}+\mathbf{1}_{|\xi|\le 1}(\xi)\bra{\tau}^{2\alpha}\lesssim_{s,b,\alpha}\bra{\xi}^{2M_1}\bra{\tau}^{2M_2}.\]
By Plancherel's theorem,
\[\|\psi_2u_{r,0}^{p,\ell}\|_{X^{s,b,\alpha}(\R^2)}^2\lesssim_{s,b,\alpha}\sum_{\nu=0}^{M_1}\sum_{\mu=0}^{M_2}\|\partial_x^\nu\partial_t^\mu(\psi_2u_{r,0}^{p,\ell})\|_{L^2(\R^2)}^2.\]
Since $\supp\psi_2\subset(-4,4)$, Lemma \ref{lem:low-derivative} gives
\[\begin{aligned}
\|\partial_x^\nu\partial_t^\mu(\psi_2u_{r,0}^{p,\ell})\|_{L^2(\R^2)}^2
=&~{}\int_{-4}^4\|\partial_x^\nu\partial_t^\mu[\psi_2(t)u_{r,0}^{p,\ell}(t)]\|_{L^2(\R)}^2\,dt\\
\lesssim&~{}8\sup_{|t|\le4}\|\partial_x^\nu\partial_t^\mu[\psi_2(t)u_{r,0}^{p,\ell}(t)]\|_{L^2(\R)}^2\\
\lesssim_{\nu,\mu}&~{}\int_{-1}^0|\widehat\varphi_\ell(\tau)|^2\,d\tau.
\end{aligned}\]
Hence
\[\|\psi_2u_{r,0}^{p,\ell}\|_{X^{s,b,\alpha}(\R^2)}^2\lesssim_{s,b,\alpha}\int_{-1}^0|\widehat\varphi_\ell(\tau)|^2\,d\tau\lesssim_s\|\varphi_\ell\|_{H^{r_\ell}(\R)}^2.\]
Since $\psi_2=1$ on $[0,2]$, the definition of the restriction norm proves \eqref{eq:low-X}.

It remains to prove \eqref{eq:low-sobolev}. If $-j\le s<0$, then
\[\|u_{r,0}^{p,\ell}(t)\|_{H^s(\R^+)}\le\|u_{r,0}^{p,\ell}(t)\|_{H^s(\R)}\lesssim_s\|u_{r,0}^{p,\ell}(t)\|_{L^2(\R)},\]
and Lemma \ref{lem:low-derivative} with $m_1=m_2=0$ gives the desired estimate. If $s\ge0$, choose an integer $M\ge s$. Then
\[\|u_{r,0}^{p,\ell}(t)\|_{H^s(\R^+)}\lesssim_s\sum_{\nu=0}^M\|\partial_x^\nu u_{r,0}^{p,\ell}(t)\|_{L^2(\R)},\]
and Lemma \ref{lem:low-derivative} with $m_2=0$ again gives \eqref{eq:low-sobolev}.
\end{proof}

The same estimates hold for the low frequency left contribution, with the temporal frequencies lying in $[0,1]$ instead of $[-1,0]$.

\subsection{Estimate of the residue term}

It remains to control the residue contribution that occurs only when $j$ is odd and $\gamma>0$. Since $\mathcal R_{j,\gamma}[0,\bphi,0]$ is a finite linear combination of terms of the form
\[e^{-\gamma x+\gamma^{2j+1}t}\widetilde\varphi_\ell(2,-i\gamma^{2j+1}),\]
it suffices to estimate the coefficients $\widetilde\varphi_\ell(2,-i\gamma^{2j+1})$ in terms of the boundary data.

\begin{lemma}\label{lem:reduced-residue}
Assume that $j$ is odd and $\gamma>0$. If
\[s>-j-\frac12,\qquad b,\alpha\in\R,\]
then
\begin{equation}\label{eq:residue-X}
\|\mathcal R_{j,\gamma}[0,\bphi,0]\|_{X_{(0,2)\times\R^+}^{s,b,\alpha}}\lesssim_{s,b,\alpha,\gamma}\sum_{\ell=1}^j\|\varphi_\ell\|_{H^{r_\ell}(\R)}.
\end{equation}
If $s\ge-j$, then
\begin{equation}\label{eq:residue-sobolev}
\sup_{t\in[0,2]}\|\mathcal R_{j,\gamma}[0,\bphi,0](t)\|_{H^s(\R^+)}\lesssim_{s,\gamma}\sum_{\ell=1}^j\|\varphi_\ell\|_{H^{r_\ell}(\R)}.
\end{equation}
\end{lemma}

\begin{proof}
We construct a whole line extension of the residue term. Let $\vartheta$ be the function defined in \eqref{eq:vartheta} and define
\[\widetilde{\mathcal R}_{j,\gamma}(t,x):=-i\psi_2(t)e^{-\gamma\vartheta(x)+\gamma^{2j+1}t}\sum_{\ell=1}^jc_{p_\ast,\ell}(i\gamma)^{2j+1-\ell}\widetilde\varphi_\ell(2,-i\gamma^{2j+1}).\]
Since $\vartheta(x)=x$ for $x\ge0$ and $\psi_2(t)=1$ for $0\le t\le2$, we have
\[\widetilde{\mathcal R}_{j,\gamma}(t,x)=\mathcal R_{j,\gamma}[0,\bphi,0](t,x),\qquad (t,x)\in\Omega_2.\]
Set
\[\Phi_\gamma(t,x):=\psi_2(t)e^{-\gamma\vartheta(x)+\gamma^{2j+1}t}.\]
Since $\gamma>0$, every spatial derivative of $e^{-\gamma\vartheta(x)}$ decays exponentially as $|x|\to\infty$, while every temporal derivative of $\psi_2(t)e^{\gamma^{2j+1}t}$ is smooth and compactly supported. Consequently,
\[\|\Phi_\gamma\|_{X^{s,b,\alpha}(\R^2)}\lesssim_{s,b,\alpha,\gamma}1.\]
It follows that
\[\|\widetilde{\mathcal R}_{j,\gamma}\|_{X^{s,b,\alpha}(\R^2)}\lesssim_{s,b,\alpha,\gamma}\sum_{\ell=1}^j|\widetilde\varphi_\ell(2,-i\gamma^{2j+1})|.\]
Thus, it suffices to show that
\[|\widetilde\varphi_\ell(2,-i\gamma^{2j+1})| \lesssim_{\gamma, r_{\ell}} \|\varphi_\ell\|_{H^{r_\ell}(\R)}.\]
Since $s>-j-\frac12$, we have
\[r_\ell>-\frac12,\qquad 1\le\ell\le j.\]
For $1\le\ell\le j$, define
\[g_\gamma(t):=e^{-\gamma^{2j+1}t}\mathbf{1}_{[0,2]}(t).\]
A direct computation gives
\[\widehat g_\gamma(\tau)=\frac{1-e^{-2(\gamma^{2j+1}+i\tau)}}{\gamma^{2j+1}+i\tau},\]
and hence
\[|\widehat g_\gamma(\tau)|\lesssim_\gamma\langle\tau\rangle^{-1}.\]
Therefore $g_\gamma\in H^\sigma(\R)$ for every $\sigma<\frac12$. Since $-r_\ell<\frac12$, it follows that $g_\gamma\in H^{-r_\ell}(\R)$. Recalling that $\varphi_\ell$ denotes its zero extension to $\R$, Sobolev duality gives
\begin{equation}\label{eq:residue-coefficient}
|\widetilde\varphi_\ell(2,-i\gamma^{2j+1})| =\left|\int_\R\varphi_\ell(t)g_\gamma(t)\,dt\right|\lesssim_{\gamma,r_\ell}\|\varphi_\ell\|_{H^{r_\ell}(\R)},
\end{equation}
which proves \eqref{eq:residue-X}.

For the spatial Sobolev estimate, since $e^{-\gamma x}\in H^s(\R^+)$ for every $s\in\R$ and $e^{\gamma^{2j+1}t}$ is uniformly bounded for $t\in[0,2]$, we have
\[\sup_{t\in[0,2]}\|\mathcal R_{j,\gamma}[0,\bphi,0](t)\|_{H^s(\R^+)}\lesssim_{s,\gamma}\sum_{\ell=1}^j|\widetilde\varphi_\ell(2,-i\gamma^{2j+1})|.\]
Together with \eqref{eq:residue-coefficient}, we complete the proof of \eqref{eq:residue-sobolev}.
\end{proof}

\subsection{Proof of Proposition \ref{prop:reduced-ibvp}}
For the smooth boundary data considered above, Lemmas \ref{lem:high-X}, \ref{lem:high-D}, and \ref{lem:low-estimates} give, for every $1\le p,\ell\le j$,
\[\|u_{p,\ell}\|_{X_{\Omega_2}^{s,b,\alpha}}\lesssim_{s,b,\alpha}\|\varphi_\ell\|_{H^{r_\ell}(\R)}.\]
If $s\ge-j$, Lemmas \ref{lem:high-sobolev} and \ref{lem:low-estimates} also give
\[\sup_{t\in[0,2]}\|u_{p,\ell}(t)\|_{H^s(\R^+)}\lesssim_s\|\varphi_\ell\|_{H^{r_\ell}(\R)}.\]
Here the left and right contour contributions satisfy the same estimates. Since each $\varphi_\ell$ is identified with its zero extension to $\R$, Lemma \ref{lem:boundary-extension} yields
\[\|\varphi_\ell\|_{H^{r_\ell}(\R)}\lesssim\|\varphi_\ell\|_{H^{r_\ell}(0,2)}.\]
Therefore, summing over $1\le p,\ell\le j$ and using Lemma \ref{lem:reduced-residue} when $j$ is odd and $\gamma>0$, we obtain
\[\|S_j[0,\bphi,0]\|_{X_{\Omega_2}^{s,b,\alpha}}\lesssim_{s,b,\alpha}\left(\sum_{\ell=1}^j\|\varphi_\ell\|_{H^{r_\ell}(0,2)}^2\right)^{1/2}=\|\bphi\|_{\mathcal H_2^s}.\]
If $s\ge-j$, we also obtain
\[\sup_{t\in[0,2]}\|S_j[0,\bphi,0](t)\|_{H^s(\R^+)}\lesssim_s\left(\sum_{\ell=1}^j\|\varphi_\ell\|_{H^{r_\ell}(0,2)}^2\right)^{1/2}=\|\bphi\|_{\mathcal H_2^s}.\]
This proves \eqref{eq:reduced-X} and \eqref{eq:reduced-sobolev} for smooth boundary data.

We now extend the reduced solution operator to general $\bphi\in\mathcal H_2^s$. Since
\[-\frac12<r_\ell<\frac12,\qquad 1\le\ell\le j,\]
we may choose $\varphi_\ell^{(n)}\in C_0^\infty((0,2))$ such that
\[\varphi_\ell^{(n)}\longrightarrow\varphi_\ell\quad\text{in }H^{r_\ell}(0,2).\]
Set
\[\bphi^{(n)}:=(\varphi_1^{(n)},\ldots,\varphi_j^{(n)}).\]
Applying the estimate for smooth data to $\bphi^{(n)}-\bphi^{(m)}$ gives
\[\|S_j[0,\bphi^{(n)},0]-S_j[0,\bphi^{(m)},0]\|_{X_{\Omega_2}^{s,b,\alpha}}\lesssim_{s,b,\alpha}\|\bphi^{(n)}-\bphi^{(m)}\|_{\mathcal H_2^s}.\]
Hence $\{S_j[0,\bphi^{(n)},0]\}_{n\ge1}$ is Cauchy in $X_{\Omega_2}^{s,b,\alpha}$. We define
\[S_j[0,\bphi,0]:=\lim_{n\to\infty}S_j[0,\bphi^{(n)},0]\quad\text{in }X_{\Omega_2}^{s,b,\alpha}.\]
The difference estimate shows that the limit is independent of the approximating sequence and satisfies \eqref{eq:reduced-X}.

If $s\ge-j$, the smooth Sobolev estimate similarly gives
\[\sup_{t\in[0,2]}\|S_j[0,\bphi^{(n)},0](t)-S_j[0,\bphi^{(m)},0](t)\|_{H^s(\R^+)}\lesssim_s\|\bphi^{(n)}-\bphi^{(m)}\|_{\mathcal H_2^s}.\]
Hence the convergence also holds in $C([0,2];H^s(\R^+))$, and \eqref{eq:reduced-sobolev} follows. This completes the proof.

\section{Forced linear IBVP estimates}\label{sec:forced-linear}

We estimate the solution of \eqref{eq:linear-robin} on $(0,T)\times\R^+$ with $0<T\le1$. The proof separates the whole line contributions generated by the initial datum and forcing from reduced boundary corrections that restore the given Robin data.

\begin{proposition}\label{prop:linear-estimates}
Fix $0<T\le1$. Assume that
\[-j-\frac12<s<\frac32,\qquad 0<b<\frac12,\qquad \frac12<\alpha<1,\qquad \alpha\le1+\frac{s}{2j+1}.\]
For $u_0\in H^s(\R^+)$, $\bphi\in\mathcal H_T^s$, and $f\in\mathcal Z_{\Omega_T}^{s,-b,\alpha-1}$, the smooth UTM solution operator extends uniquely to
\[S_j[u_0,\bphi,f]\in X_{\Omega_T}^{s,b,\alpha}\]
and satisfies
\begin{equation}\label{eq:linear-X}
\|S_j[u_0,\bphi,f]\|_{X_{\Omega_T}^{s,b,\alpha}}\lesssim_{s,b,\alpha}\|u_0\|_{H^s(\R^+)}+\|\bphi\|_{\mathcal H_T^s}+\|f\|_{\mathcal Z_{\Omega_T}^{s,-b,\alpha-1}}.
\end{equation}
If $s\ge-j$, then
\[S_j[u_0,\bphi,f]\in C([0,T];H^s(\R^+))\]
and
\begin{equation}\label{eq:linear-sobolev}
\sup_{t\in[0,T]}\|S_j[u_0,\bphi,f](t)\|_{H^s(\R^+)}\lesssim_{s,b,\alpha}\|u_0\|_{H^s(\R^+)}+\|\bphi\|_{\mathcal H_T^s}+\|f\|_{\mathcal Z_{\Omega_T}^{s,-b,\alpha-1}}.
\end{equation}
The implicit constants may depend on the fixed parameters $j$ and $\gamma$, but they are independent of $T\in(0,1]$.
\end{proposition}

Recall from \eqref{eq:intro-r-ell} that
\[r_\ell=\frac{s+j-\ell}{2j+1},\qquad 0\le\ell\le j.\]
Under the assumptions of Proposition \ref{prop:linear-estimates},
\[-\frac12<r_\ell<\frac12,\qquad 1\le\ell\le j.\]
Recall that $\psi_1=1$ on $[-1,1]$.

For the whole line analysis, we use the linear group $S(t)$ associated with the higher order dispersive equation on the whole line, defined by
\[\widehat{S(t)f}(\xi):=e^{it\xi^{2j+1}}\widehat f(\xi),\qquad t\in\R.\]
Thus $U(t)=S(t)f$ solves
\[\partial_tU+(-1)^{j+1}\partial_x^{2j+1}U=0,\qquad U(0)=f,\]
on $\R^2$.

\begin{proposition}\label{prop:homogeneous-whole-line}
Let $s,b,\alpha\in\R$ and
\[U(t):=S(t)U_0.\]
Then
\begin{equation}\label{eq:homogeneous-whole-line}
\|\psi_1U\|_{X^{s,b,\alpha}(\R^2)}+\sup_{t\in\R}\|U(t)\|_{H^s(\R)}\lesssim\|U_0\|_{H^s(\R)}.
\end{equation}
Moreover,
\begin{equation}\label{eq:homogeneous-traces}
\sup_{x\in\R}\|\psi_1(t)\partial_x^{\ell}U(\cdot,x)\|_{H^{r_\ell}(\R)}\lesssim_s\|U_0\|_{H^s(\R)},\qquad 0\le \ell\le j.
\end{equation}
\end{proposition}

\begin{proof}
The estimate \eqref{eq:homogeneous-whole-line} and \eqref{eq:homogeneous-traces} for $0\le \ell \le j-1$ follow from Theorem 3.1 in \cite{Himonas2022}. We prove the endpoint $\ell=j$.

Set $m:=2j+1$. The change of variables $\tau=\xi^m$ gives
\[\begin{aligned}\|\partial_x^jU(\cdot,x)\|_{H^{s/m}(\R)}^2\lesssim&~{}\int_\R\langle\xi^m\rangle^{2s/m}|\xi|^{2j-(m-1)}|\widehat U_0(\xi)|^2\,d\xi\\
=&~{}\int_\R\langle\xi^m\rangle^{2s/m}|\widehat U_0(\xi)|^2\,d\xi\\
\lesssim_s&~{}\|U_0\|_{H^s(\R)}^2,\end{aligned}\]
because $m-1=2j$ and
\[\langle\xi^m\rangle^{1/m}\sim\langle\xi\rangle.\]
The estimate is uniform in $x$. Multiplication by $\psi_1$ is bounded on $H^{s/m}(\R)$. Since $r_j=s/m$, this proves \eqref{eq:homogeneous-traces} for $\ell=j$.
\end{proof}

\begin{proposition}\label{prop:inhomogeneous-whole-line}
Assume that
\[s\in\R,\qquad 0<b<\frac12,\qquad \frac12<\alpha<1.\]
Let
\[W(t):=\int_0^tS(t-t')w(t')\,dt'.\]
Then
\begin{equation}\label{eq:inhomogeneous-whole-line}
\|\psi_1W\|_{X^{s,b,\alpha}(\R^2)}+\sup_{t\in[0,1]}\|W(t)\|_{H^s(\R)}\lesssim\|w\|_{\mathcal Z^{s,-b,\alpha-1}(\R^2)}.
\end{equation}
Moreover,
\begin{equation}\label{eq:inhomogeneous-traces}
\sup_{x\in\R}\|\psi_1(t)\partial_x^\ell W(\cdot,x)\|_{H^{r_\ell}(\R)}\lesssim\|w\|_{\mathcal Z^{s,-b,\alpha-1}(\R^2)},\qquad 0\le \ell \le j.
\end{equation}
\end{proposition}

\begin{proof}
The $X^{s,b,\alpha}$ estimate in \eqref{eq:inhomogeneous-whole-line} and the trace estimates in \eqref{eq:inhomogeneous-traces} for $0\le \ell \le j-1$ follow from Theorem 3.3 in \cite{Himonas2022}. It remains to establish the endpoint $\ell=j$ and the stated $C_tH_x^s$ estimate.

Set $m:=2j+1$, and let $\chi\in C_0^\infty(\R)$ satisfy
\[\chi(\lambda)=1\quad\text{for }|\lambda|\le1.\]
The Fourier representation of the Duhamel term gives, up to harmless constants,
\[\psi_1(t)\partial_x^jW=\mathcal A-\mathcal B+\mathcal C,\]
where
\[\mathcal A(t,x):=\psi_1(t)\int_{\R^2}e^{i(x\xi+t\tau)}\frac{1-\chi(\tau-\xi^m)}{\tau-\xi^m}(i\xi)^j\mathcal F(w)(\tau,\xi)\,d\xi\,d\tau,\]
\[\mathcal B(t,x):=\psi_1(t)\int_\R e^{i(x\xi+t\xi^m)}(i\xi)^jF_1(\xi)\,d\xi,\]
with
\[F_1(\xi):=\int_\R\frac{1-\chi(\tau-\xi^m)}{\tau-\xi^m}\mathcal F(w)(\tau,\xi)\,d\tau,\]
and
\[\mathcal C(t,x):=\psi_1(t)\int_{\R^2}e^{i(x\xi+t\xi^m)}\chi(\tau-\xi^m)\frac{e^{it(\tau-\xi^m)}-1}{\tau-\xi^m}(i\xi)^j\mathcal F(w)(\tau,\xi)\,d\xi\,d\tau.\]

We first estimate $\mathcal A$. Since multiplication by $\psi_1$ is bounded on $H^{s/m}(\R)$, it suffices to estimate the expression before multiplication by $\psi_1$. By the Cauchy Schwarz inequality in $\xi$,
\[\begin{aligned}
\langle\tau\rangle^{2s/m}\left|\int_\R e^{ix\xi}\frac{1-\chi(\tau-\xi^m)}{\tau-\xi^m}\xi^j\mathcal F(w)(\tau,\xi)\,d\xi\right|^2\lesssim&~{}\langle\tau\rangle^{2s/m}\left(\int_\R\frac{|\xi|^{2j}}{\langle\tau-\xi^m\rangle^{2-2b}}\,d\xi\right)\\
&~{}\times\left(\int_\R\langle\tau-\xi^m\rangle^{-2b}|\mathcal F(w)(\tau,\xi)|^2\,d\xi\right).
\end{aligned}\]
The change of variables $\eta=\xi^m$ gives
\[\int_\R\frac{|\xi|^{2j}}{\langle\tau-\xi^m\rangle^{2-2b}}\,d\xi=\frac{1}{m}\int_\R\frac{d\eta}{\langle\tau-\eta\rangle^{2-2b}}\lesssim_b1,\]
because $b<\frac12$. Hence
\[\sup_{x\in\R}\|\mathcal A(\cdot,x)\|_{H^{s/m}(\R)}\lesssim_b\|w\|_{Y^{s,-b}(\R^2)}.\]

For $\mathcal B$, Proposition \ref{prop:homogeneous-whole-line} with $\ell =j$ gives
\[\sup_{x\in\R}\|\mathcal B(\cdot,x)\|_{H^{s/m}(\R)}\lesssim_s\|F_1\|_{H^s(\R)}.\]
The Cauchy Schwarz inequality in $\tau$ yields
\[\|F_1\|_{H^s(\R)}^2\lesssim\int_\R\langle\xi\rangle^{2s}\left(\int_\R\frac{|\mathcal F(w)(\tau,\xi)|}{\langle\tau-\xi^m\rangle}\,d\tau\right)^2d\xi \lesssim_b\|w\|_{X^{s,-b}(\R^2)}^2.\]

For $\mathcal C$, use
\[\frac{e^{it\lambda}-1}{\lambda}=\sum_{N=1}^\infty\frac{i^Nt^N\lambda^{N-1}}{N!}.\]
Set
\[\widehat c_N(\xi):=\int_\R\chi(\tau-\xi^m)(\tau-\xi^m)^{N-1}\mathcal F(w)(\tau,\xi)\,d\tau.\]
Then $\mathcal C$ is a convergent sum of terms of the form
\[\frac{t^N\psi_1(t)}{N!}\partial_x^jS(t)c_N.\]
For the fixed exponent $s/m$, multiplication by $t^N\psi_1(t)$ is bounded on $H^{s/m}(\R)$ with operator norm at most $C_s^N$. Proposition \ref{prop:homogeneous-whole-line} with $\ell=j$ gives
\[\sup_{x\in\R}\|\mathcal C(\cdot,x)\|_{H^{s/m}(\R)}\lesssim_s\sum_{N=1}^\infty\frac{C_s^N}{N!}\|c_N\|_{H^s(\R)}.\]
Since $\chi$ is compactly supported,
\[\|c_N\|_{H^s(\R)}\lesssim_bC_\chi^N\|w\|_{X^{s,-b}(\R^2)}.\]
The series converges, and therefore
\[\sup_{x\in\R}\|\mathcal C(\cdot,x)\|_{H^{s/m}(\R)}\lesssim_{s,b}\|w\|_{X^{s,-b}(\R^2)}.\]
Combining the estimates for $\mathcal A$, $\mathcal B$, and $\mathcal C$ proves \eqref{eq:inhomogeneous-traces} for $\ell=j$.

It remains to prove the $C_tH_x^s$ estimate. For $0\le t\le1$,
\[\widehat W(t,\xi)=\int_\R\frac{e^{it\tau}-e^{it\xi^m}}{i(\tau-\xi^m)}\mathcal F(w)(\tau,\xi)\,d\tau.\]
On $|\tau-\xi^m|\le1$, the quotient is uniformly bounded. On $|\tau-\xi^m|>1$, it is bounded by $2|\tau-\xi^m|^{-1}$. The Cauchy Schwarz inequality in $\tau$ gives
\[\sup_{t\in[0,1]}\|W(t)\|_{H^s(\R)}\lesssim_b\|w\|_{X^{s,-b}(\R^2)},\]
because
\[\int_{|\lambda|>1}|\lambda|^{-2}\langle\lambda\rangle^{2b}\,d\lambda<\infty\]
for $b<\frac12$. This completes the proof.
\end{proof}

\begin{proof}[Proof of Proposition \ref{prop:linear-estimates}]
We follow the standard superposition argument used for forced linear IBVPs in \cite{Himonas2021,Himonas2022}. Let $U_0$ be a whole line extension of $u_0$ and let $w$ be a whole line extension of $f$. We use the corresponding whole line solutions $U$ and $W$ defined in Propositions \ref{prop:homogeneous-whole-line} and \ref{prop:inhomogeneous-whole-line}, respectively.
For $1\le\ell\le j$, define
\[\varphi_{\ell,1}:=\varphi_\ell-\left(\partial_x^\ell U(\cdot,0)+\gamma\partial_x^{\ell-1}U(\cdot,0)\right),\]
and
\[W_\ell:=-\left(\partial_x^\ell W(\cdot,0)+\gamma\partial_x^{\ell-1}W(\cdot,0)\right).\]
Let
\[\Phi_{\ell,1}:=\left.\mathcal E_T\varphi_{\ell,1}\right|_{(0,2)},\qquad \mathcal W_\ell:=\left.\mathcal E_TW_\ell\right|_{(0,2)}.\]
By superposition, on $\Omega_T$,
\[S_j[u_0,\bphi,f]=U+S_j[0,(\Phi_{1,1},\ldots,\Phi_{j,1}),0]+W+S_j[0,(\mathcal W_1,\ldots,\mathcal W_j),0].\]

Since
\[r_{\ell-1}=r_\ell+\frac{1}{2j+1},\]
we have $H^{r_{\ell-1}}(0,T)\hookrightarrow H^{r_\ell}(0,T)$. Hence Propositions \ref{prop:homogeneous-whole-line} and \ref{prop:inhomogeneous-whole-line} give
\[\sum_{\ell=1}^j\|\varphi_{\ell,1}\|_{H^{r_\ell}(0,T)}\lesssim\|\bphi\|_{\mathcal H_T^s}+\|U_0\|_{H^s(\R)},\]
and
\[\sum_{\ell=1}^j\|W_\ell\|_{H^{r_\ell}(0,T)}\lesssim\|w\|_{\mathcal Z^{s,-b,\alpha-1}(\R^2)}.\]
Therefore, Proposition \ref{prop:reduced-ibvp}, Lemma \ref{lem:boundary-extension}, and the whole line estimates yield
\[\|S_j[u_0,\bphi,f]\|_{X_{\Omega_T}^{s,b,\alpha}}\lesssim\|U_0\|_{H^s(\R)}+\|\bphi\|_{\mathcal H_T^s}+\|w\|_{\mathcal Z^{s,-b,\alpha-1}(\R^2)}.\]
If $s\ge-j$, the same argument gives
\[\sup_{t\in[0,T]}\|S_j[u_0,\bphi,f](t)\|_{H^s(\R^+)}\lesssim\|U_0\|_{H^s(\R)}+\|\bphi\|_{\mathcal H_T^s}+\|w\|_{\mathcal Z^{s,-b,\alpha-1}(\R^2)}.\]
Taking the infimum over all admissible whole line extensions $U_0$ and $w$ proves \eqref{eq:linear-X} and \eqref{eq:linear-sobolev} for smooth data. The general case follows by density and the corresponding difference estimates. The constants are uniform for $T\in(0,1]$ by Lemma \ref{lem:boundary-extension}.
\end{proof}

\section{Proof of Theorem \ref{mainresult}}\label{sec:well-posedness}
Fix
\[-j+\frac14<s<\frac32.\]
Choose $b_0$, $b$, $\widetilde b$, $b_1$, $\alpha$, and $\widetilde\alpha$ as in Proposition \ref{prop:bilinear}, and set
\begin{equation}\label{eq:theta}
\theta:=\min\{b_1-\widetilde b,\widetilde\alpha-\alpha,b-b_0\}>0.
\end{equation}
The following estimate extracts the positive power of the lifespan required in the contraction argument.
\begin{lemma}\label{lem:restriction-nonlinear}
For every $0<T\le1$,
\begin{equation}\label{eq:restriction-product}
\|\partial_x(uv)\|_{\mathcal Z_{\Omega_T}^{s,-b_1,\alpha-1}} \lesssim T^\theta \|u\|_{X_{\Omega_T}^{s,b,\alpha}}\|v\|_{X_{\Omega_T}^{s,b,\alpha}}.
\end{equation}
Moreover,
\begin{equation}\label{eq:restriction-difference}
\|\partial_x(u^2-v^2)\|_{\mathcal Z_{\Omega_T}^{s,-b_1,\alpha-1}} \lesssim T^\theta \left(\|u\|_{X_{\Omega_T}^{s,b,\alpha}}+\|v\|_{X_{\Omega_T}^{s,b,\alpha}}\right) \|u-v\|_{X_{\Omega_T}^{s,b,\alpha}}.
\end{equation}
\end{lemma}

\begin{proof}
Choose extensions $\widetilde u,\widetilde v\in X^{s,b,\alpha}(\R^2)$ satisfying
\[\|\widetilde u\|_{X^{s,b,\alpha}(\R^2)} \le2\|u\|_{X_{\Omega_T}^{s,b,\alpha}}, \qquad \|\widetilde v\|_{X^{s,b,\alpha}(\R^2)}\le2\|v\|_{X_{\Omega_T}^{s,b,\alpha}}.\]
Since $\psi_T=1$ on $[0,T]$, the function
\[\psi_T(t)\partial_x(\widetilde u\widetilde v)\]
is an extension of $\partial_x(uv)$ from $\Omega_T$ to $\R^2$.

We first estimate the $X^{s,-b_1,\alpha-1}$ component. By Lemma \ref{lem:time-cutoff} and Proposition \ref{prop:bilinear},
\[\begin{aligned}
\|\psi_T\partial_x(\widetilde u\widetilde v)\|_{X^{s,-b_1,\alpha-1}(\R^2)}\lesssim&~{}T^{\min\{b_1-\widetilde b,\widetilde\alpha-\alpha\}}\|\partial_x(\widetilde u\widetilde v)\|_{X^{s,-\widetilde b,\widetilde\alpha-1}(\R^2)}\\
\lesssim&~{}T^\theta \|\widetilde u\|_{X^{s,b,\alpha}(\R^2)}\|\widetilde v\|_{X^{s,b,\alpha}(\R^2)}.
\end{aligned}\]

For the $Y^{s,-b_1}$ component, since $\psi_T$ depends only on time,
\[\psi_T\partial_x(\widetilde u\widetilde v)=\partial_x((\psi_T\widetilde u)\widetilde v).\]
Using \eqref{eq:bilinear-Y-used},
\[\|\psi_T\partial_x(\widetilde u\widetilde v)\|_{Y^{s,-b_1}(\R^2)}\lesssim\|\psi_T\partial_x(\widetilde u\widetilde v)\|_{X^{s,-b_1}(\R^2)} +\|\psi_T\widetilde u\|_{X^{s,b_0}(\R^2)}\|\widetilde v\|_{X^{s,b_0}(\R^2)}.\]
The first term is controlled by the preceding estimate, while Lemma \ref{lem:time-cutoff} gives
\[\|\psi_T\widetilde u\|_{X^{s,b_0}(\R^2)}\lesssim T^{b-b_0}\|\widetilde u\|_{X^{s,b}(\R^2)}.\]
Also,
\[\|\widetilde v\|_{X^{s,b_0}(\R^2)}\lesssim\|\widetilde v\|_{X^{s,b,\alpha}(\R^2)}.\]
Since $\theta\le b-b_0$, we obtain
\[\|\psi_T\partial_x(\widetilde u\widetilde v)\|_{\mathcal Z^{s,-b_1,\alpha-1}(\R^2)}\lesssim T^\theta \|\widetilde u\|_{X^{s,b,\alpha}(\R^2)}\|\widetilde v\|_{X^{s,b,\alpha}(\R^2)}.\]
Taking the infimum over the extensions proves \eqref{eq:restriction-product}.

Finally, applying the same argument to
\[u^2-v^2=(u-v)(u+v)\]
gives \eqref{eq:restriction-difference}.
\end{proof}

\begin{proof}[Proof of Theorem \ref{mainresult}]
Since $\bphi\in\mathcal H_{\mathrm{loc}}^s$, we have $\bphi|_{(0,1)}\in\mathcal H_1^s$. Set
\[D:=\|u_0\|_{H^s(\R^+)}+\|\bphi\|_{\mathcal H_1^s}.\]
Choose
\[T:=\min\{1,c_0(1+D)^{-1/\theta}\},\]
where $c_0>0$ will be chosen sufficiently small, and use the restriction of $\bphi$ to $(0,T)$ without changing notation. Define
\[\Gamma(u):=S_j\left[u_0,\bphi,-\frac12\partial_x(u^2)\right].\]
Since $s>-j$, Proposition \ref{prop:linear-estimates} with $b_1$ in place of $b$, the embedding
\[X_{\Omega_T}^{s,b_1,\alpha}\hookrightarrow X_{\Omega_T}^{s,b,\alpha},\]
and Lemma \ref{lem:restriction-nonlinear} give
\begin{equation}\label{eq:fixed-map}
\|\Gamma(u)\|_{X_{\Omega_T}^{s,b,\alpha}}+\sup_{t\in[0,T]}\|\Gamma(u)(t)\|_{H^s(\R^+)} \le C\left(D+T^\theta\|u\|_{X_{\Omega_T}^{s,b,\alpha}}^2\right).
\end{equation}
Similarly, if
\[\widetilde\Gamma(v):=S_j\left[\widetilde u_0,\widetilde{\bphi},-\frac12\partial_x(v^2)\right],\]
then
\begin{equation}\label{eq:fixed-difference}
\begin{aligned}
&~{}\|\Gamma(u)-\widetilde\Gamma(v)\|_{X_{\Omega_T}^{s,b,\alpha}}+\sup_{t\in[0,T]}\|\Gamma(u)(t)-\widetilde\Gamma(v)(t)\|_{H^s(\R^+)}\\
\lesssim&~{}\|u_0-\widetilde u_0\|_{H^s(\R^+)}+\|\bphi-\widetilde{\bphi}\|_{\mathcal H_T^s}+T^\theta\left(\|u\|_{X_{\Omega_T}^{s,b,\alpha}}+\|v\|_{X_{\Omega_T}^{s,b,\alpha}}\right)
\|u-v\|_{X_{\Omega_T}^{s,b,\alpha}}.
\end{aligned}
\end{equation}
Let
\[M:=2CD.\]
By choosing $c_0$ sufficiently small, \eqref{eq:fixed-map} and \eqref{eq:fixed-difference} with identical data show that $\Gamma$ maps
\[B_M:=\left\{u\in X_{\Omega_T}^{s,b,\alpha}:\ \|u\|_{X_{\Omega_T}^{s,b,\alpha}}\le M\right\}\]
into itself and is a strict contraction. Hence Banach's fixed point theorem gives a unique fixed point
\[u\in X_{\Omega_T}^{s,b,\alpha}\cap C([0,T];H^s(\R^+)),\]
and \eqref{eq:fixed-map} yields
\[\|u\|_{X_{\Omega_T}^{s,b,\alpha}}+\sup_{t\in[0,T]}\|u(t)\|_{H^s(\R^+)}\lesssim D.\]

For two data sets in a sufficiently small neighborhood of $(u_0,\bphi)$ in $H^s(\R^+)\times\mathcal H_{\mathrm{loc}}^s$, chosen so that their restrictions to $(0,1)$ remain uniformly bounded in $H^s(\R^+)\times\mathcal H_1^s$, the same $T$ may be chosen. Applying \eqref{eq:fixed-difference} to the corresponding fixed points and absorbing the last term gives local Lipschitz dependence on the initial and boundary data.

The contraction argument gives uniqueness in $B_M$. To remove this restriction, let $u_1,u_2\in X_{\Omega_T}^{s,b,\alpha}$ be two UTM solutions with the same data. By applying the local difference estimate on sufficiently short time intervals, one obtains uniqueness successively on $\Omega_T$, following the standard continuation argument used in \cite[Section 4]{Himonas2022-1}.
\end{proof}

\appendix
\section{Special cases}\label{app:special-cases}

In this appendix, we record several reductions of the UTM representation \eqref{eq:utm-full-formula} as mentioned in Remark \ref{rem:intro-related-results}. These reductions recover the classical KdV Robin and Neumann formulas \cite{Himonas2021,Himonas2022-1} and clarify the relation between the present Robin hierarchy and the higher order Dirichlet problem studied in \cite{Himonas2022}.

\subsection{The classical KdV equation with Robin boundary conditions}

Let $j=1$. Then
\[\omega=e^{2\pi i/3},\qquad D^+=D_2^+,\qquad \alpha_{1,1}=\omega,\qquad \alpha_{1,2}=\omega^2.\]
Moreover,
\[Q_1(z)=(z-\omega)(z-\omega^2)=z^2+z+1,\qquad Q_1(1)=3,\]
and
\[Q_1'(\omega)=i\sqrt3,\qquad Q_1'(\omega^2)=-i\sqrt3.\]
A direct computation from \eqref{eq:P-pn} gives
\[\mathcal P_{1,1}(k)=-\omega+i\sqrt3\frac{k}{k-i\gamma},\qquad \mathcal P_{1,2}(k)=-\omega^2-i\sqrt3\frac{k}{k-i\gamma}.\]
Using $1+\omega+\omega^2=0$, these can be rewritten as
\begin{equation}\label{eq:app-j1-P}
-\mathcal P_{1,1}(k)=\frac{k+i(\omega+1)\gamma}{\omega(k-i\gamma)},\qquad -\mathcal P_{1,2}(k)=-\frac{(\omega+1)k+i\gamma}{\omega(k-i\gamma)}.
\end{equation}
For the boundary coefficient, the Euclidean division in \eqref{eq:S-p-ell} gives $S_{1,1}(z)=1$, and hence
\[c_{1,1}=3i.\]
Therefore, the boundary contribution in \eqref{eq:utm-full-formula} becomes
\[\frac{3i}{2\pi}\int_{\partial D^+}e^{ikx+ik^3t}\frac{k^2}{k-i\gamma}\widetilde\varphi_1(T,k^3)\,dk.\]

If $\gamma>0$, then $i\gamma\in D^+$ and the residue contribution reduces to
\begin{equation}\label{eq:app-j1-residue}
\begin{aligned}
\mathcal R_{1,\gamma}[u_0,\varphi_1,f](t,x)=&~{}-\frac{(2+\omega)\gamma}{\omega}\Big[\widehat u_0(i\gamma\omega^2)+F(T,i\gamma\omega^2)-\widehat u_0(i\gamma\omega)-F(T,i\gamma\omega)\Big]e^{-\gamma x+\gamma^3t}\\
&~{}-3\gamma^2\widetilde\varphi_1(T,-i\gamma^3)e^{-\gamma x+\gamma^3t}.
\end{aligned}
\end{equation}
For $\gamma\le0$, no residue term is present. Substituting \eqref{eq:app-j1-P} and \eqref{eq:app-j1-residue} into \eqref{eq:utm-full-formula}, we recover \cite[formula (1.16)]{Himonas2022-1}, with their $\sigma$ identified with $\omega$ and their Fourier variable $\xi$ with $k$.

\subsection{The derivative hierarchy}

Let $\gamma=0$. Then the boundary conditions reduce to
\[\partial_x^\ell u(t,0)=\varphi_\ell(t),\qquad 1\le\ell\le j.\]
Since
\[\frac{Q_p(1)}{Q_p'(\alpha_{p,n})}=(1-\alpha_{p,n})L_{p,n}(1),\]
formula \eqref{eq:P-pn} gives
\begin{equation}\label{eq:app-gamma-zero-P}
\mathcal P_{p,n}(k)=\alpha_{p,n}L_{p,n}(1),
\end{equation}
independently of $k$. In particular,
\[\sum_{n=1}^{j+1}\mathcal P_{p,n}(k)=1,\]
because the polynomial $\sum_{n=1}^{j+1}\alpha_{p,n}L_{p,n}(z)$ interpolates the function $z$ at the $j+1$ nodes $\alpha_{p,1},\ldots,\alpha_{p,j+1}$.

Moreover,
\[\frac{k^{2j+1-\ell}}{k-i\gamma}=k^{2j-\ell},\qquad \mathcal R_{j,0}=0.\]
Thus \eqref{eq:utm-full-formula} becomes
\begin{equation}\label{eq:app-gamma-zero-formula}
\begin{aligned}
u(t,x)=&~{}\frac{1}{2\pi}\int_{\R}e^{ikx+ik^{2j+1}t}\big(\widehat u_0(k)+F(t,k)\big)\,dk\\
&~{}-\frac{1}{2\pi}\sum_{p=1}^j\sum_{n=1}^{j+1}\alpha_{p,n}L_{p,n}(1)\int_{\partial D_{2p}^+}e^{ikx+ik^{2j+1}t}\big(\widehat u_0(\alpha_{p,n}k)+F(T,\alpha_{p,n}k)\big)\,dk\\
&~{}+\frac{1}{2\pi}\sum_{p=1}^j\sum_{\ell=1}^jc_{p,\ell}\int_{\partial D_{2p}^+}e^{ikx+ik^{2j+1}t}k^{2j-\ell}\widetilde\varphi_\ell(T,k^{2j+1})\,dk.
\end{aligned}
\end{equation}
Hence, when $\gamma=0$, the UTM representation does not distinguish between odd and even $j$: the parity distinction in the Robin formula is entirely due to the possible pole $k=i\gamma$.

For $j=1$, \eqref{eq:app-gamma-zero-formula} further reduces to
\[
\begin{aligned}
u(t,x)=&~{}\frac{1}{2\pi}\int_{\R}e^{ikx+ik^3t}\big(\widehat u_0(k)+F(t,k)\big)\,dk\\
&~{}+\frac{1}{2\pi}\int_{\partial D^+}e^{ikx+ik^3t}\Big\{\omega^2\big[\widehat u_0(\omega k)+F(T,\omega k)\big]+\omega\big[\widehat u_0(\omega^2k)+F(T,\omega^2k)\big]\Big\}\,dk\\
&~{}+\frac{3i}{2\pi}\int_{\partial D^+}e^{ikx+ik^3t}k\widetilde\varphi_1(T,k^3)\,dk.
\end{aligned}
\]
This is precisely the Neumann formula obtained from \cite[formula (1.16)]{Himonas2022-1} by setting $\gamma=0$.

\subsection{The formal Dirichlet limit}

To describe the relation with the higher order Dirichlet problem, introduce the reciprocal parameter
\[\varepsilon:=\gamma^{-1}\]
and set
\[h_\ell:=\gamma^{-1}\varphi_\ell,\qquad 1\le\ell\le j.\]
Then the Robin conditions can be written as
\begin{equation}\label{eq:app-dual-bc}
(\varepsilon\partial_x+1)\partial_x^{\ell-1}u(t,0)=h_\ell(t),\qquad 1\le\ell\le j.
\end{equation}
Formally setting $\varepsilon=0$ in \eqref{eq:app-dual-bc} gives
\[\partial_x^{\ell-1}u(t,0)=h_\ell(t),\qquad 1\le\ell\le j,\]
which is exactly the higher order Dirichlet hierarchy considered in \cite{Himonas2022}.

For $\varepsilon\neq0$, \eqref{eq:app-dual-bc} is the Robin hierarchy with Robin parameter $\varepsilon^{-1}$ and boundary data $\varepsilon^{-1}h_\ell$. At the level of the UTM coefficients, for each fixed $k$,
\[\frac{k}{k-i/\varepsilon}=\frac{\varepsilon k}{\varepsilon k-i}\longrightarrow0\]
and
\[\frac{1}{\varepsilon}\frac{k^{2j+1-\ell}}{k-i/\varepsilon}=\frac{k^{2j+1-\ell}}{\varepsilon k-i}\longrightarrow i k^{2j+1-\ell}\]
as $\varepsilon\to0$. Consequently,
\[\mathcal P_{p,n}(k)\longrightarrow L_{p,n}(1).\]
Thus, formally substituting these coefficient limits into the nonresidue part of \eqref{eq:utm-full-formula}, and using the same contour deformation as in the derivation of the UTM representation to replace $F(T,\alpha_{p,n}k)$ by $F(t,\alpha_{p,n}k)$, yields
\begin{equation}\label{eq:app-dirichlet-formula}
\begin{aligned}
u(t,x)=&~{}\frac{1}{2\pi}\int_{\R}e^{ikx+ik^{2j+1}t}\big(\widehat u_0(k)+F(t,k)\big)\,dk\\
&~{}-\frac{1}{2\pi}\sum_{p=1}^j\sum_{n=1}^{j+1}L_{p,n}(1)\int_{\partial D_{2p}^+}e^{ikx+ik^{2j+1}t}\big(\widehat u_0(\alpha_{p,n}k)+F(t,\alpha_{p,n}k)\big)\,dk\\
&~{}+\frac{i}{2\pi}\sum_{p=1}^j\sum_{\ell=1}^jc_{p,\ell}\int_{\partial D_{2p}^+}e^{ikx+ik^{2j+1}t}k^{2j+1-\ell}\widetilde h_\ell(T,k^{2j+1})\,dk.
\end{aligned}
\end{equation}

Formula \eqref{eq:app-dirichlet-formula} has the same form as the higher order Dirichlet UTM representation \cite[formula (1.11)]{Himonas2022}. Indeed, after the index change $\ell'=\ell-1$,
\[(ik)^{2j-\ell'}=i^{2j+1-\ell}k^{2j+1-\ell},\]
and the corresponding coefficients are identified by
\[C_{p,n}=-\frac{L_{p,n}(1)}{2\pi},\qquad C'_{p,\ell-1}=\frac{i^{\ell-2j}}{2\pi}c_{p,\ell}=\frac{(-1)^{j+1}}{2\pi}Q_p(1)S_{p,\ell}(1).\]

The calculation above is only a formal consistency check at the level of the boundary conditions and the UTM coefficients. In particular, we do not pass to the limit in the contour integrals or the residue term, and no convergence of the corresponding Robin solutions to the Dirichlet solution as $|\gamma|\to\infty$ is asserted.

\end{document}